\documentclass[11pt]{amsart}

\usepackage{amssymb,amsfonts,textcomp}
\usepackage{hyperref}
\usepackage{enumerate}
\usepackage{comment}

\usepackage[utf8]{inputenc}
\usepackage[T1]{fontenc}

\usepackage[mathscr]{eucal}

\usepackage[a4paper, twoside=false, vmargin={2cm,3cm}, includehead]{geometry}

 \hypersetup{
     colorlinks=true,
     linktocpage=true,
     linkcolor=red,
     filecolor=blue,
     citecolor = blue,
     urlcolor=cyan,
     }

\newtheorem{lemma}{Lemma}[section]
\newtheorem{theorem}[lemma]{Theorem}
\newtheorem*{theorem*}{Theorem}
\newtheorem{corollary}[lemma]{Corollary}

\newtheorem{proposition}[lemma]{Proposition}
\newtheorem*{proposition*}{Proposition}
\newtheorem{conjecture}[lemma]{Conjecture}

\newtheorem*{problem*}{Problem}

\theoremstyle{definition}
\newtheorem*{claim*}{Claim}

\newtheorem*{definition}{Definition}

\newtheorem{example}{Example}

\newtheorem*{remark}{Remark}
\newtheorem*{remarks}{Remarks}

\newcommand{\C}{{\mathbb C}}
\newcommand{\E}{{\mathbb E}}

\newcommand{\N}{{\mathbb N}}
\renewcommand{\P}{{\mathbb P}}
\newcommand{\Q}{{\mathbb Q}}
\newcommand{\R}{{\mathbb R}}

\newcommand{\T}{{\mathbb T}}
\newcommand{\Z}{{\mathbb Z}}

\newcommand{\CA}{{\mathcal A}}

\newcommand{\CC}{{\mathcal C}}
\newcommand{\CD}{{\mathcal D}}
\newcommand{\CE}{{\mathcal E}}
\newcommand{\CF}{{\mathcal F}}
\newcommand{\CI}{{\mathcal I}}
\newcommand{\CK}{{\mathcal K}}

\newcommand{\CX}{{\mathcal X}}
\newcommand{\CY}{{\mathcal Y}}

\newcommand{\CZ}{{\mathcal Z}}

\newcommand{\FD}{{\mathfrak D}}

\newcommand{\FL}{{\mathfrak L}}

\newcommand{\ba}{{\mathbf{a}}}

\newcommand{\bc}{{\mathbf{c}}}
\newcommand{\be}{{\mathbf{e}}}

 \renewcommand{\a}{{\textbf{a}}}
 \renewcommand{\b}{{\textbf{b}}}
 \newcommand{\uh}{{\underline{h}}}
 \newcommand{\ui}{{\underline{i}}}

 \newcommand{\um}{{\underline{m}}}
 \newcommand{\p}{\textbf{p}}

\newcommand{\eps}{\epsilon}

\newcommand{\ueps}{{\underline{\epsilon}}}

\newcommand{\norm}[1]{\left\Vert #1\right\Vert}
\newcommand{\nnorm}[1]{\lvert\!|\!| #1|\!|\!\rvert}
\newcommand{\Bignnorm}[1]{\Big|\!\Big|\!\Big| #1\Big|\!\Big|\!\Big|}

\DeclareMathOperator{\spec}{Spec}
\DeclareMathOperator{\supp}{Supp}

\newcommand{\Krat}{{\CK_{\text{\rm rat}}}}

\newcommand{\abs}[1]{\mathopen{}\left| #1\mathclose{}\right|}

\newcommand{\brac}[1]{\mathopen{}\left( #1 \mathclose{}\right)}

\begin{document}
	
	\title[Joint ergodicity for commuting transformations and applications]{Joint ergodicity for commuting transformations and applications to polynomial sequences}
	\thanks{The authors were supported  by the Hellenic Foundation for Research and Innovation, Project No: 1684
and		ELIDEK HFRI-NextGenerationEU-15689. For a part of the project, the second author was also supported by the Academy of Finland, Project Nos. 309365, 314172, 321896.}

\author{Nikos Frantzikinakis and Borys Kuca}
\address[Nikos Frantzikinakis]{University of Crete, Department of mathematics and applied mathematics, Voutes University Campus, Heraklion 71003, Greece} \email{frantzikinakis@gmail.com}

\address[Borys Kuca]{Jagiellonian University, Faculty of Mathematics and Computer Science, 30-348 Krak\'ow, Poland}
\email{borys.kuca@uj.edu.pl}
\begin{abstract}
	We give necessary and sufficient conditions for  joint ergodicity results of collections of sequences with respect to systems of commuting measure preserving transformations.	Combining these results with a new technique that we call  ``seminorm smoothing'', we settle several conjectures related to multiple ergodic averages of commuting transformations with polynomial iterates. We show that the Host-Kra factor is characteristic for pairwise independent polynomials, and that under certain ergodicity conditions the associated ergodic averages converge  to the product of integrals. Moreover,  when the polynomials are linearly independent, we show that   the rational Kronecker factor is characteristic and  deduce Khintchine-type lower bounds for the related multiple recurrence problem. Finally, we prove a nil plus null decomposition  result for multiple correlation sequences of commuting transformations in the case  where the iterates  are given by families of  pairwise independent polynomials.
\end{abstract}


\maketitle

		\tableofcontents
	\section{ Introduction }
	\subsection{History and our main goals}	
	The works of Furstenberg-Weiss~\cite{FuW96}, Bergelson-Leibman~\cite{BL96}, and others, have initiated a fruitful study of multiple ergodic averages with polynomial iterates.  The setup is as follows: we are given a system $(X, \CX, \mu,T_1,\ldots, T_\ell)$, meaning invertible	commuting measure preserving transformations $T_1,\ldots, T_\ell$ acting on a Lebesgue probability space $(X,\CX, \mu)$,   functions $f_1,\ldots,f_\ell\in L^\infty(\mu)$, and polynomials $p_1,\ldots, p_\ell\in \Z[n]$,  and we are interested in the limiting behavior in $L^2(\mu)$ of  the multiple ergodic averages
	\begin{align}\label{E:polies}
	\frac{1}{N}\sum_{n=1}^N \, T_1^{p_1(n)}f_1\cdots T_\ell^{p_\ell(n)}f_\ell.
	\end{align}	
	In the case of a single transformation $T_1 =\cdots = T_\ell=T$, a lot is known about the limiting behavior of the averages \eqref{E:polies} thanks to an impressive body of works that utilizes the theory of Host-Kra factors~\cite{HK05a} and  examines various aspects of this question. By contrast, much less has been established when the transformations are allowed to differ; this setting involves significant technical difficulties absent from the single transformation case. Although mean convergence of the above averages was established for linear iterates by Tao~\cite{Ta08} and for polynomial iterates  by Walsh~\cite{Wal12}, the methods they used  do not give any information about the limit function and leave many natural problems unanswered. Subsequent methods developed by Austin~\cite{Au09} and Host~\cite{Ho09}  have been useful for problems involving linear iterates, but they do not provide strong enough information to deal with polynomial iterates, at least not for the problems in which we are interested.
	
	As a typical problem,  consider the case $\ell=2$ and $p_1(n)=n^2$, $p_2(n)=n^2+n$,  for which we seek to analyse the limiting behavior of the averages
	\begin{equation}\label{E:n2n2n}
	\frac{1}{N}\sum_{n=1}^N \, T_1^{n^2}f_1\cdot T_2^{n^2+n}f_2
\end{equation}
	in $L^2(\mu)$.
Three natural questions arise: is the  limiting behavior of \eqref{E:n2n2n} controlled by one of the ergodic seminorms constructed by Host and Kra, which have proved so useful in examining averages involving a single transformation? Do the averages \eqref{E:n2n2n} converge to  the product of the integrals if the transformations $T_1,T_2$ are totally ergodic? Is it the case that  for all measure preserving transformations $T_1, T_2$, measurable set $A$ and $\varepsilon>0$,  we can obtain the bound
	$$
	\mu(A\cap T_1^{n^2}A \cap T_2^{n^2+n}A)\geq (\mu(A))^3-\varepsilon
	$$
	for some $n\in\N$, or even better, for a set of $n\in\N$ with bounded gaps?
	When $T_1=T_2$, these questions were answered affirmatively in  \cite{HK05b, Lei05c}, \cite{FrK05a}, \cite{FrK06} respectively,
	using the  theory of Host-Kra factors~\cite{HK05a} and equidistribution results on nilmanifolds from~\cite{Lei05a}.
	Unfortunately,  this toolbox could not be put  into good use  to cover the case of general commuting transformations $T_1,T_2$.    The reason is that  for averages such as \eqref{E:n2n2n} convenient characteristic factors
	were not known, and  even in  cases where useful characteristic factors were known (for example when the polynomial iterates  have distinct degrees \cite[Theorem~1.2]{CFH11}), the
	equidistribution results on nilmanifolds  needed  to solve  these problems  seemed very difficult.


In this article, we  take a different approach that will allow us to address the above questions  and  a wide variety of related problems and conjectures. In particular,  the   applications of the techniques developed will include the following results, the precise statements of which are given in Section~\ref{S:Results}:\footnote{For simplicity, all polynomials are assumed to have zero constant terms.}

	\begin{enumerate}
	\item\label{C:1} (Characteristic factors for pairwise independent polynomials) In Theorem~\ref{T:polies0}, we show that if  the polynomials $p_1,\ldots,p_\ell$ are
	pairwise independent,  then there exists  $s\in\N$ such that for every system $(X, \CX, \mu,T_1,\ldots, T_\ell)$,
	the Host-Kra factors $\CZ_s(T_j)$ are characteristic for the mean convergence of the averages \eqref{E:polies}. This  resolves a  conjecture from \cite[Section~1.3]{CFH11} (see also \cite[Problem~15]{Fr16}).
		
	\item\label{C:2} (Limit formula for linearly independent polynomials) For linearly independent polynomials $p_1,\ldots,p_\ell$, we show in  Theorem~\ref{T:polies2}, by combining Theorem~\ref{T:main2} with  Theorem~\ref{T:polies0}, that  for every system $(X, \CX, \mu,T_1,\ldots, T_\ell)$,   the  rational Kronecker factors $\Krat(T_j)$ are characteristic for the mean convergence of the averages \eqref{E:polies}. This was conjectured in  \cite[Section~1.3]{CFH11} (see also \cite[Problem~16]{Fr16}).

	\item\label{C:2'} (Khintchine-type lower bounds for linearly  independent polynomials) If the polynomials $p_1,\ldots,p_\ell$ are
	linearly  independent, then we show in  Corollary~\ref{C:polies2'}   that we have
	Khintchine-type lower bounds for the corresponding multiple intersections and we deduce in Corollary~\ref{C:combinatorics} a corresponding result in combinatorics. This
	answers a conjecture from  \cite[Section~1.3]{CFH11} (see also \cite[Problem~17]{Fr16}).
	
	\item\label{C:3} (Joint ergodicity for pairwise independent polynomials) For a given system $(X, \CX, \mu,T_1,\ldots, T_\ell)$, we show in Theorem~\ref{C:polies1'},  by combining Theorem~\ref{T:main1} with  Theorem~\ref{T:polies0},  that for pairwise independent polynomials,\footnote{The polynomials $p_1,\ldots, p_\ell \in\Z[n]$, with zero constant terms, {\em are pairwise independent,} if
		$p_i/p_j$ is non-constant for distinct $i,j\in \{1,\ldots, \ell\}$.} under certain necessary ergodicity assumptions,
	 the averages \eqref{E:polies} converge in $L^2(\mu)$ to the product of the integrals of the individual functions.  This was conjectured in \cite[Conjecture~1.5]{DKS19} as  part of a more general statement. In an upcoming work, we resolve this conjecture for families of not necessarily pairwise independent polynomials.

	\item\label{C:4} (Nil plus null decomposition for pairwise independent polynomials) Following previous work in \cite{BHK05, Fr15,Lei11, L15},  it became plausible  (and was explicitly asked, for example,  in the comment following \cite[Problem~18]{Fr16}) that  for all   polynomials $p_1,\ldots, p_\ell$, systems $(X, \CX, \mu,T_1,\ldots, T_\ell)$, and functions $f_0,\ldots, f_\ell\in L^\infty(\mu)$,
	the multicorrelation sequence
	$$
	C(n):=\int f_0\cdot T_1^{p_1(n)}f_1\cdots T_\ell^{p_\ell(n)}f_\ell\, d\mu
	$$
	can be decomposed into a sum of a nilsequence and a null-sequence. Using Theorem~\ref{T:polies0},  we verify in  Theorem~\ref{T:decomposition} this conjecture for all families of  pairwise independent polynomials.
\end{enumerate}
{We also get similar results as in \eqref{C:1}, \eqref{C:2}, \eqref{C:2'}  for averages over the primes, see Section~\ref{SS:primes}.}

Our approach does not use the theory of Host-Kra factors and equidistribution results on nilmanifolds, but instead relies on   two results of independent interest that we describe next.
	The first is a result about  joint ergodicity of  collections of general sequences  $a_1,\ldots, a_\ell\colon \N\to \Z$. This result is given in Theorem~\ref{T:main2} (see also  Theorem~\ref{T:main2'}), and it enables us to determine the limits of the averages
		\begin{equation}\label{E:general}
	\frac{1}{N}\sum_{n=1}^N \, T_1^{a_1(n)}f_1\cdots T_\ell^{a_\ell(n)}f_\ell
	\end{equation}	
when the sequences $a_1,\ldots, a_\ell$ satisfy two necessary conditions. The first condition is an equidistribution property that is typically easy to verify. The second condition involves some seminorm control that is not easy to come by when we deal with commuting transformations.  Our second main result is  Theorem~\ref{T:polies0}, and it addresses precisely this issue in several cases of interest. We show that for families of pairwise independent polynomials, the needed seminorm control is  satisfied.

The proof of Theorem~\ref{T:main2} uses a ``degree lowering'' argument, which was inspired by  the one used in \cite{Fr21} to deal with the case where $T_1=\cdots=T_\ell$. In our more general setting, however, there are quite a few additional difficulties, caused mainly by the fact that we cannot reduce to the case where all the transformations are ergodic, which  was an essential maneuver for the argument in \cite{Fr21} to work.   The proof of Theorem~\ref{T:polies0}
uses a new argument, which we call ``seminorm smoothing''.  We start with some  bounds that were obtained in \cite{DFMKS21} and involve rather complicated box seminorms, and  manage to successively smooth them out, in the sense that
at each step we produce bounds that involve less complicated seminorms. This way, we  eventually arrive at bounds that use the Gowers-Host-Kra seminorms, in which case Theorem~\ref{T:main1} becomes applicable.
In addition to other things, this argument makes  essential use of a result of independent interest that allows to get  ``soft'' quantitative seminorm bounds from purely qualitative ones. This is the context of Proposition~\ref{P:Us}, which we  prove in a more abstract functional analytic setting in Proposition~\ref{P:abstract}
using  the Baire category theorem and the lemma of Mazur.
The reader will find a more detailed summary of the key ideas in the proofs and the obstacles we had to face in Section~\ref{SS:ideas}.

\subsection{Structure of the paper} The structure of the paper is as follows. In Section \ref{S:Results}, we state our main results. Sections \ref{S:background}, \ref{S:eigenfunctions}, and \ref{S:lemmas} contain fundamental definitions and results from ergodic theory together with several technical propositions that play essential role in our subsequent arguments. Sections~\ref{S:joint ergodicity} and \ref{S:degree lowering} are devoted to the proofs of Theorems~\ref{T:main1} and \ref{T:main2}: the former section sets up the induction scheme and reduces the argument to the degree lowering property while the latter gives the degree lowering argument in detail. In Section \ref{S:smoothing}, we present the proof of Theorem \ref{T:polies0}, the heart of which  is the seminorm smoothing argument. In Appendix \ref{A:soft}, the reader will find the statement and proof of Proposition~\ref{P:Us}, which enables a transition between qualitative and quantitative results. Finally, Appendix \ref{A:PET} contains the statement and proof of Proposition \ref{strong PET bound}, a PET induction bound which is a boosted version of \cite[Theorem 2.5]{DFMKS21} necessary for some of our arguments.

\subsection{Definitions and notation}
  The letters $\C, \R, \Z, \N_0, \N$ stand for the set of complex numbers, real numbers, integers, nonnegative integers, and positive integers.    With $\T$, we denote the one dimensional torus, and we often identify it with $\R/\Z$ or  with $[0,1)$. We let $[N]:=\{1, \ldots, N\}$ for any $N\in\N$.

  With  $\Re(z)$, we denote the real part of the complex number $z$. For $t\in \R$, we let $e(t):=e^{2\pi i t}$. With $\Z[n]$, we denote the collection of polynomials with integer coefficients.

  If $a\colon \N^s\to \C$ is a  bounded sequence for some $s\in \N$ and $A$ is a non-empty finite subset of $\N^s$,  we let
  $$
  \E_{n\in A}\,a(n):=\frac{1}{|A|}\sum_{n\in A}\, a(n).
  $$

  We commonly use the letter $\ell$ to denote the number of transformations in our system or the number of functions in an average while the letter $s$ usually stands for the degree of ergodic seminorms. We normally write  tuples of length $\ell$ in bold, e.g. $\b\in\Z^\ell$, and we underline tuples of length $s$ (or $s+1$, or $s-1$) that are typically used for averaging, e.g. $\uh\in\Z^s$.

\subsection{Acknowledgement} We would like to thank the referees for useful comments.

 \section{Main results} \label{S:Results}
In this section, we give a detailed description of our main results.
				Throughout this article, we make the standard assumption that all our probability  spaces $(X,\CX,\mu)$ are {\em regular} (or {\em Lebesgue}), in the sense that $X$ can be endowed with a metric that makes it a compact metric space and $\CX$ consists of all Borel sets in $X$. All notions present in the statements of the results below are defined in Section \ref{S:background}.
				\subsection{Joint ergodicity for general sequences}
				The results in this subsection extend results from~\cite{Fr21} that dealt with systems defined using a single transformation.
			We first state some ``global'' joint ergodicity results, which hold for all systems,  since they require a small amount of notation, and then we proceed to their ``local'' versions, which hold for particular systems.	
			\subsubsection{Global results} We start with the definition of some good properties needed in the statements of our main results, the reader will find the definition of the ergodic seminorms in Section~\ref{SS:seminorms} and of $\CE(T)$ in Section~\ref{SS:eigenfunctions}, let us just say for the moment that $\CE(T)$ consists of all non-ergodic eigenfunctions of the system and in the ergodic case it coincides with the set of eigenfunctions of the system.
				\begin{definition}The collection of sequences  $a_1,\ldots,a_\ell\colon \N\to \Z$ is
				\begin{itemize}	
					\item  {\em good for seminorm control}  if for every system   $(X, \CX, \mu,T_1,\ldots, T_\ell)$   there exists $s\in \N$ such that  the following holds: if $f_1,\ldots, f_\ell\in L^\infty(\mu)$ are such that
					$\nnorm{f_m}_{s,T_m}=0$ for some $m\in [\ell]$, and $f_j\in \CE(T_j)$ for $j=m+1,\ldots, \ell$, then
					\begin{equation}\label{E:zero}
						\lim_{N\to\infty}\E_{n\in[N]} \, T_1^{a_1(n)}f_1\cdots T_\ell^{a_\ell(n)}f_\ell=0
					\end{equation}
					in $L^2(\mu)$.\footnote{This condition is weaker, and in some cases easier to verify, than the more natural requirement that \eqref{E:zero} holds in $L^2(\mu)$ whenever
						$f_1,\ldots, f_\ell\in L^\infty(\mu)$ satisfy
						$\nnorm{f_m}_{s,T_m}=0$ for some $m\in [\ell]$.}

					\item  {\em good for  equidistribution} if
					$$
					\lim_{N\to\infty} \E_{n\in [N]}\, e(a_1(n)t_1+\cdots+a_\ell(n)t_\ell)=0
					$$
					for all  $t_1,\ldots, t_\ell\in [0,1)$ not all of them zero.

					\item  {\em good for irrational equidistribution} if
					\begin{equation}\label{E:zerot}
						\lim_{N\to\infty} \E_{n\in [N]}\, e(a_1(n)t_1+\cdots+a_\ell(n)t_\ell)=0
					\end{equation}
					for all  $t_1,\ldots, t_\ell\in [0,1)$ not all of them rational.
					
					\item {\em good for rational convergence} if the limit
					\begin{equation}\label{E:limt}
						\lim_{N\to\infty} \E_{n\in [N]}\, e(a_1(n)t_1+\cdots+a_\ell(n)t_\ell)
					\end{equation}
					exists for all   $t_1,\ldots, t_\ell\in \Q$.
				\end{itemize}
			\end{definition}
		\begin{remark}
			We also say informally that the seminorm $\nnorm{f_m}$ {\em controls the average}
			\begin{align*}
				\E_{n\in[N]}\, T_1^{a_1(n)}f_1\cdots T_\ell^{a_\ell(n)}f_\ell
			\end{align*}
			if the $L^2(\mu)$ limit of the average vanishes whenever $\nnorm{f_m} = 0$.
			\end{remark}
			For instance, we know that polynomial sequences with distinct degrees are good for seminorm control~\cite[Theorem 1.2]{CFH11}, linearly independent polynomial sequences with zero constant terms are good for irrational equidistribution, and we easily  get that  arbitrary polynomial sequences are good for rational convergence.

	{In the following statements and throughout, given a system $(X,\CX,\mu,T)$, we let $\mathcal{I}(T)$ be the subspace of $L^2(\mu)$ consisting of all $T$-invariant functions.  We also denote by $\Krat(T)$
		the $L^2(\mu)$ closure
		of the linear span of the rational eigenfunctions of the system, i.e. non-zero
		functions $f\in L^2(\mu)$ such that $Tf=e(t)\cdot f$  for some $t\in \mathbb{Q}$. Finally, given a function $f\in L^2(\mu)$ and a closed subspace $\mathcal{A}$ of $L^2(\mu)$, we denote by  $\E(f|\mathcal{A})$ the orthogonal projection of the function $f$ onto $\mathcal{A}$, often called the {\em conditional expectation of $f$ with respect to  $\mathcal{A}$}.  }

			Our first main result is the following:
					\begin{theorem}\label{T:main1}
				Let  $a_1,\ldots, a_\ell\colon \N\to \Z$ be sequences. 	The following two properties are equivalent:
					\begin{enumerate}
						\item \label{i:11}The sequences  $a_1,\ldots, a_\ell$ are good for seminorm control and equidistribution.
						
						\item \label{i:12} For all systems $(X, \CX, \mu,T_1,\ldots, T_\ell)$ and functions $f_1,\ldots,f_\ell\in L^\infty(\mu)$, we have
						\begin{equation}\label{E:liminv}
						\lim_{N\to\infty}\E_{n\in[N]} \,  T_1^{a_1(n)}f_1\cdots T_\ell^{a_\ell(n)}f_\ell =\E(f_1|\CI(T_1))\cdots \E(f_\ell|\CI(T_\ell))
						\end{equation}
						in $L^2(\mu)$.
					\end{enumerate}
				\end{theorem}
				Note that the previous result does not apply to polynomial sequences, a defect that the next result does not have.
					\begin{theorem}\label{T:main2}
						Let  $a_1,\ldots, a_\ell\colon \N\to \Z$ be sequences. The following two properties are equivalent:
					\begin{enumerate}
						\item \label{i:21}The sequences  $a_1,\ldots, a_\ell$ are good for seminorm control and  irrational equidistribution.
						
						\item \label{i:22} For all systems $(X, \CX, \mu,T_1,\ldots, T_\ell)$ and functions $f_1,\ldots,f_\ell\in L^\infty(\mu)$, we have
						\begin{equation}\label{E:Kratchar}
						\lim_{N\to\infty}\norm{\E_{n\in[N]}\prod_{j\in[\ell]} T_j^{a_j(n)}f_j -\E_{n\in[N]}\prod_{j\in[\ell]} T_j^{a_j(n)}\E(f_j|\Krat(T_j))}_{L^2(\mu)}=0.
						\end{equation}
					\end{enumerate}
				\end{theorem}
					The previous  result   combined with the known seminorm control properties from \cite[Theorem 1.2]{CFH11} implies that the rational Kronecker factors are characteristic when the iterates are given by distinct degree polynomials. This already improves on the results from \cite{CFH11} in which only a weaker statement \cite[Proposition 7.3]{CFH11} was proved. We will prove a much stronger result  though (see Theorem~\ref{T:polies2}), but with substantial additional effort.
		
			\begin{corollary}\label{C:main3}
				Let  $a_1,\ldots, a_\ell\colon \N\to \Z$ be sequences.	The following two properties are equivalent:
				\begin{enumerate}
					\item \label{i:31}The sequences  $a_1,\ldots, a_\ell$ are good for seminorm control, irrational equidistribution, and rational convergence.
					
					\item \label{i:32} For all systems $(X, \CX, \mu,T_1,\ldots, T_\ell)$ and functions $f_1,\ldots,f_\ell\in L^\infty(\mu)$,
					the limit
					\begin{equation}\label{E:limit}
						\lim_{N\to\infty}\E_{n\in[N]} \, T_1^{a_1(n)}f_1\cdots T_\ell^{a_\ell(n)}f_\ell
					\end{equation}
					exists in  $L^2(\mu)$ and \eqref{E:Kratchar} holds.
				\end{enumerate}
			\end{corollary}
			
			Heuristically, Theorems \ref{T:main1}, \ref{T:main2}, and Corollary \ref{C:main3} say that if the sequences $a_1,\ldots, a_\ell$  are good for seminorm control, then the convergence properties stated above hold for all systems if and only if they hold for all systems defined by circle rotations.
			
			In all results, 	the implication $\eqref{i:22}\implies \eqref{i:21}$  is very simple (the seminorm property follows trivially and the equidsitribution and rational convergence properties follow by considering appropriate rotations on the circle).
				So we will only bother ourselves with the forward implication.
			We will only give the proof of Theorem~\ref{T:main2}, the proof of Theorem~\ref{T:main1} is similar and simpler (it does not need  Proposition~\ref{P:finiterational}).
				Also, the mean convergence part of Corollary~\ref{C:main3} follows from Theorem ~\ref{T:main2} by using \eqref{E:Kratchar} and then proving mean convergence for the averages
				$$
				\E_{n\in[N]} \, T_1^{a_1(n)}(\E(f_1|\Krat(T_1)))\cdots T_\ell^{a_\ell(n)}(\E(f_\ell|\Krat(T_\ell)))
				$$
				 by approximating for all $j\in[\ell]$  the functions
				$\E(f_j|\Krat(T_j))$ in $L^2(\mu)$ with linear combinations of rational $T_j$-eigenfunctions and then using the rational convergence assumption of the sequences $a_1,\ldots, a_\ell$.
				

				\subsubsection{Local results}
            Theorems \ref{T:main1} and \ref{T:main2} will be derived from their  local variants in which we impose assumptions on a particular system and get conclusions for the same system. This allows for some  extra flexibility that will be utilised in an essential way later in the article, for example in the proof of Corollary~\ref{C:nil} and  Theorem~\ref{C:polies1'}.
                      In order to state the local results, we need to first define local variants of the good properties used in the previous subsection.
Again,          the reader will find the definition of the  various notions used in  Sections~\ref{SS:seminorms} and~\ref{SS:eigenfunctions}.
            \begin{definition}
            	The collection of sequences  $a_1,\ldots,a_\ell\colon \N\to \Z$ is
            	\begin{enumerate}
            		\item\label{i:local-seminorms} {\em good for seminorm control  for the  system   $(X, \CX, \mu,T_1,\ldots, T_\ell)$}, if    there exists $s\in \N$ such that  the following holds: if  $f_1,\ldots, f_\ell\in L^\infty(\mu)$ are such that
            		$\nnorm{f_m}_{s,T_m}=0$ for some $m\in [\ell]$, and $f_j\in \CE(T_j)$ for $j=m+1,\ldots, \ell$, then \eqref{E:zero} holds
            		in $L^2(\mu)$.
            		
            		\item\label{i:local-rational}    {\em good for equidistribution   for the  system
            			$(X, \CX, \mu,T_1,\ldots, T_\ell)$}, if for all  functions $\chi_j\in \CE(T_j)$,  $j\in [\ell]$, we have
            		$$
            		\lim_{N\to\infty}\E_{n\in[N]} \,  T_1^{a_1(n)}\chi_1\cdots T_\ell^{a_\ell(n)}\chi_\ell =\E(\chi_1|\CI(T_1))\cdots \E(\chi_\ell|\CI(T_\ell))
            		$$
            		in $L^2(\mu)$.

            		\item\label{i:local-irrational}  {\em good for  irrational equidistribution   for the  system
            			$(X, \CX, \mu,T_1,\ldots, T_\ell)$}, if for all functions $\chi_j\in \CE(T_j)$,  $j\in [\ell]$, we have
            		\begin{equation}\label{E:Krat1}
            			\lim_{N\to\infty}\norm{\E_{n\in[N]}\prod_{j\in[\ell]} T_j^{a_j(n)}\chi_j -\E_{n\in[N]}\prod_{j\in[\ell]} T_j^{a_j(n)}\E(\chi_j|\Krat(T_j))}_{L^2(\mu)}=0.
            		\end{equation}
            		
            		\item\label{:local-rational} \emph{good for rational convergence for the  system
            			$(X, \CX, \mu,T_1,\ldots, T_\ell)$} if for all functions  $f_j\in \Krat(T_j)$, $j\in[\ell]$, the limit \eqref{E:limit} exists  in $L^2(\mu)$.
            	\end{enumerate}
            \end{definition}
                The next lemma establishes   a natural relationship between the global and local versions of the good properties used in the statements of our main results. It will be proved in Section~\ref{SS:globallocal}. 		
        \begin{lemma}\label{L:necsuf}
        	The collection of sequences  $a_1,\ldots,a_\ell\colon \N\to \Z$ is
        	\begin{enumerate}
        		\item\label{i:semcon} good for seminorm control if and only if  it is good for seminorm control for every system.
        		\item\label{i:equisimple}    good for equidistribution  if  and only if  it is good for equidistribution for every system.

        		\item\label{i:irrsimple}  good for  irrational equidistribution  if  and only if
        		it is good for irrational equidistribution for every system.

        		\item\label{i:ratsimple}  good for rational convergence   if and only if  it is good for rational convergence for every system.
        	\end{enumerate}
        \end{lemma}

            The next two results are  local variants of Theorems~\ref{T:main1} and \ref{T:main2} respectively.
				\begin{theorem}\label{T:local-main1}
					Let  $a_1,\ldots, a_\ell\colon \N\to \Z$ be sequences and $(X, \CX, \mu,T_1,\ldots, T_\ell)$ be a system. 	The following two properties are equivalent:
					\begin{enumerate}
						\item \label{i:local11}The sequences  $a_1,\ldots, a_\ell$ are good for seminorm control and equidistribution for the system $(X, \CX, \mu,T_1,\ldots, T_\ell)$.

						\item \label{i:local12} For all   functions $f_1,\ldots,f_\ell\in L^\infty(\mu)$,
						equation  \eqref{E:liminv} holds
						in $L^2(\mu)$.
					\end{enumerate}
				\end{theorem}
				Note that Theorem~\ref{T:local-main1} combined with Lemma~\ref{L:necsuf} implies Theorem~\ref{T:main1}.
				

				\begin{theorem}\label{T:local-main2}
					Let  $a_1,\ldots, a_\ell\colon \N\to \Z$ be sequences and $(X, \CX, \mu,T_1,\ldots, T_\ell)$ be a system. The following two properties are equivalent:
					\begin{enumerate}
						\item \label{i:local21} The sequences  $a_1,\ldots, a_\ell$ are good for seminorm control and  irrational equidistribution for the system $(X, \CX, \mu,T_1,\ldots, T_\ell)$, and equation \eqref{E:Kratchar} holds when all but one of the functions are rational $T_j$-eigenfunctions.\footnote{For the implication (i)$\implies$(ii), instead of the last condition, we can assume good seminorm control for each system $(X\times X,\mu\times \mu, T_j\times T_j)$ for $j\in [\ell]$. }
						
						\item \label{i:local22} For all   functions $f_1,\ldots,f_\ell\in L^\infty(\mu)$,  equation \eqref{E:Kratchar} holds.
					\end{enumerate}
				\end{theorem}
			\begin{remark}
					In practice, equation  \eqref{E:Kratchar} is easy to verify when all but one of the functions are rational eigenfunctions. For instance, this is the case when the sequences $a_1,\ldots, a_\ell$ are good for irrational equidistribution,  see  Lemma~\ref{L:a1l}.
			\end{remark}
				Note that Theorem~\ref{T:local-main2} combined with Lemma~\ref{L:necsuf}   and Lemma~\ref{L:a1l} gives Theorem~\ref{T:main2}.

			Lastly, we  record one application to equidistribution results for nilsystems.\footnote{Given a nilpotent Lie group $G$ and a cocompact subgroup $\Gamma$, we call $X=G/\Gamma$ a {\em nilmanifold} and the transformation $Tx=bx$, where $b\in G, x\in X,$ a {\em nilrotation}. The system $(X, m_X,T)$, where $m_X$ is the Haar measure on $X$,   is called a {\em nilsystem}. }
			\begin{corollary}\label{C:nil}
				Suppose that the sequences $a_1,\ldots, a_\ell$ are good for equidistribution (or irrational equidistribution). Then for every  nilmanifold $X$ with Haar measure $m_X$ and commuting nilrotations
				$T_1,\ldots, T_\ell$ acting on $X$, the limit formula \eqref{E:liminv} (respectively \eqref{E:limit}) holds in $L^2(m_X)$ for all $f_1,\ldots, f_\ell\in L^\infty(m_X)$.
			\end{corollary}
	This result follows immediately  from Theorems~\ref{T:local-main1} and \ref{T:local-main2} because nilsystems always satisfy the  seminorm control property. This is so because from  \cite[Chapter~12, Theorem~17]{HK18}   there exists an  $s\in \N$ such that $\nnorm{\cdot}_s$
	(which equals $\nnorm{\cdot}_{s,T}$ for any ergodic nilrotation $T$ on $X$) is a norm in $L^\infty(\mu)$ and for all $j\in [\ell]$ we have $\nnorm{\cdot}_s\leq \nnorm{\cdot}_{T_j,s}$. 		

					\subsection{Characteristic factors for pairwise independent  polynomials}
For the convenience of exposition,  we work throughout the article, unless explicitly stated otherwise,  under the following standing assumption:
$$
\text{ All   polynomials } \, p_1,\ldots, p_\ell \, \text{  have integer coefficients and  zero  constant term.}
$$

It is well-known thanks to Host-Kra~\cite{HK05b} and Leibman~\cite{Lei05c}, that the Host-Kra factors are characteristic for the mean convergence of ergodic averages with   distinct polynomial iterates of a single transformation; this result has proved of utmost importance in more detailed studies of questions regarding such averages. A similar claim for averages involving commuting transformations and pairwise independent polynomial iterates has been conjectured in \cite[Section~1.3]{CFH11} (see also \cite[Problem~15]{Fr16}), but so far it has been well out of reach of the existing methods. We establish this conjecture using our newly developed seminorm smoothing technique.
\begin{theorem}\label{T:polies0}
		Let  $p_1,\ldots, p_\ell\in \Z[n]$ of degree at most $d$. The following two properties are equivalent:
		\begin{enumerate}
			\item The polynomials $p_1,\ldots, p_\ell$  are pairwise independent.
			
			\item There exists $s\in \N$ (depending only on $d$ and $\ell$) such that for all systems  $(X, \CX, \mu,T_1,\ldots, T_\ell)$ the factors $\CZ_s(T_1), \ldots, \CZ_s(T_\ell)$ are characteristic for the corresponding multiple ergodic averages \eqref{E:polies}, in the sense that the $L^2(\mu)$ limit of \eqref{E:polies} is 0 whenever $\E(f_j|\CZ_s(T_j)) = 0$ for some $j\in[\ell]$.
		\end{enumerate}			
\end{theorem}
	The converse direction is straightforward,\footnote{To see this, suppose that  $kp_i=lp_j$, for some distinct $i,j\in[\ell]$ and non-zero $k,l\in \Z$,  and  take $T_i:=R^k$, $T_j=R^l$, for some  weak mixing transformation $R$.} so the content of Theorem \ref{T:polies0} really is the forward direction. The implication (i)$\implies$(ii)
   was established   for distinct degree polynomials in \cite[Theorem~1.2]{CFH11}
	and the problem was open even when $\ell=2$ and $p_1(n)=n^2$ and $p_2(n)=n^2+n$, in which case it was proved in \cite{DFMKS21} under the additional assumptions that the transformations $T_1,T_2, T_1T_2^{-1}$ are ergodic.

The previous result is  of interest even for weak mixing systems and gives the following
seemingly simple but not easy to prove result.
\begin{corollary}\label{C:wm}
		Let  $p_1,\ldots, p_\ell\in \Z[n]$. The following two properties are equivalent:
\begin{enumerate}
	\item The polynomials $p_1,\ldots, p_\ell$  are pairwise independent.
	
	\item  If  $T_1,\ldots, T_\ell$ are commuting weak mixing transformations acting on a probability space $(X,\CX,\mu)$, then the multiple ergodic averages \eqref{E:polies} converge in $L^2(\mu)$  to the product of the integrals of the individual functions.
	\end{enumerate}
\end{corollary}
 When $T_1=\cdots=T_\ell$ the implication (i)$\implies$(ii)  is contained in the so called  weakly mixing PET result of Bergelson~\cite{Be87a}, and in this restricted case the result applies to all 
  distinct polynomials. Corollary~\ref{C:wm} is the natural analogue of this result  for the case of several commuting transformations and
 was open even when $\ell=2$ and $p_1(n)=n^2$, $p_2(n)=n^2+n$.

\subsection{Limit formula for linearly independent polynomials}
To prove  the next  result we use the bulk of the material contained in this article.
By combining our two main results,  Theorems \ref{T:local-main2} and \ref{T:polies0},  with Lemma~\ref{L:a1l} (see the remark following the lemma), we resolve the following conjecture from  \cite[Section~1.3]{CFH11}
 (see also \cite[Problem~16]{Fr16}).
	
\begin{theorem}\label{T:polies2}
	Let $p_1,\ldots, p_\ell\in \Z[n]$ be linearly independent polynomials. Then for every system $(X, \CX, \mu,T_1,\ldots, T_\ell)$, the rational Kronecker factor is characteristic for the corresponding ergodic averages, meaning that \eqref{E:Kratchar} holds with $p_1,\ldots, p_\ell$ in place of $a_1,\ldots, a_\ell$.
\end{theorem}	
\begin{remark}
In particular, if all the transformations are totally ergodic, then  the averages \eqref{E:n2n2n} converge in $L^2(\mu)$ to the product of the integrals of the individual functions, extending the main result in \cite{FrK05a}, which covers the case $T_1=\cdots=T_\ell$.
\end{remark}
This was established under the  assumption  $T_1=\cdots=T_\ell$  in \cite[Theorem~1.1]{FrK06}. For distinct degree monomials partial progress was made in  \cite[Proposition~7.3]{CFH11}.	The problem was open for distinct degree polynomials even in the case of nilsystems, since it seemed to depend upon some  difficult problems of algebraic nature; we sidestepped  this obstacle by using the degree lowering argument of Theorem~\ref{T:main2}. For linearly independent polynomials of not necessarily distinct degrees, the main additional obstacle was to establish the seminorm control property since the regular PET scheme only yields bounds in terms of box seminorms. Our seminorm smoothing argument allows us to pass from a control by box seminorms towards a control by Gowers-Host-Kra seminorms.

Using this result and the estimate from \cite[Lemma~1.6]{Chu11}, we deduce in a standard way (see for example the proof of \cite[Theorem~1.3]{FrK06})
the following multiple recurrence result with  Khintchine-type lower bounds conjectured in  \cite[Section~1.3]{CFH11} (see also \cite[Problem~16]{Fr16}). Recall that the standing assumption in all of our results is that the polynomials have zero constant terms.
 \begin{corollary}\label{C:polies2'}
 	Let $p_1,\ldots, p_\ell\in \Z[n]$ be linearly independent polynomials. Then for every system $(X, \CX, \mu,T_1,\ldots, T_\ell)$, set $A\in\CX$, and   $\varepsilon>0$, there exists $n\in\N$ such that
 	\begin{equation}\label{E:lower}
 	\mu(A\cap T_1^{-p_1(n)}A\cap \cdots \cap T_\ell^{-p_\ell(n)}A)\geq (\mu(A))^{\ell+1}-\varepsilon.
 	\end{equation}
 	In fact, the set of $n\in \N$ for which the previous lower bound holds has bounded gaps.
 \end{corollary}
\begin{remark}
$\bullet$	Using a standard argument, we can also deduce from Theorem~\ref{T:polies2} that the lower bounds \eqref{E:lower} hold for all linearly independent {\em jointly intersective polynomials,}\footnote{The family $p_1,\ldots, p_\ell\in \Z[n]$ is {\em jointly intersective} if for every $r\in \N$ there exists $n\in\N$ such that $r$ divides $p_i(n)$ for all $i\in[\ell]$.}  thus covering a special case of a conjecture of
	Bergelson-Leibman-Lesigne~\cite[Conjecture~6.3]{BLL08} (which claims  positive lower bounds in \eqref{E:lower} for all jointly intersective polynomials).
	
	$\bullet$ If the non-zero polynomials $p_1,\ldots, p_\ell$ are not pairwise independent, then a simple modification of  \cite[Theorem~2.1]{BHK05} shows that we cannot in general get lower bounds of the form \eqref{E:lower} even when $T_1,\ldots, T_\ell$ are powers of the same transformation, and in fact no power of $\mu(A)$ in place of $ (\mu(A))^{\ell+1}$ helps. So for $\ell=2$ our result shows that  the lower bounds $\eqref{E:lower}$ hold for all systems $(X,\mu,T_1,T_2)$ if and only if the polynomials $p_1,p_2$ are linearly independent. It would be interesting to get a similar characterization for $\ell\geq 3$.
\end{remark}
	Corollary~\ref{C:polies2'}  was established under the  assumption  $T_1=\cdots=T_\ell$  in \cite[Theorem~1.3]{FrK06} and for distinct degree monomials in \cite[Theorem~1.3]{Fr15}. For general linearly independent polynomials the problem   remained open even for weakly mixing transformations.

 By invoking the variant of the correspondence principle of Furstenberg~\cite{Fu81} that appears in~\cite{BL96}, one deduces the following combinatorial consequence.
\begin{corollary}\label{C:combinatorics}
	Let $p_1,\ldots, p_\ell\in \Z[n]$ be linearly independent polynomials. Then for every set $\Lambda\subset \Z^d$, ${\bf v}_1,\ldots, {\bf v}_\ell\in \Z^d$, and       $\varepsilon>0$ we have
 $$
 \bar{d}(\Lambda \cap (\Lambda+ p_1(n) {\bf v}_1) \cap \cdots\cap (\Lambda+p_\ell(n) {\bf v}_\ell) )\geq (\bar{d}(\Lambda))^{\ell+1}-\varepsilon
 $$
 for a set of $n\in \N$ with bounded gaps.
 \end{corollary}

\subsection{Results about the prime numbers}\label{SS:primes}
{We prove variants of the previous results dealing with  the  prime numbers, starting with a variant of  Theorem~\ref{T:polies2}. Below, $\P$ denotes the set of primes.
\begin{theorem}\label{T:polies2primes}
	Let $p_1,\ldots, p_\ell\in \Z[n]$ be linearly independent polynomials. Then for all systems $(X, \CX, \mu,T_1,\ldots, T_\ell)$ and functions $f_1,\ldots,f_\ell\in L^\infty(\mu)$, we have
$$
		\lim_{N\to\infty}\norm{\E_{n\in\P\cap[N]}\prod_{j\in[\ell]} T_j^{p_j(n)}f_j -\E_{n\in\P\cap[N]}\prod_{j\in[\ell]} T_j^{p_j(n)}\E(f_j|\Krat(T_j))}_{L^2(\mu)}=0.
$$
\end{theorem}	
Let us briefly see how this follows from  Theorem~\ref{T:polies2}. It suffices to show that if
 $\E(f_j|\Krat(T_j))=0$ for some $j\in[\ell]$, then the averages $\E_{n\in\P\cap[N]}\prod_{j\in[\ell]} T_j^{p_j(n)}f_j$ converge to $0$ in $L^2(\mu)$. By  the proof of \cite[Theorem~1.3]{FrHK11} or \cite[Theorem~1.2]{KT23} (the crucial number theoretic input needed comes from \cite[Theorem~7.2]{GT10}), it suffices to show  that if $\E(f_j|\Krat(T_j))=0$ for some $j\in[\ell]$, then  for every $W,b\in \N$ the averages  $ \E_{n\in[N]}\prod_{j\in[\ell]} T_j^{p_j(Wn+b)}f_j $
 converge  to $0$ in $L^2(\mu)$. Note that for every $W,b\in \N$ the polynomials $p_1(Wn+b)-p_1(b),\ldots, p_\ell(Wn+b)-p_\ell(b)$ have zero constant terms and
 are linearly independent if $p_1,\ldots, p_\ell$ are,   so the necessary convergence to $0$  follows from  Theorem~\ref{T:polies2}.

 Similarly, we get variants of Theorems~\ref{T:polies0} and Corollary~\ref{C:wm}, this time using that for every $W,b\in \N$ the  polynomials $p_1(Wn+b)-p_1(b),\ldots, p_\ell(Wn+b)-p_\ell(b)$ have zero constant terms and
 are pairwise independent if $p_1,\ldots, p_\ell$ are.

 Finally, using Theorem~\ref{T:polies2primes}, the fact that for every $r\in \N$ the set
 $$
 S_r:=\{n\in\P\cap\N\colon n\equiv 1 \! \! \! \pmod{r}\}
 $$
 has positive relative lower density inside $\P$, and
  the estimate from \cite[Lemma~1.6]{Chu11}, we deduce in a standard way (see for example the proof of \cite[Theorem~1.3]{FrK06}) the following variant of Corollary~\ref{C:polies2'}
  (and the obvious variant of Corollary~\ref{C:combinatorics}).
   \begin{corollary}\label{C:polies2primes}
  	Let $p_1,\ldots, p_\ell\in \Z[n]$ be linearly independent polynomials. Then for every system $(X, \CX, \mu,T_1,\ldots, T_\ell)$, set $A\in\CX$, and   $\varepsilon>0$, there exists $n\in\P-1$ such that
  	\begin{equation*}
  		\mu(A\cap T_1^{-p_1(n)}A\cap \cdots \cap T_\ell^{-p_\ell(n)}A)\geq (\mu(A))^{\ell+1}-\varepsilon.
  	\end{equation*}
  	In fact, the set of $n\in \P-1$ for which the previous lower bound holds has positive lower density.
  \end{corollary}
We can also replace $\P-1$ with $\P+1$ in the iterates, and taking $\ell=1$ and $p_1$ to be linear, we can see  that no other shift of the primes works.
  }

\subsection{Joint ergodicity conjecture for polynomial iterates} We start with a definition from \cite{DFMKS21}.
\begin{definition}
	We say that a sequence of measure preserving transformations $(T_n)_{n\in\N}$, acting on a probability space $(X,\CX,\mu)$, is {\em ergodic for $\mu$}, if
	$$
	\lim_{N\to\infty}\E_{n\in[N]}\, T_nf =\int f\, d\mu
	$$
	in $L^2(\mu)$ for every $f\in L^2(\mu)$.
\end{definition}
The following conjecture was made in \cite[Conjecture~1.5]{DKS19} (the conjecture there covers a more general setting).
	\begin{conjecture}\label{C:polies1}
		Let $p_1,\ldots, p_\ell\in \Z[n]$ and $(X, \CX, \mu,T_1,\ldots, T_\ell)$ be a system with $T_1,\ldots, T_\ell$ ergodic. The following two properties are equivalent:
		\begin{enumerate}			
			\item The sequence $(T_i^{p_i(n)}T_j^{-p_j(n)})_{n\in\N}$ is ergodic for $\mu$ for all
			distinct $i,j\in[\ell]$ and the sequence $(T_1^{p_1(n)}\times \cdots \times T_\ell^{p_\ell(n)})_{n\in\N}$ is ergodic for $\mu\times\cdots \times \mu$.
			\item   The polynomials $p_1,\ldots, p_\ell$ are jointly ergodic for $(X,\CX, \mu,T_1,\ldots, T_\ell)$, meaning,
			equation  \eqref{E:liminv} holds
			in $L^2(\mu)$  for all   functions $f_1,\ldots,f_\ell\in L^\infty(\mu)$.
		\end{enumerate}
	\end{conjecture}
We resolve Conjecture \ref{C:polies1} for families of pairwise independent polynomials.\footnote{In upcoming work we plan to use the techniques of this article to fully cover Conjecture~\ref{C:polies1}, but this  is  a much more demanding task.} In fact, we obtain the following rather pleasing statement.
 	\begin{theorem}\label{C:polies1'}
 	Let $p_1,\ldots, p_\ell\in \Z[n]$ be pairwise independent and $(X,\CX, \mu,T_1,\ldots, T_\ell)$ be a system with $T_1,\ldots, T_\ell$ ergodic. The following three properties are equivalent:
 	\begin{enumerate}	
 		\item The polynomials $p_1,\ldots, p_\ell$ are good for equidistribution for $(X,\CX,\mu,T_1,\ldots, T_\ell)$.
 				
 		\item The sequence $(T_1^{p_1(n)}\times \cdots \times T_\ell^{p_\ell(n)})_{n\in\N}$ is ergodic for $\mu\times\cdots \times \mu$.
 		
 		\item   The polynomials $p_1,\ldots, p_\ell$ are jointly ergodic for $(X, \CX, \mu, T_1,\ldots, T_\ell)$.
 	\end{enumerate}
 \end{theorem}
\begin{remark}
	If $(X,\CX, \mu,T_1,\ldots, T_\ell)$ is a general  system (no ergodicity condition imposed), Theorems \ref{T:main1} and \ref{T:polies0} show that if property $(i)$ holds, then
	equation  \eqref{E:liminv} holds
	in $L^2(\mu)$  for all   functions $f_1,\ldots,f_\ell\in L^\infty(\mu)$.
	\end{remark}
\begin{proof}
The direction (ii) $\Rightarrow$ (i) follows by taking non-ergodic eigenfunctions of $T_1, \ldots, T_\ell$, defined in Section \ref{SS:eigenfunctions}. The direction (i) $\Rightarrow$ (iii), which is the hardest,  is a consequence of our Theorems \ref{T:main1} and \ref{T:polies0}.

 We briefly explain the direction (iii) $\Rightarrow$ (ii).
 To show that the sequence $(T_1^{p_1(n)}\times \cdots \times T_\ell^{p_\ell(n)})_{n\in\N}$ is ergodic for $\mu\times\cdots \times \mu$, using an approximation argument and passing to weak limits, it suffices to show that
\begin{equation}\label{E:figi}
	\lim_{N\to\infty}\E_{n\in[N]}\, \prod_{j\in [\ell]} \int f_j\cdot T_j^{p_j(n)}g_j\, d\mu =\prod_{j\in [\ell]} \Big(\int f_j\, d\mu \cdot \int g_j\, d\mu\Big)
\end{equation}
for all $1$-bounded functions $f_j,g_j\in L^\infty(\mu)$, $j\in [\ell]$.
Equivalently, it suffices to show that if $\int f_j\, d\mu=0$ or $\int g_j\, d\mu=0$ for some $j\in [\ell]$, then the limit on the left is zero.
Note that  for every $j\in [\ell]$, this limit is  bounded by 	
$$
\limsup_{N\to\infty}\E_{n\in[N]}\,  \Big|\int f_j\cdot T_j^{p_j(n)}g_j	\, d\mu\Big|,
$$
and it is well known that this limit is zero whenever $f_j$ or $g_j$ is orthogonal to the Kronecker factor of $T_j$. Hence,  using another approximation
argument and linearity, it suffices to verify that \eqref{E:figi} holds when the functions $g_j$ are $T_j$-eigenfunctions for $j\in [\ell]$.  If all the functions $g_j$ are $T_j$-invariant, then they are constant because $T_j$ is ergodic, and \eqref{E:figi} holds trivially. Hence, we can assume that at least one of the eigenfunctions $g_j$ is not $T_j$-invariant, in which case it has  zero integral.   In this case, it suffices to show that
$$
\lim_{N\to\infty}\E_{n\in[N]}\, e(p_1(n)t_1+\cdots + p_\ell(n)t_\ell)=0
$$
when $t_j\in \spec(T_j)$, $j\in [\ell]$, are not all of them zero.
But this is an easy consequence of our joint ergodicity assumption.
\end{proof}

	\subsection{Decomposition results for multiple correlation sequences}
	Lastly, we record an application to decomposition results of multiple correlation sequences defined on systems of commuting measure preserving transformations.
	\begin{definition}
Following \cite{BHK05}, 		we say that a bounded sequence $(a(n))_{n\in\N}$ is a
		\begin{enumerate}
			\item {\em basic $d$-step nilsequence}, if there exists a $d$-step nilmanifold $X=G/\Gamma$,  $f\in C(X)$, and $g\in G$, $x\in X$, such that $a(n)=f(g^nx)$ for every $n\in\N$.
			
				\item {\em  $d$-step nilsequence}, if  it is the uniform limit of basic $d$-step nilsequences.
			
			\item {\em null-sequence}, if $\lim_{N\to \infty} \E_{n\in I_N} |a(n)|^2=0$ for every F\o lner sequence of subsets $I_N$, $N\in \N$,  of $\N$.\footnote{A sequence $(I_N)_{N\in\N}$ of finite subsets of $\Z^d$  is called {\em F\o lner}, if $\lim_{N\to\infty}\frac{|(I_N+\uh)\triangle I_N|}{|I_N|}=0$ for every $\uh\in\Z^d$.}
		\end{enumerate}
		
	\end{definition}
	
	\begin{definition}
		Let  $p_1,\ldots, p_\ell\in \Z[n]$  be polynomials and $(X, \CX, \mu,T_1,\ldots, T_\ell)$ be a system.
		We say that we have {\em nil plus null decomposition for the corresponding multicorrelation sequences} if there exists $d\in \N$ such that  for every $f_0,\ldots, f_\ell\in L^\infty(\mu)$, the sequence
		\begin{equation}\label{E:correlation}
		C(n):=\int f_0\cdot T_1^{p_1(n)}f_1\cdots T_\ell^{p_\ell(n)}f_\ell\, d\mu
		\end{equation}
		can be decomposed as a sum of  a $d$-step nilsequence and a null-sequence.		
	\end{definition}
	It was shown  in \cite{Fr15} that for every  $\varepsilon>0$, a  multicorrelation sequence $(C(n))_{n\in\N}$ of the form \eqref{E:correlation} can be decomposed as a sum of  a  $d$-step nilsequence and a sequence $(e(n))_{n\in\N}$ that satisfies
	 $$
	 \lim_{N\to \infty} \E_{n\in I_N} |e(n)|^2\leq \varepsilon
	 $$ for every F\o lner sequence $(I_N)_{N\in \N}$ of subsets of $\N$. A widely known question in the area  is   whether the stronger nil plus null decomposition holds
	 for all such multicorrelation sequences. This was established in \cite{Lei11} when all the transformations are equal   and  for commuting transformations   partial progress was made in \cite[Theorem~4.1]{Fe21} and in  \cite[Theorem~2.2]{DFMKS21}  under some rather strong  ergodicity assumptions. The next result answers this question affirmatively  for all collections
	 of pairwise independent polynomials without  making any ergodicity assumptions.

  \begin{theorem}\label{T:decomposition}
	Suppose that the polynomials  $p_1,\ldots, p_\ell\in \Z[n]$  are  pairwise independent.  Then  for every system $(X, \CX, \mu,T_1,\ldots, T_\ell)$, we have nil plus null decomposition for the corresponding multicorrelation sequences \eqref{E:correlation}.
\end{theorem}
For instance, Theorem \ref{T:decomposition} covers the multicorrelation sequences
$$
C(n):= \int f_0\cdot T_1^nf_1\cdot T_2^{n^2}f_2\cdot T_3^{n^2+n}f_3\, d\mu.
$$
Previously this case was covered  in  \cite[Theorem~2.2]{DFMKS21}  under  the additional assumption that all  the  transformations $T_1,T_2,T_3, T_3T_2^{-1}$ are ergodic.

Theorem~\ref{T:decomposition} is an immediate consequence of  Theorem~\ref{T:polies0} (in fact the variant that covers averaging over  F\o lner sequences, see comment at the end of Sections~\ref{SS:general}) and the following result, which is a rather easy consequence of  \cite[Proposition~4.2]{L15} and the decomposition result   \cite[Proposition~3.1]{CFH11}.
\begin{proposition}
	Let $(X, \CX, \mu,T_1, \ldots, T_\ell)$ be a system and $p_1,\ldots, p_\ell\in \Z[n]$ be polynomials.
	Suppose that for some $s\in \N$ the following holds: For every $j\in[\ell]$  the factors $\CZ_s(T_j\times T_j)$ are characteristic for $L^2(\mu\time\mu)$-convergence of the averages
	$$
	\lim_{N\to\infty}\E_{n\in[N]} \, (T_1\times T_1)^{p_1(n)}g_1\cdots (T_\ell\times T_\ell)^{p_\ell(n)}g_\ell
	$$
 and $g_1, \ldots, g_\ell\in L^\infty(\mu\times\mu)$.
	 Then the multicorrelation sequences defined on the system $(X, \CX, \mu,T_1,\ldots, T_\ell)$ admit a nil plus null decomposition.
	\end{proposition}
\begin{proof}
	Let  $C(n)$ be  as in \eqref{E:correlation} and  $\varepsilon>0$. We can assume that the functions $f_0,\ldots, f_\ell$ are $1$-bounded.
	
	 We start by combining our assumption with  the fact that if $f$ is orthogonal to  $\CZ_{s+1}(T)$,  	then   $f\otimes \bar{f}$ is orthogonal to  $\CZ_{s}(T\times T)$ (this is so, since it is easy to show that $\nnorm{f}_{s+1,T}=0 \implies \nnorm{f\otimes \bar{f}}_{s,T\times T}=0$).
	We deduce, using a standard argument  (see, for example, the beginning of Section~5 in \cite{L15}) that if $C(n)$ is as in \eqref{E:correlation} and
	$$
	\widetilde{C}(n):=\int f_0\cdot T_1^{p_1(n)}\tilde{f}_1\cdots T_\ell^{p_\ell(n)}\tilde{f}_\ell\, d\mu,
	$$
	where $\tilde{f}_j:=\E(f|\CZ_{s+1}(T_j))$ for $j\in [\ell]$,  then
	$$
	z(n):=C(n)-\widetilde{C}(n)
	$$
	is a nullsequence.
	Using   \cite[Proposition~3.1]{CFH11}, we get for each $j\in [\ell]$ a decomposition
	$$
	\tilde{f}_j=g_j+h_j
	$$ where $\norm{h_j}_{L^1(\mu)}\leq \varepsilon/\ell$ and  the sequence $(g_j(T_j^nx))$ is a basic $(s+1)$-step nilsequence for $\mu$-a.e. $x\in X$. A telescoping argument gives
	$$
	\tilde{C}(n)=\int f_0(x)\cdot g_1(T_1^{p_1(n)}x)\cdots  g_\ell(T_\ell^{p_\ell(n)}x)\, d\mu
	+e(n)
	$$
	where $\norm{e}_{\infty}\leq \varepsilon$.
	Note also that    since $(g_j(T_j^n x))$ is a basic $(s+1)$-step nilsequence, the sequence
	 $(g_j(T_j^{p_j(n)}x))$ is a basic  $d_j(s+1)$-step nilsequence~\cite[Proposition~2.3]{Lei05a} where $d_j:=\deg(p_j)$.  Hence,
for $\mu$-a.e. $x\in X$, the sequence
$$
N_x(n):=f_0(x)\cdot g_1(T_1^{p_1(n)}x)\cdots  g_\ell(T_\ell^{p_\ell(n)}x)
$$
is a  $1$-bounded  basic $d(s+1)$-step nilsequence where
	  $d$ is the maximum degree of the polynomials $p_1,\ldots, p_\ell$.  Furthermore,  the
	  map $x\mapsto N_x(n)$ is $\mu$-measurable for every $n\in \N$.
	  By  \cite[Proposition~4.2]{L15}, the sequence
	  $$
	  N(n):=\int N_x(n)\, d\mu
	  $$
	  is a $d(s+1)$-step nilsequence plus a null-sequence.
	  Summarizing, we have  shown that  for every $\varepsilon>0$, the sequence $(C(n))_{n\in\N}$ can be decomposed as
	  $$
	  C(n)=N(n)+z(n)+e(n)
	  $$
	  where $(N(n))_{n\in\N}$ is a $d(s+1)$-step nilsequence, $(z(n))_{n\in\N}$ is a null-sequence, and $\norm{e}_{\infty}\leq \varepsilon$.
	  Arguing as
	  in the end of the proof of \cite[Theorem~2.2]{DFMKS21}, we get the asserted statement.
\end{proof}

		 \subsection{Extensions to other averaging schemes}\label{SS:general} Our arguments can be modified in a straightforward manner to cover the more general setting, where we replace the averages over $[N]$ and single variable sequences by weighted averages over any F\o lner sequence in $\Z^d$ and sequences of $d$ variables. These rather trivial modifications lead to results that we will state shortly after giving a set of necessary definitions.

	 	\begin{definition}	Let $d\in \N$, $I=(I_N)_{N\in\N}$ be a F{\o}lner sequence of subsets of $\N^d$, and $w\colon \N^d\to\C$ be bounded. The collection of sequences  $a_1,\ldots,a_\ell\colon \N^d\to \Z$ is
	 	\begin{itemize}	
	 		\item  {\em good for seminorm control with respect to  $w$ and $I$} if for every fixed  system   $(X, \CX, \mu,T_1,\ldots, T_\ell)$,   there exists $s\in \N$ such that  the following property holds: if $f_1,\ldots, f_\ell\in L^\infty(\mu)$ are such that
	 		$\nnorm{f_m}_{s,T_m}=0$ for some $m\in [\ell]$, and $f_j\in \CE(T_j)$ for $j=m+1,\ldots, \ell$, then
	 		\begin{equation}
	 			\lim_{N\to\infty}\E_{\underline{n}\in I_N} \, w_{\underline{n}}\cdot T_1^{a_1(\underline{n})}f_1\cdots T_\ell^{a_\ell(\underline{n})}f_\ell=0
	 		\end{equation}
	 		in $L^2(\mu)$.
	 		\item  {\em good for  equidistribution with respect to $w$ and $I$} if
	 		$$
	 		\lim_{N\to\infty} \E_{\underline{n}\in I_N}\, w_{\underline{n}}\cdot  e(a_1(\underline{n})t_1+\cdots+a_\ell(\underline{n})t_\ell)=0
	 		$$
	 		for all  $t_1,\ldots, t_\ell\in [0,1)$ not all of them zero.

	 		\item  {\em good for irrational equidistribution with respect to $w$ and $I$} if
	 		\begin{equation}
	 			\lim_{N\to\infty} \E_{\underline{n}\in I_N }\, w_{\underline{n}} \cdot e(a_1(\underline{n})t_1+\cdots+a_\ell(\underline{n})t_\ell)=0
	 		\end{equation}
	 		for all  $t_1,\ldots, t_\ell\in [0,1)$ not all of them rational.

	 		\item {\em good for rational convergence with respect to $w$ and $I$} if the limit
	 		\begin{equation}\label{E:limt'}
	 			\lim_{N\to\infty} \E_{\underline{n}\in I_N}\,  w_{\underline{n}} \cdot e(a_1(\underline{n})t_1+\cdots+a_\ell(\underline{n})t_\ell)
	 		\end{equation}
	 		exists for all   $t_1,\ldots, t_\ell\in \Q$.
	 	\end{itemize}
	 \end{definition}

 \begin{theorem}\label{T:main1'}
 	Let $d\in \N$, $I=(I_N)_{N\in\N}$ be o F\o lner sequence of subsets of $\N^d$, and  $w\colon \N^d\to \C$ be a bounded sequence with average $W$ over $I$. 	Let also   $a_1,\ldots, a_\ell\colon \N^d\to \Z$ be sequences. 	The following two properties are equivalent:
 	\begin{enumerate}
 		\item \label{i:11'}The sequences  $a_1,\ldots, a_\ell$ are good for seminorm control and equidistribution with respect to $w$ and $I$.
 		
 		\item \label{i:12'} For all systems $(X, \CX, \mu,T_1,\ldots, T_\ell)$ and functions $f_1,\ldots,f_\ell\in L^\infty(\mu)$, we have
 		$$
 		\lim_{N\to\infty}\E_{\underline{n}\in I_N} \, w_{\underline{n}} \cdot T_1^{a_1(\underline{n})}f_1\cdots T_\ell^{a_\ell(\underline{n})}f_\ell =W\cdot \E(f_1|\CI(T_1))\cdots \E(f_\ell|\CI(T_\ell))
 		$$
 		in $L^2(\mu)$.
 	\end{enumerate}
 \end{theorem}

 \begin{theorem}\label{T:main2'}
 	Let $d\in \N$, $I=(I_N)_{N\in\N}$ be a F\o lner sequence of subsets of $\N^d$, and $w\colon \N^d\to \C$ be bounded. 	Let also  $a_1,\ldots, a_\ell\colon \N^d\to \Z$ be sequences.
 	The following two properties are equivalent:
 	\begin{enumerate}
 		\item \label{i:21'}The sequences  $a_1,\ldots, a_\ell$ are good for seminorm control and  irrational equidistribution with respect to $w$ and $I$.
 		
 		\item \label{i:22'} For all systems $(X, \CX, \mu,T_1,\ldots, T_\ell)$ and functions $f_1,\ldots,f_\ell\in L^\infty(\mu)$, we have
 		\begin{equation}\label{E:Kratchar'}
 			\lim_{N\to\infty}\norm{\E_{\underline{n}\in I_N}\, w_{\underline{n}}\cdot \prod_{j\in[\ell]} T_j^{a_j(\underline{n})}f_j -\E_{\underline{n}\in I_N}\, w_{\underline{n}}\cdot \prod_{j\in[\ell]} T_j^{a_j(\underline{n})}\E(f_j|\Krat(T_j))}_{L^2(\mu)}=0.
 		\end{equation}
 	\end{enumerate}
 \end{theorem}
 We immediately deduce the following statement.

 \begin{corollary}\label{C:main3'}
 	Let $d\in \N$, $I=(I_N)_{N\in\N}$ be a F{\o}lner sequence of subsets of $\N^d$, and $w\colon \N^d\to \C$ be bounded. 	Let also   $a_1,\ldots, a_\ell\colon \N^d\to \Z$ be sequences. The following two properties are equivalent:
 	\begin{enumerate}
 		\item \label{i:31'}The sequences  $a_1,\ldots, a_\ell$ are good for seminorm control, irrational equidistribution, and rational convergence with respect to $w$ and $I$.
 		
 		\item \label{i:32'} For all systems $(X, \CX, \mu,T_1,\ldots, T_\ell)$ and functions $f_1,\ldots,f_\ell\in L^\infty(\mu)$
 		the limit
 		\begin{equation}\label{E:limit'}
 			\lim_{N\to\infty}\E_{\underline{n}\in I_N} \, w_{\underline{n}}\cdot  T_1^{a_1(\underline{n})}f_1\cdots T_\ell^{a_\ell(\underline{n})}f_\ell
 		\end{equation}
 		exists in  $L^2(\mu)$ and \eqref{E:Kratchar'} holds.
 	\end{enumerate}
 \end{corollary}

 We can also get local versions of the previous results, extending the results in Theorems~\ref{T:local-main1} and \ref{T:local-main2}.

 Lastly, we can extend Theorem~\ref{T:polies0} to the case  where $p_1,\ldots,p_\ell\colon \N^d\to \Z$ are pairwise independent polynomials on $d$ variables and the averages are taken  over
  any F\o lner sequence  in $\Z^d$.\footnote{In order to have access to the needed variant of Proposition~\ref{P:Us} we use the mean convergence result in \cite{Zo15a} that covers
multivariate  	polynomials and averages over arbitrary  F\o lner sequences.}
  Using the previous results, it is an easy matter to  modify the proof of Theorem~\ref{T:polies2} and Corollary~\ref{C:polies2'}, as well as Theorems~\ref{C:polies1'} and   \ref{T:decomposition}, in order to get corresponding extensions to pairwise independent polynomials on $d$ variables and averages  over
  any F\o lner sequence in $\Z^d$.\footnote{Dealing with weighted averages requires to also impose additional mean convergence assumptions in order to have access to the needed variant of Proposition~\ref{P:Us}.}
  To do so, we need to  use a multivariate extension of
  Proposition~\ref{strong PET bound} with averages taken over  arbitrary F\o lner sequences. This extension  can be  obtained using \cite[Theorem 2.5]{DFMKS21} (as stated there) and
  a straightforward extension of \cite[Proposition 6.1]{Fr15a}.
  Lastly,  we also need the extension of the mean convergence result of Walsh~\cite{Wal12} that was obtained  by Zorin-Kranich in \cite{Zo15a}.
  Although these modifications are straightforward,    we chose not to state and prove our results in this more general setting, as this would add notational complexity to an already notationally heavy argument.

	\subsection{Key ideas and  obstacles}\label{SS:ideas}
	We explain the main innovations of the article  and the key obstacles that have to be tackled in order to prove our main results.
	
\subsubsection{Joint ergodicity results for general sequences.}		
The starting point in the proof of 	 Theorem~\ref{T:main1}  (or its localised version given in Theorem~\ref{T:local-main2})		
	 	is to use  a degree lowering argument
similar to the one used in \cite{Fr21} to deal with the case $T_1=\cdots=T_\ell$. This argument was motivated by work of Peluse~\cite{P19a} and Peluse and Prendiville~\cite{PP19} on   finitary variants of  special cases of   the polynomial Szemer\'edi theorem. If all the transformations $T_1,\ldots, T_\ell$ are ergodic, then the argument in \cite{Fr21} can be implemented without essential changes to get what we want; a summary of the main ideas of this argument can be found in \cite[Section~4.2]{Fr21}. The general case presents substantial additional difficulties; in order to overcome them, we  use:
\begin{itemize}
	\item The existence of a relative orthonormal basis of non-ergodic eigenfunctions \cite[Theorem~5.2]{FH18} in order to get Proposition \ref{U^2 inverse}, a non-ergodic variant of the inverse theorem for the Gowers-Host-Kra seminorms of degree 2.

\item An averaging trick based on the simple observation recorded in Lemma~\ref{L:verystrong}, which allows us to boost weak convergence to mean convergence in $L^2(\mu)$. This is needed in order to be able to pass information from the system to its ergodic components (see, for example  Corollary~\ref{C:verystrong}~\eqref{i:VS2}).

\item An ergodic decomposition argument with respect  to only one of the transformations; while doing so, we pay particular attention to the problems created by the fact that the other transformations no longer preserve the measure  defined by the ergodic component.  In particular, the $s=1$ case of Lemma~\ref{L:dense} is crucial since it allows to sidestep the non-trivial task of measurably selecting eigenfunctions in the proof of Proposition~\ref{P:Hx}.
\end{itemize}

For the proof of 	 Theorem~\ref{T:main2}  (or its localised version given in Theorem~\ref{T:local-main2}), the argument used in \cite{Fr21} to deal with
$T_1=\cdots=T_\ell$ collapses at the very beginning, since it is not possible to get an
ergodic decomposition that works simultaneously for all transformations, a fact that was crucial in  \cite{Fr21} in order to reduce to the case of systems with finite rational Kronecker factor. Instead we do the following:
\begin{itemize}
	\item We use Proposition~\ref{P:finiterational} that  enables us to get an  inverse
	theorem for the Gowers-Host-Kra seminorms  using only eigenfunctions belonging to
	a factor that has finite rational Kronecker; in proving this, we make essential use of a corollary from \cite{HK05a} recorded in Theorem~\ref{T:HostKra}.	
	
	\item We run the argument in the proof of Theorem~\ref{T:main1} and at a crucial point (towards the end of Section~\ref{SS:Pm-1}), we use Proposition~\ref{P:finiterational} in order to deduce that certain eigenfunctions have rational eigenvalues with bounded denominators.
\end{itemize}

	\subsubsection{Seminorm smoothing argument for polynomial sequences.}			
	For the proof 	of Theorem~\ref{T:polies0},  we argue as follows:
	\begin{itemize}
		\item Our starting point is to use \cite[Theorem~2.5]{DFMKS21}, which is based on a PET-induction argument and concatenation results from \cite{TZ16}, in order to get control  of the averages \eqref{E:polies} by the box seminorms $\nnorm{f_\ell}_{\b_1,\ldots, \b_{s+1}}$ defined in Section~\ref{SS:seminorms}. In fact, we boost \cite[Theorem~2.5]{DFMKS21} into Proposition \ref{strong PET bound}, which is of independent interest.
		
		\item Our goal is then to use a ``seminorm smoothing'' argument and deduce control by the seminorm  $\nnorm{f_\ell}_{\b_1,\ldots, \b_{s}, \be_\ell^{\times s'}}$. Applying this procedure $\ell$ times gives control by the seminorms  $\nnorm{f_\ell}_{\be_\ell^{\times s''}}$, which is exactly the Gowers-Host-Kra seminorm of order $s''$ of  the function $f_\ell$ with respect to the transformation $T_\ell$.
		
		\item Once we obtain a control of an average by the seminorm $\nnorm{f_\ell}_{\be_\ell^{\times s''}}$, we use induction and Proposition \ref{dual decomposition}, a decomposition result based on dual functions, to obtain a control by seminorms of other functions.
	\end{itemize}

		Roughly speaking,	the key ideas in the proof  of the seminorm smoothing argument for pairwise independent polynomials are as follows:
			\begin{itemize}
			\item We develop a ``ping-pong'' strategy in which we use the control of an average by the seminorm $\nnorm{f_\ell}_{\b_1,\ldots, \b_{s+1}}$ to obtain an intermediate control by a seminorm  $\nnorm{f_i}_{\b_1,\ldots, \b_s, \be_\ell^{\times s_1}}$ for a different function $f_i$ for some $i\neq \ell$.   We then go back, using the auxiliary control by the seminorm  $\nnorm{f_i}_{\b_1,\ldots, \b_s, \be_\ell^{\times s_1}}$ to deduce that the seminorm $\nnorm{f_\ell}_{\b_1,\ldots, \b_{s}, \be_\ell^{\times s'}}$ also controls the average. In this two-step strategy, we use two new versions of the dual-difference interchange argument presented in Proposition \ref{dual-difference interchange}: one of them involves invariant functions while the other uses dual functions.
			\item We introduce a notion of type which distinguishes simpler averages \eqref{E:polies} from more complex ones. It allows us to reduce the question of getting a seminorm control over \eqref{E:polies} to a simpler one, giving us an induction scheme that organises the proof of Theorem \ref{T:polies0}.
		\end{itemize}
		
		The reader will find a more in depth explanation of our seminorm smoothing argument
	as well as illustrative examples in Section~\ref{S:smoothing}.

			
			In order to implement the previous  plan, a key technical fact needed in the application of the induction hypothesis is the following:
				\begin{itemize} \item We can replace the  qualitative seminorm  control of the averages \eqref{E:polies}  with a quantitative one, uniformly with respect to all the functions involved.
	 The precise statement needed is given in Proposition~\ref{P:Us}.
	 \end{itemize}
	 	Unfortunately, the method in \cite{DFMKS21} seems hard to quantify in this way because of the way it relies on the concatenation results from \cite{TZ16}, and we were also not able to use a  compactness argument and the  Furstenberg correspondence principle to get what we want. So  in  order to get such a statement, we prove:
	 		\begin{itemize} \item A more  abstract functional analytic result in Proposition~\ref{P:abstract} that deals with continuity properties of multilinear forms.
			Two key ingredients in the proof are Mazur's lemma and the Baire category theorem.
			\end{itemize}



\section{Ergodic background and definitions}\label{S:background}
The next two sections discuss various notions from ergodic theory that are used extensively in the paper. Some of them are standard and well-known, others are lesser known, and some are coined for the purpose of this article.
\subsection{Basic notation}\label{SS:notation}
We start with explaining some additional basic notation used in this section and throughout the paper.

For a vector $\b=(b_1,\ldots, b_\ell)\in \Z^\ell$ and a system $(X, \CX, \mu, T_1, \ldots, T_\ell)$, we let
 $$
 T^{\b}:=T_1^{b_1}\cdots T_\ell^{b_\ell}.
 $$
For $j\in[\ell]$, we set $\be_j$ to be the unit vector in $\Z^\ell$ in the $j$-th direction, and we let $\be_0 = \mathbf{0}$, so that $T^{\be_j} = T_j$ for $j\in[\ell]$ and $T^{\be_0}$ is the identity transformation.

We often write $\ueps\in\{0,1\}^s$ for a vector of 0s and 1s of length $s$. For $\ueps\in\{0,1\}^s$ and $\uh, \uh'\in\Z^s$, we set
\begin{itemize}
    \item $\ueps\cdot \uh:=\eps_1 h_1+\cdots+ \eps_s h_s$;
    \item $\abs{\uh} := |h_1|+\cdots+|h_s|$;
    \item $\uh^\ueps := (h_1^{\eps_1}, \ldots, h_s^{\eps_s})$, where $h_j^0:=h_j$ and $h_j^1:=h_j'$ for $j=1,\ldots, s$;
\end{itemize}
{Note that the coordinates of $\uh^\ueps$ depend on both $\uh, \uh'$, but to simplify the notation
we do not write $(\uh, \uh')^\ueps$.}

We let $\CC z := \bar{z}$ be the complex conjugate of $z\in \C$.

\subsection{Ergodic seminorms}\label{SS:seminorms}
We review some basic facts about two families of ergodic seminorms: the Gowers-Host-Kra  seminorms and the box seminorms.
\subsubsection{Gowers-Host-Kra seminorms}
Given a system $(X, \CX, \mu,T)$, we will use the family of ergodic seminorms $\nnorm{\cdot}_{s, T}$,  also known as \emph{Gowers-Host-Kra seminorms}, which were originally introduced in  \cite{HK05a} for ergodic systems, and the reader will find
a detailed exposition of their basic properties in  \cite[Chapter~8]{HK18}.
 If we work with several measures at a time, we will also use the notation $\nnorm{\cdot}_{s, T, \mu}$ to specify the  measure with respect to which the seminorm is defined.
These seminorms are inductively defined for  $f\in L^\infty(\mu)$ as follows (for convenience, we also define $\nnorm{\cdot}_0$, which is   not a seminorm):
$$
\nnorm{f}_{0,T}:=\int f\, d\mu,
$$
and for $s\in \N_0$, we let
\begin{equation}\label{E:seminorm1}
	\nnorm{f}_{s+1,T}^{2^{s+1}}:=\lim_{H\to\infty}\E_{h\in [H]} \nnorm{\Delta_{T; h}f}_{s,T}^{2^{s}},
\end{equation}
where
$$
\Delta_{T;h}f:=f\cdot T^h\bar{f}, \quad h\in \Z,
$$
is the \emph{multiplicative derivative of $f$}  with respect to $T$.
The limit can be shown to exist by successive applications of the mean ergodic theorem and for $f\in L^\infty(\mu)$ and $s\in \N_0$ we have $\nnorm{f}_{s,T}\leq \nnorm{f}_{s+1,T}$ (see \cite{HK05a} or  \cite[Chapter~8]{HK18}).
It follows immediately from the definition that
$$
\nnorm{f}_{1,T}=\norm{\E(f|\CI(T))}_{L^2(\mu)},
$$
where $\CI(T)=\{f\in L^2(\mu)\colon Tf=f\}$ and
\begin{equation}\label{E:seminorm2}
	\nnorm{f}_{s,T}^{2^s}=\lim_{H_1\to\infty}\cdots \lim_{H_s\to\infty}\E_{h_1\in [H_1]}\cdots \E_{h_s\in [H_s]} \int \Delta_{s, T; \uh}f\, d\mu,
\end{equation}
where  for $\uh=(h_1,\ldots, h_s)\in \Z^s$, we let
$$
\Delta_{s,T;\uh}f:=\Delta_{T;h_1}\cdots \Delta_{T;h_s}f=\prod_{\ueps\in \{0,1\}^s}\mathcal{C}^{|\ueps|}T^{\epsilon\cdot \uh}f
$$
be the \emph{multiplicative derivative of $f$ of degree $s$} with respect to $T$.

It can be shown that we can take any $s'\leq s$  of the  iterative limits to be simultaneous limits (i.e. average over $[H]^{s'}$ and let $H\to\infty$)
without changing the value of the limit in \eqref{E:seminorm2}. This was originally proved in \cite{HK05a} using  the main structural result of \cite{HK05a}; a more ``elementary'' proof  can be deduced from \cite[Lemma~1.12]{BL15} once the convergence of the uniform Ces\`aro averages is known (and yet another proof can be found in \cite[Lemma~1]{Ho09}).
For $s'=s$, this gives the identity
\begin{equation}\label{E:seminorm}
\nnorm{f}_{s,T}^{2^s}=\lim_{H\to\infty}\E_{\uh\in [H]^s} \int \Delta_{s, T; \uh}f\, d\mu.
\end{equation}
 Moreover,  for $s'=s-2$, we have
\begin{equation}\label{E:seminorm4}
	\nnorm{f}_{s,T}^{2^{s}}=\lim_{H\to\infty}\E_{\uh\in [H]^{s-2}} \nnorm{\Delta_{s-2,T;\uh}f}_2^{4}.
\end{equation}

 Ergodic seminorms behave well under the ergodic decomposition: if $\mu = \int \mu_x \,d\mu(x)$ is the ergodic decomposition of $\mu$ with respect to $T$, then
\begin{equation}\label{E:seminonerg}
    \nnorm{f}_{s, T, \mu}^{2^s} = \int \nnorm{f}_{s, T, \mu_x}^{2^s} d\mu(x)
\end{equation}
for any $f\in L^\infty(\mu)$ and $s\in\N$. For the proof of this fact, see e.g. 
\cite[Chapter~8, Proposition~18]{HK18}.

It has been established in \cite{HK05a} for ergodic systems and in \cite[Chapter~8, Theorem~14]{HK18} for general systems, that the seminorms are intimately connected with a certain family of factors of the system. Specifically, for every $s\in\N$ there exists a factor $\CZ_s(T)\subseteq\CX$, known as the \emph{Host-Kra factor} of \emph{degree} $s$, with the property that
\begin{equation}\label{E:semifactor}
\nnorm{f}_{s, T} = 0  \text{ if and only if } f \text{ is orthogonal to } \CZ_{s-1}(T).
\end{equation}
 Equivalently, $\nnorm{\cdot}_{s, T}$ defines a norm on the space $L^2(\CZ_{s-1}(T))$ (for a proof see  \cite[Theorem~15, Chapter~9]{HK18}).

\subsubsection{Box seminorms} More generally, in Section~\ref{S:smoothing} we use analogues of \eqref{E:seminorm} defined with regards to several commuting transformations. These seminorms originally appeared in the work of Host \cite{Ho09}; their finitary versions are often called \emph{box seminorms}, and we sometimes employ this terminology. Let $(X, \CX, \mu, T_1, \ldots, T_\ell)$ be a system.
 For each $f\in L^{\infty}(\mu)$, $h\in \Z$, and $\b \in \Z^\ell$, we define
 $$
 \Delta_{\b; h} f := f \cdot T^{\b h} \bar{f}
 $$
   and for $\uh\in\Z^s$ and $\b_1,\ldots, \b_s\in \Z^\ell$, we let
$$
\Delta_{{\b_1}, \ldots, {\b_{s}}; \uh} f  := \Delta_{{\b_1; h_1}}\cdots\Delta_{{\b_s; h_s}} f = \prod_{\ueps\in\{0,1\}^s} \CC^{|\ueps|} T^{\b_1 \eps_1 h_1 + \cdots + \b_s \eps_s h_s}f.
$$
 We let
$$
\nnorm{f}_{\emptyset}:=\int f\, d\mu
$$
and
\begin{align}\label{ergodic identity}
	\nnorm{f}_{{\b_1}, \ldots, {\b_{s+1}}}^{2^{s+1}}:=\lim_{H\to\infty}\E_{h\in[H]}\nnorm{\Delta_{{\b_{s+1}}; h}f}_{{\b_1}, \ldots, {\b_{s}}}^{2^s}.
\end{align}
In particular, if $\b_1 = \cdots = \b_s:=\b$, then $ \Delta_{\b_1, \ldots, \b_s; \uh}=\Delta_{s, T^\b; \uh}$ and $  \nnorm{\cdot}_{\b_1, \ldots, \b_s}=\nnorm{\cdot}_{s, T^\b}$.
 We remark that these seminorms were defined in a slightly different way in \cite{Ho09}
 and the above identities were established in  \cite[Section~2.3]{Ho09}.

 Iterating \eqref{ergodic identity}, we get the identity
 \begin{equation}\label{E:seminormb1}
 	\nnorm{f}_{{\b_1}, \ldots, {\b_{s}}}^{2^{s+1}}=\lim_{H_1\to\infty}\cdots \lim_{H_s\to\infty}\E_{h_1\in [H_1]}\cdots \E_{h_s\in [H_s]} \int\Delta_{{\b_1; h_1}}\cdots\Delta_{{\b_s; h_s}} f\, d\mu,
 \end{equation}
which extends \eqref{E:seminorm2}. In complete analogy with the remarks made for the Gowers-Host-Kra seminorms, we have the following:
 using  \cite[Lemma~1]{Ho09}\footnote{Which implies  the convergence of the uniform Ces\`aro averages over $\uh\in \Z^s$ of  $\int \Delta_{{\b_1}, \ldots, {\b_{s}}; \uh} f\, d\mu$.} and \cite[Lemma~1.12]{BL15}, we get
 that we can take any $s'\leq s$  of the  iterative limits to be simultaneous limits (i.e. average over $[H]^{s'}$ and let $H\to\infty$)
without changing the value of the limit in \eqref{E:seminormb1}.
Taking $s'=s$  gives  the identity
$$
\nnorm{f}_{{\b_1}, \ldots, {\b_{s}}}^{2^s}=\lim_{H\to\infty}\E_{\uh\in [H]^s}\,\int \Delta_{{\b_1}, \ldots, {\b_{s}}; \uh} f\, d\mu.
$$
More generally, for any $1\leq s'\leq s$ and $f\in L^\infty(\mu)$, we get the identity
\begin{align}\label{inductive formula}
	\nnorm{f}_{\b_1, \ldots, \b_s}^{2^s} = \lim_{H\to\infty}\E_{\uh\in[H]^{s-s'}}\nnorm{\Delta_{\b_{s'+1}, \ldots, \b_s; \uh}f}_{\b_1, \ldots, \b_{s'}}^{2^{s'}}.
\end{align}

As an example of a box seminorm that  is not a Gowers-Host-Kra seminorm, consider $s=2$ and the vectors $\be_1=(1,0)$, $\be_2=(0,1)$, in which case
$$
\nnorm{f}_{{\be_1},{\be_2}}^4 = \lim_{H\to\infty}\E_{h_1, h_2\in [H]^2}\int f\cdot T_1^{h_1}\bar{f}\cdot T_2^{h_2}\bar{f}\cdot T_1^{h_1}T_2^{h_2}f\, d\mu.
$$
More generally, for  $s=2$ and $\ba=(a_1,a_2)$, $\b=(b_1,b_2)$, we have
$$
\nnorm{f}_{{\ba},{\b}}^4=\lim_{H\to\infty}\E_{h_1, h_2\in [H]^2}\int f\cdot T_1^{a_1h_1}T_2^{a_2h_1}\bar{f}\cdot T_1^{b_1h_2}T_2^{b_2h_2}\bar{f}\cdot T_1^{a_1h_1+b_1h_2}T_2^{a_2h_1+b_2h_2}f\, d\mu.
$$
If the vector $\ba$ repeats $s$ times, we abbreviate it as $\ba^{\times s}$, e.g.
\begin{align*}
    \nnorm{f}_{\ba^{\times 2}, \b^{\times 3}, \bc}=\nnorm{f}_{\ba, \ba, \b, \b, \b, \bc}.
\end{align*}

Box seminorms satisfy the following Gowers-Cauchy-Schwarz inequality  \cite[Proposition~2]{Ho09}:
\begin{align}\label{E:GCS}
    \limsup_{H\to\infty}\abs{\E_{\uh\in[H]^s}\int \prod_{\ueps\in\{0,1\}^s} \CC^{|\ueps|} T^{\b_1 \eps_1 h_1 + \cdots + \b_s \eps_s h_s}f_\ueps\, d\mu}\leq \prod_{\ueps\in\{0,1\}^s}\nnorm{f_\ueps}_{\b_1, \ldots, \b_s}.
\end{align}
(One can replace the limsup with a limit since it is known to exist.) Later on, we mention two different versions thereof needed at various spots of our arguments.

We frequently bound one seminorm in terms of another. A simple application of the Gowers-Cauchy-Schwarz inequality \eqref{E:GCS} yields the following monotonicity property:
\begin{align}\label{monotonicity property}
    \nnorm{f}_{\b_1, \ldots, \b_s}\leq \nnorm{f}_{\b_1, \ldots,\b_s,  \b_{s+1}},
\end{align}
a special case of which is the bound $\nnorm{f}_{s,T}\leq \nnorm{f}_{s+1, T}$ for any $f\in L^\infty(\mu)$ and system $(X, \CX, \mu, T)$.

Quite frequently, we have to deal simultaneously both with a collection of transformations and their powers.  The following estimate compares the box seminorms in both cases and provides a quantitative proof of \cite[Theorem 2]{L05'}.
\begin{lemma}\label{L:seminorm of power}
    Let $\ell,s\in\N$, $(X, \CX, \mu, T_1, \ldots, T_\ell)$ be a system, $f\in L^\infty(\mu)$ be a function, $\b_1, \ldots, \b_s\in\Z^\ell$ be vectors and $r_1, \ldots, r_s\in\Z$ be nonzero. Then
    \begin{align}\label{seminorm of power 1}
       \nnorm{f}_{\b_1, \ldots, \b_s}\leq \nnorm{f}_{r_1 \b_1, \ldots, r_s \b_s},
    \end{align}
    and if $s\geq 2$, we additionally get the bound
    \begin{align}\label{seminorm of power 2}
        \nnorm{f}_{r_1 \b_1, \ldots, r_s \b_s} \leq (r_1\cdots r_s)^{1/2^{s}}\nnorm{f}_{\b_1, \ldots, \b_s}.
    \end{align}
\end{lemma}

\begin{proof}
We assume that $r_1, \ldots, r_s$ are positive, otherwise the result follows from the fact that $\nnorm{f}_{r_1 \b_1, \ldots, r_s \b_s} = \nnorm{f}_{|r_1| \b_1, \ldots, |r_s| \b_s}$.


For the inequality \eqref{seminorm of power 1}, we observe that up to $O(H^{s-1})$ points, the box $[r_1 H]\times \cdots \times [r_s H]$ can be split into $r_1 \cdots r_s$ arithmetic progressions $(r_1\cdot [H])\times \cdots \times (r_s\cdot [H])  + \ui$. Using the fact that the seminorm can be defined by taking a limit along any F{\o}lner sequence and applying the aforementioned partition, we rewrite
\begin{align*}
    \nnorm{f}_{\b_1, \ldots, \b_s}^{2^{s}} &=\lim_{H\to\infty}\E_{\uh\in [r_1 H]\times \cdots \times [r_s H]} \int \prod\limits_{{\ueps}\in \{0,1\}^s}\mathcal{C}^{|{\ueps}|}T^{\b_1 \eps_1 h_1 + \cdots + \b_s\eps_s h_s}f\, d\mu \\   &= \E_{\ui\in[r_1]\times\cdots\times[r_s]}\lim_{H\to\infty}\E_{\uh\in [H]^s} \int \prod\limits_{{\ueps}\in \{0,1\}^s}\mathcal{C}^{|{\ueps}|}T^{r_1\b_1 \eps_1 h_1 + \cdots + r_s\b_s\eps_s h_s}f_{\ueps, \ui}\, d\mu,
\end{align*}
where $f_{\ueps, \ui}:= T^{\b_1 \eps_1 i_1 + \cdots + \b_s\eps_s i_s}f$. The result then follows from an application of the Gowers-Cauchy-Schwarz inequality \eqref{E:GCS} and the observation that $\nnorm{f_{\ueps, \ui}}_{r_1 \b_1, \ldots, r_s \b_s} = \nnorm{f}_{r_1 \b_1, \ldots, r_s \b_s}$ for any $\ueps\in\{0,1\}^s$ and $\ui\in\Z^s$.

We move on to prove the inequality \eqref{seminorm of power 2}. From the inductive definition \eqref{ergodic identity} of the seminorm, we have
\begin{align*}
    \nnorm{f}_{r_1 \b_1, \ldots, r_s \b_s}^{2^s}&=\lim_{H\to\infty}\E_{h_s\in [H]}\nnorm{f\cdot T^{r_s\b_s h_s}\bar{f}}_{r_1 \b_1, \ldots, r_{s-1} \b_{s-1}}^{2^{s-1}}\\
    &= \lim_{H\to\infty}\E_{h_s\in r_s\cdot[H]}\nnorm{f\cdot T^{\b_s h_s}\bar{f}}_{r_1 \b_1, \ldots, r_{s-1} \b_{s-1}}^{2^{s-1}}.
\end{align*}
The assumption $s\geq 2$ implies that $\nnorm{f\cdot T^{r_s\b_s h_s}\bar{f}}_{r_1 \b_1, \ldots, r_{s-1} \b_{s-1}}$ is nonnegative. Extending the summation over $r_s\cdot [H]$ to the summation over $[r_s H]$ by nonnegativity, we obtain the bound
\begin{align*}
    \nnorm{f}_{r_1 \b_1, \ldots, r_s \b_s}^{2^s} \leq r_s\lim_{H\to\infty}\E_{h_s\in [r_s H]}\nnorm{f\cdot T^{\b_s h_s}\bar{f}}_{r_1 \b_1, \ldots, r_{s-1} \b_{s-1}}^{2^{s-1}},
\end{align*}
and so
\begin{align*}
    \nnorm{f}_{r_1 \b_1, \ldots, r_s \b_s}^{2^s} \leq r_s \nnorm{f}_{r_1 \b_1, \ldots, r_{s-1} \b_{s-1}, \b_s}^{2^s}
\end{align*}
by the inductive definition of the seminorm \eqref{ergodic identity}. Repeating this procedure $s-1$ more times, we arrive at the claimed inequality.
\end{proof}

\subsection{Dual functions and sequences}\label{SS:dual}
Let $s\in\N$ and $\{0,1\}^s_* = \{0,1\}^s\setminus\{\underline{0}\}$. For a system $(X, \CX, \mu, T)$ and $f\in L^\infty(\mu),$ we define
\begin{align*}
    \CD_{s, T}(f) := \lim_{M\to\infty}\E_{\um\in [M]^s}\prod_{\ueps\in\{0,1\}^s_*}\CC^{|\ueps|}T^{\ueps\cdot\um}f
\end{align*}
(the limit exists in $L^2(\mu)$ by \cite{HK05a}).
  We call $\CD_{s,T}(f)$ the \emph{dual function} of $f$ of \emph{level} $s$ with respect to $T$. In some cases, when the underlying measure is not clearly defined, we write $\CD_{s,T,\mu}(f)$.  The name comes because of the identity
\begin{align}\label{dual identity}
    \nnorm{f}_{s, T}^{2^s} = \int f \cdot \CD_{s, T}(f)\, d\mu,
\end{align}
a consequence of which is that the span of dual functions of degree $s$ is dense in $L^1(\CZ_{s-1}(T))$.

Let $(X, \CX, \mu, T_1, \ldots, T_\ell)$ be a system.
Using  the identities \eqref{inductive formula} and  \eqref{dual identity} we get
\begin{equation}\label{dual inverse}
\nnorm{f}_{{\b_1}, \ldots, {\b_{s}}, \be_i^{\times s'}}^{2^{s+s'}}=\lim_{H\to\infty}\E_{\uh\in [H]^s}\int  \Delta_{{\b_1}, \ldots, {\b_{s}}; \uh} f \cdot \CD_{s', T_i}(\Delta_{{\b_1}, \ldots, {\b_{s}}; \uh}f)\, d\mu,
\end{equation}
the special case of which is
\begin{equation}\label{invariant inverse}
\nnorm{f}_{{\b_1}, \ldots, {\b_{s}}, \be_i}^{2^{s+1}} = \lim_{H\to\infty}\E_{\uh\in [H]^s} \int  \Delta_{{\b_1}, \ldots, {\b_{s}}; \uh}f \cdot \E(\Delta_{{\b_1}, \ldots, {\b_{s}}; \uh}\bar{f}|\CI(T_i))\, d\mu.
\end{equation}

For $s\in\N$, we denote
 \begin{equation}\label{E:Ds}
 \FD_s := \{(T_j^n \CD_{s', T_j}f)_{n\in\Z}\colon\; f\in L^\infty(\mu),\ j\in[\ell],\ 1\leq s'\leq s\}
 \end{equation}
to be the set of sequences of 1-bounded functions coming from dual functions of degree up to $s$ for the transformations $T_1,\ldots, T_\ell$, and moreover we define $\FD := \bigcup_{s\in\N}\FD_s$.

The utility of dual functions comes from the following approximation result that will be used in  Section~\ref{S:smoothing}.
\begin{proposition}[Dual decomposition, {\cite[Proposition 3.4]{Fr15a}}]\label{dual decomposition}
    Let $(X, \CX, \mu, T)$ be a system, $f\in L^\infty(\mu)$, $s\in\N$, and $\varepsilon>0$. Then we can decompose $f = f_1 + f_2 + f_3$, where
    \begin{enumerate}
        \item (Structured component) $f_1 = \sum_k c_k \CD_{s, T}(g_k)$ is a linear combination of (finitely many) dual functions of level $s$ with respect to $T$;
        \item (Small component) $\nnorm{f_2}_{L^1(\mu)}\leq \varepsilon$;
        \item (Uniform component) $\nnorm{f_3}_{s, T} = 0$.
    \end{enumerate}
\end{proposition}
Proposition \ref{dual decomposition} will be used in the following way. Suppose that the $L^2(\mu)$ limit of the average $\E_{n\in[N]}\, T_1^{a_1(n)}f_1 \cdots T_\ell^{a_\ell(n)}f_\ell$ vanishes whenever $\nnorm{f_\ell}_{s, T_\ell} = 0$. If
\begin{align*}
    \limsup_{N\to\infty}\norm{\E_{n\in[N]}\, T_1^{a_1(n)}f_1 \cdots T_\ell^{a_\ell(n)}f_\ell}_{L^2(\mu)}>0
\end{align*}
for some functions $f_1, \ldots, f_\ell\in L^\infty(\mu)$, then we decompose $f_\ell$ as in Proposition \ref{dual decomposition} for sufficiently small $\varepsilon>0$ so that
\begin{align*}
    \limsup_{N\to\infty}\norm{\E_{n\in[N]}\prod_{j\in[\ell-1]}T_j^{a_j(n)}f_j \cdot \sum_k c_k T_\ell^{a_\ell(n)}\CD_{s, T_\ell}(g_k)}_{L^2(\mu)}>0
\end{align*}
for some (finite) linear combinations of dual functions. Applying the triangle inequality and the pigeonhole principle, we deduce that there exists $k$ for which
\begin{align*}
    \limsup_{N\to\infty}\norm{\E_{n\in[N]}\prod_{j\in[\ell-1]}T_j^{a_j(n)}f_j \cdot \CD(a_\ell(n))}_{L^2(\mu)}>0,
\end{align*}
where $\CD(n)(x) := T_\ell^{n}\CD_{s, T_\ell}g_k(x)$ for $n\in\N$ and $x\in X$. This way, we replace the term $T_\ell^{a_\ell(n)}f_\ell$ in the original average by the  more structured piece $\CD(a_\ell(n))$.

We are also going to use the following result (only the case $s=1$ will be needed in the proof of Proposition~\ref{P:Hx}):
\begin{lemma}\label{L:dense}
	Let $(X,\CX,\mu,T)$  be a system with ergodic decomposition $\mu=\int \mu_x\, d\mu$ and $s\in \N$.  Then there exists a countable  set of functions $\CA\subset L^\infty(\CZ_s(T,\mu))$ such that
	for $\mu$-a.e. $x\in X$,  the set $\CA$ is dense in   $L^2(\CZ_s(T,\mu_x))$.
\end{lemma}
\begin{proof}
	Let $\CF$ be a countable dense subset of $C(X)$, then $\CF$ is dense in  $L^2(\mu_x)$ for every $x\in X$ (since $\mu_x$ is a regular measure for every $x\in X$). We claim that the set  $\CA$  of finite linear combinations with rational coefficients of dual  functions of the form $\CD_{s,T,\mu}(f)$,  where $f$ ranges over the set $ \CF$, does the job. To begin with, we deduce from  \cite[Proposition~3.2]{CFH11}  that for $\mu$-a.e. $x\in X$, the following holds:
	$$
	\CD_{s,T,\mu}(f)= \CD_{s,T,\mu_x}(f) \text{ with respect to the measure } \mu_x\text{ for every } f\in \CF.
	$$
	By \cite[Proposition~3.2]{CFH11} (or \cite[Chapter~9, Theorem~12]{HK18}), for every $x\in X$,  finite linear combinations with rational coefficients of functions of the form
	$\CD_{s,T,\mu_x}(f)$, where $f$ ranges over  $ L^\infty(\mu_x)$, are dense in $L^2(\CZ_s(T,\mu_x))$. Since
	$\CF$ is dense in $L^2(\mu_x)$ for every $x\in X$, the same property holds when $f$ ranges over $\CF$, and in this case  the collection of such linear combinations is countable.
	The claim follows.
\end{proof}

\section{Non-ergodic eigenfunctions and related inverse theorems}\label{S:eigenfunctions}

\subsection{Non-ergodic eigenfunctions and relative orthonormal basis} 	\label{SS:eigenfunctions}
The identity \eqref{dual identity} provides a weak inverse theorem for Gowers-Host-Kra seminorms. For $s=2$, given an ergodic system $(X, \CX, \mu, T)$, a strong inverse theorem is known and frequently used; it is based on the fact that the $\CZ_1(T)$ factor coincides with the Kronecker factor of the system, i.e. the factor spanned by eigenfunctions. This identification no longer holds for non-ergodic systems since there are systems such as $T(x,y) = (x, y+x)$ on $\T^2$ with no eigenfunctions but a nontrivial $\CZ_1(T)$ factor. To overcome this difficulty, we use the following more generalised notion of eigenfunctions, originally defined in \cite{FH18}.
		\begin{definition}
			Let  $(X, \CX, \mu,T)$ be a system,  $\chi\in L^\infty(\mu)$, and  $\lambda\in L^\infty(\mu)$ be a $T$-invariant function. We say that $\chi\in L^\infty(\mu)$ is a {\em (non-ergodic) eigenfunction}
			with {\em eigenvalue $\lambda$} if
			\begin{enumerate}
                \item $|\chi(x)|$ has value $0$ or $1$ for $\mu$-a.e. $x\in X$, and $\lambda(x)=0$ whenever $\chi(x)=0$;
                \item $T\chi=\lambda\, \chi$, $\mu$-a.e..
			 \end{enumerate}
		\end{definition}
		 For ergodic systems, a non-ergodic eigenfunction is either the  zero function  or  a classical  unit modulus eigenfunction. In general, if $\mu=\int \mu_x\, d\mu$ is the ergodic decomposition of $\mu$, then $f\in \CE(T,\mu)$ if and only if $f\in L^\infty(\mu)$ has modulus $0$ or $1$ and for  $\mu$-a.e. $x\in X$, the function $f\in \CE(T,\mu_x)$ is a usual eigenfunction, i.e. the identity $Tf = e(\phi(x)) f$ holds $\mu_x$-almost everywhere for some $\phi(x)\in \T$. In fact, for $f\in \CE(T,\mu)$, we have $f(Tx)={\bf 1}_E(x) \, e(\phi(x))\, f(x)$ for some $T$-invariant set $E\in \CX$ and measurable $T$-invariant function $\phi \colon X\to \T$.

				\begin{definition} For a given system $(X, \CX, \mu,T)$, we let
				\begin{itemize}
					\item $\CE(T):= \{\chi \in L^\infty(\mu)\colon \chi\,  \text{ is a non-ergodic eigenfunction}\}$;
					\item $\CI(T):=\{f\in L^2(\mu)\colon Tf=f\}$.
					\item $\CK_r(T):=\CI(T^r)$.
					\item 	$\Krat(T)$ be the closed subspace of $L^2(\mu)$  spanned by all $f\in L^2(\mu)$ such that  $Tf=e(\alpha)f$ for some  $\alpha\in \Q$.
					
					\item 			$\CK(T)$ be the closed subspace of $L^2(\mu)$ spanned by all $f\in L^2(\mu)$ such that  $Tf=e(\alpha)f$ for some  $\alpha\in \R$.
				\end{itemize}
			\end{definition}
			
We need an appropriate analogue of the classical notion of a basis of $L^2(\mu)$. Following \cite{FH18}, we say that a countable family of functions $(\chi_j)_{j\in\N}$ in $L^2(\mu)$ is a \emph{relative orthonormal basis} (with respect to $\mathcal{I}(T)$) if
\begin{enumerate}
	\item $\E(|\chi_j|^2|\mathcal{I(T)})$ has value 0 or 1 $\mu$-a.e. for every $j\in\N$;\\
	\item $\E(\chi_i\cdot \bar{\chi}_j|\mathcal{I(T)})=0$ $\mu$-a.e. for all $i,j\in\N$ with $i\neq j$;\\
	\item the linear span of functions of the form $\psi \chi_j$, where $j\in\N$ and $\psi\in \CI(T)$, is dense in $\mathcal{Z}_1(T)$.
	\end{enumerate}

Given a relative orthonormal basis $(\chi_j)_{j\in\N}$ and $f\in L^2(\mu)$ we call the functions
$$
f_j(x):=\E(f\cdot \bar{\chi}_j|\mathcal{I}(T))(x)=\int f\cdot \bar{\chi}_j\, d\mu_x, \quad j\in\N,
$$
the \emph{coordinates} of $f$ in this basis. It was shown in \cite[Proposition~5.1]{FH18} that  if  $f\in L^2(\mu)$, then for $\mu$-a.e. $x\in X$, we have
$$
f=\sum_{j\in\N}f_j\cdot \chi_j  \quad \text{in }  L^2(\mu_x).
$$

The following fact was established in \cite[Theorem~5.2]{FH18} and explains the utility of the notion of non-ergodic eigenfunctions.
\begin{proposition}[Existence of a relative orthonormal basis]\label{P:basis}
 	Let $(X, \CX, \mu,T)$ be a system. Then the factor $\mathcal{Z}_1(T)$ admits a relative orthonormal basis of eigenfunctions.
\end{proposition}

 \begin{example}
Let $N_1, \ldots, N_k$ be distinct primes, $X_i := \Z/N_i\Z$, and for $i=1,\ldots, k$ let  $T_i\colon X_i \to X_i$ be given by $T_i x_i := x_i+1$. Consider the system $X$ that is a disjoint union of $X_1, \ldots, X_k$, endowed with the discrete $\sigma$-algebra $\CX$ and uniform probability measure $\mu$, and define $T:X\to X$ by $T(x) := \sum\limits_{i\in [k]} {\bf 1}_{X_i}(x) \cdot T_i(x_i)$, so that $T|_{X_i}=T_i$ for each $i = 1, \ldots, k$. The subsystems $X_1, \ldots, X_k$ are the ergodic components of $X$. For each $i\in [k]$, the group of eigenfunctions of $T_i$ consists of functions of the form $x_i\mapsto c_i e(a_i x_i/N_i)$ for some $a_i\in [N_i]$ and $|c_i| = 1$, and the corresponding eigenvalue is $e(a_i/N_i)$. Each non-ergodic eigenfunction of $T$ is formed by gluing together eigenfunctions of $T_1, \ldots, T_k$, with the option of replacing some of them by 0. Thus, a non-ergodic eigenfunction of $T$ takes the form
 \begin{align*}
     \chi(x) = \sum_{i \in [k]} {\bf 1}_{X_i}(x) \, c_i \, e(a_i x_i/N_i)
 \end{align*}
 for some $|c_i|\in \{0, 1\}$ and $a_i\in [N_i]$, and the associated eigenvalue
 \begin{align*}
     \lambda(x) = \sum_{i \in [k]} {\bf 1}_{X_i}(x) \, {\bf 1}_{c_i\neq 0} \, e(a_i/N_i)
 \end{align*}
 takes the value $e(a_i/N_i)$ on each $X_i$ if $c_i \neq 0$ and 0 otherwise. In particular, the set of non-ergodic eigenfunctions
 \begin{align*}
     \{\chi_{i, a_i}(x): = {\bf 1}_{X_i}(x) \, e(a_i x_i/N_i)\colon\; i\in [k], a_i\in[N_i]\}
 \end{align*}
 forms a relative orthonormal basis of $L^2(\mu)$.
 \end{example}
 \begin{example}
Endow $X := \T^2$ with the Borel $\sigma$-algebra $\CX$ and the Haar measure $\mu$, and let $T\colon\T^2\to\T^2$ be the non-ergodic transformation given by $T(x, y) := (x, y+x)$. The set $X_x := \{x\}\times X$ is an ergodic component of $T$ if $x$ is irrational, otherwise it is a union of ergodic components. On each $X_x$, the transformation $T_x := T|_{X_x}$ has  eigenfunctions of the form $y\mapsto e(ay)$ for $a\in\Z$. However,  $\CK(T)$ is trivial, meaning that $T$ has no (classical) eigenfunctions ``globally''. By contrast, the system $(X, \CX, \mu, T)$ admits a relative orthonormal basis $(e(ay))_{a\in\Z}$. This shows the robustness of the new notion of eigenvalues: even if a system admits no (classical) eigenfunctions to work with, it may still contain plenty of non-ergodic eigenfunctions to be used.

 \end{example}

The following  identities  were established in \cite[Proposition~5.1]{FH18}.
 \begin{proposition}[Properties of a relative orthonormal basis]\label{P:Properties of the relative orthonormal basis}
 Suppose that  $(X, \CX, \mu,T)$ is  a system  such that $\CZ_1(T)=\CX$ and ergodic decomposition $\mu=\int \mu_x\, d\mu(x)$.
  Let  $(\chi_j)_{j\in\N}$ be a relative orthonormal basis of eigenfunctions. Let also $f\in L^\infty(\mu)$ and $(f_j)_{j\in \N}$ be the coordinates of $f$ in this basis. Then for $\mu$-a.e. $x\in X$,  we have
  $$
 \norm{f}_{L^2(\mu_x)}^2=\sum_{j\in \N}|f_j(x)|^2 \quad \textrm{and} \quad \nnorm{f}_{2,T,\mu_x}^4=\sum_{j\in \N}|f_j(x)|^4.
 $$
 \end{proposition}

\subsection{Inverse theorem for the $\nnorm{\cdot}_2$-seminorm}
For ergodic transformations $T$, we have the estimate
\begin{align}\label{U^2 inverse ergodic}
    \nnorm{f}_{2, T}^4\leq \sup_{\chi\in\CE(T)} \Re\Big( \int f\cdot \chi\, d\mu\Big)
\end{align}
for every 1-bounded $f\in L^\infty(\mu)$ \cite[Proposition 3.1]{Fr21}, which is the strong inverse theorem for the degree 2 seminorm mentioned at the beginning of this section.
 Since $T$ is ergodic, the supremum is taken over all classical eigenfunctions. This identity played a crucial role in the arguments of \cite{Fr21}. In our setting of commuting transformations, the transformations may no longer be ergodic, and so we now formulate an appropriate analogue of \eqref{U^2 inverse ergodic} for nonergodic transformations. The appearance of nonergodic eigenfunctions in the result below is our primary motivation for using them.
 \begin{proposition}[Inverse theorem for degree 2 seminorms]\label{U^2 inverse}
 	Let $(X, \CX, \mu,T)$ be a system. Then for every 1-bounded  $f\in L^\infty(\mu)$, we have
 	\begin{align*}
 	    \nnorm{f}_{2,T}^4\leq  \sup_{\chi\in\CE(T)} \Re\Big(\int f \cdot \chi\, d\mu\Big).
 	\end{align*}
 \end{proposition}

\begin{proof}
   Let $\tilde{f}:=\E(f|\mathcal{Z}_1(T))$. It follows from  \eqref{E:seminonerg} and \eqref{E:semifactor}   that
	$$
	\nnorm{f}_{2,T}=\nnorm{\tilde{f}}_{2,T} = \left(\int \nnorm{\tilde{f}}_{2,T, \mu_x}^4 d\mu(x)\right)^\frac{1}{4}.
	$$
    Let $(\chi_j)_{j\in\N}$ be a relative orthonormal basis of nonergodic eigenfunctions for $\mathcal{Z}_1(T)$, given by Proposition~\ref{P:basis}. Then if we denote with $ \tilde{f}_j$ the coordinates of $\tilde{f}$ in this basis, we have
	\begin{align*}
	\nnorm{\tilde{f}}_{2,T}^4= \int \nnorm{\tilde{f}}_{2,T, \mu_x}^4 d\mu(x) = \int \sum_{j\in \N}|\tilde{f}_j(x)|^4\, d\mu(x) = \sum_{j\in \N}\int|\tilde{f}_j(x)|^4\, d\mu(x),
	\end{align*}
	where the first identity follows from \eqref{E:seminonerg}, the second from Proposition~\ref{P:Properties of the relative orthonormal basis}, and the third from the monotone convergence theorem.
	Let $\delta := \nnorm{\tilde{f}}_{2,T}$. For every $\varepsilon>0$ there exists $M\in\N$ for which
	\begin{align*}
	    \delta^4-\varepsilon &\leq \sum_{j\in[M]}\int|\tilde{f}_j(x)|^4\, d\mu(x) \leq \int\max_{j\in [M]} |\tilde{f}_j(x)|^2 \cdot \sum_{j\in [M]}|\tilde{f}_j(x)|^2\, d\mu(x).
	\end{align*}
    Extending the sum over $[M]$ to all of $\N$ by positivity and applying Proposition \ref{P:Properties of the relative orthonormal basis}, we deduce that
	\begin{align*}
	    \delta^4-\varepsilon &\leq \int\max_{j\in [M]} |\tilde{f}_j(x)|^2 \cdot \sum_{j\in\N}|\tilde{f}_j(x)|^2\, d\mu(x) = \int\max_{j\in [M]} |\tilde{f}_j(x)|^2 \cdot \|\tilde{f}\|^2_{L^2(\mu_x)}\, d\mu(x),
	\end{align*}
where we used Proposition~\ref{P:Properties of the relative orthonormal basis} to get the last identity.
	The assumption $\norm{f}_{L^\infty(\mu)}\leq 1$ allows us to infer that
	\begin{align*}
	    \delta^4-\varepsilon\leq \int\max_{j\in [M]} |\tilde{f}_j(x)|\, d\mu(x).
	\end{align*}
	
	For $k\in [M]$, let $E_k$ be the set of $x\in X$ for which $|\tilde{f}_k(x)| = \max_{j\in [M]} |\tilde{f}_j(x)|$. The sets $E_k$ are $T$-invariant since the functions $\tilde{f}_j$, $j\in \N$,  are,  and we have
	\begin{align*}
	    \delta^4-\varepsilon\leq \int \sum_{j\in [M]}{\bf 1}_{E_j}(x)\cdot |\tilde{f}_j(x)|\, d\mu(x).
	\end{align*}

	We further define $\phi_j(x) := \frac{|\tilde{f}_j(x)|}{\tilde{f}_j(x)}= \frac{\left|\E(f\cdot \bar{\chi}_j|\CI(T))(x)\right|}{\E(f\cdot \bar{\chi}_j|\CI(T))(x)}$ when $\tilde{f}_j(x) = \E(f\cdot \bar{\chi}_j|\CI(T))(x)\neq 0$ and 1 otherwise. The functions $\phi_j$ are $T$-invariant and have modulus 1. Observing that
	$$
	\tilde{f}_j(x)=\E(\tilde{f}\cdot \bar{\chi}_j|\CI(T))(x) = \E(f\cdot \bar{\chi}_j|\CI(T))(x)
	$$
	for all $j\in[M]$ and $\mu$-a.e. $x\in X$ (which holds since $f-\tilde{f}$ is orthogonal to $\CE(T)$),
	we deduce that
	\begin{align*}
	    |\tilde{f}_j(x)| = |\E(f\cdot \bar{\chi}_j|\CI(T))(x)| = \E(f\cdot \bar{\chi}_j|\CI(T))(x) \cdot \phi_j(x) \quad  \mu\text{-a.e.}.
	\end{align*}
	It follows that
	\begin{align*}
	    \delta^4-\varepsilon &\leq \int \sum_{j\in [M]}\, {\bf 1}_{E_j}\cdot  \E(f\cdot \bar{\chi}_j|\CI(T)) \cdot \phi_j\,  d\mu = \int \sum_{j\in [M]}\, {\bf 1}_{E_j}\cdot  f\cdot \bar{\chi}_j \cdot \phi_j\,  d\mu\\
	    &= \int f\cdot  \sum_{j\in [M]}\, {\bf 1}_{E_j}\cdot \bar{\chi}_j \cdot \phi_j\,  d\mu,
	\end{align*}
	where the first equality is the consequence of the $T$-invariance of $E_j$ and $\phi_j$. We observe that $\chi := \sum\limits_{j\in [M]}{\bf 1}_{E_j} \, \bar{\chi}_j\,  \phi_j$ is a nonergodic eigenfunction of $T$ with eigenvalue $\lambda := \sum\limits_{j\in [M]}{\bf 1}_{E_j}\,  \bar{\lambda}_j$, where $\lambda_j$ is the eigenvalue of $\chi_j$ for $j\in [M]$. The result follows upon noticing that $\int f\cdot \chi\, d\mu$ is real and taking $\varepsilon\to 0$.
\end{proof}

\subsection{Inverse theorem for the $\nnorm{\cdot}_s$-seminorm}
Combining Proposition \ref{U^2 inverse} and  \eqref{E:seminorm4},
we get  that for every $s\in\N$ and $f\in L^\infty(\mu)$, we have
\begin{align}\label{inductive U^2 inverse}
    \nnorm{f}_{s+1, T}^{2^{s+1}}\leq \liminf_{H\to\infty}\E_{\uh\in[H]^{s-1}} \sup_{\chi_{\uh}\in\CE(T)}\Re\Big(\int \Delta_{s-1, T; \uh} f\cdot \chi_\uh\, d\mu\Big).
\end{align}
In the proofs of Theorem~\ref{T:main2} and \ref{T:local-main2}, which are needed for our applications to ergodic averages with polynomial sequences, we need a stronger  version of \eqref{inductive U^2 inverse} in which we have some control over the range of the possible rational eigenvalues of eigenfunctions appearing in \eqref{inductive U^2 inverse}. Before we state this version, we present some definitions and context.

	Let $(X, \CX, \mu, T)$ be a system for which $\CZ_s(T) = \CX$, or equivalently  the seminorm $\nnorm{\cdot}_{s+1,T}$ is a norm on $L^\infty(\mu)$ \cite[Theorem~15, Chapter~9]{HK18}.  If this property holds, we call $(X, \CX, \mu, T)$ a {\em system of order $s$}. Define also
	$$
	\spec(T):=\{t\in [0,1)\colon Tf=e(t)f\text{ for some } f\in L^2(\mu) \text{ with } f\neq 0\},
	$$
	and for any $A\subset \T$,  let $\spec_A(T):=\spec(T)\cap A$. We say that $(X, \CX, \mu, T)$ has {\em finite rational spectrum} if $\Krat(T)=\CK_r(T)$ for some $r\in \N$, or equivalently, if $\spec(T)$ is a finite subset of the rationals in $[0,1)$.
	
	Lastly, we say that a system $(X,\CX,\mu,T)$ is an {\em inverse limit of systems with finite rational spectrum}, if there exists a sequence $\CX_n$, $n\in\N$, of $T$-invariant sub-$\sigma$-algebras of $\CX$, and $r_n\in \N$,  such that $\bigvee_{n\in\N} \CX_n=\CX$ and $\E(f|\Krat)=\E(f|\CK_{r_n})$ for every $f\in L^2(\CX_n,\mu)$, $n\in\N$.  Note that then the subalgebras $L^\infty(\CX_n,\mu)$ are conjugation closed and for every $f\in L^2(\mu)$ and $\varepsilon>0$, there exist $n\in \N$ and $f_0\in L^\infty(\CX_n,\mu)$ such that
	$\norm{f-f_0}_{L^2(\mu)}\leq \varepsilon$.
	
 We will make essential use of  the following result of Host and Kra, which is a consequence of the main  result from~\cite{HK05a}.
	\begin{theorem}[Host-Kra~\cite{HK05a}]\label{T:HostKra}
		Let $(X, \CX, \mu,T)$ be an ergodic system of order $s$ for some $s\in \N$. Then  $(X, \CX, \mu,T)$ is an inverse limit of systems with finite rational spectrum. 	
	\end{theorem}
\begin{remark}
	Note that in \cite{HK05a}, it is additionally shown that a system of order $s$ is an inverse limit of nilsystems, but we will not need this extra input.
\end{remark}
Our goal is to prove a variant of the estimate \eqref{inductive U^2 inverse} for ergodic systems in which the eigenfunctions have a finite rational spectrum; this will then
be used in the degree lowering argument in Section~\ref{S:degree lowering}.
	\begin{proposition}\label{P:finiterational}
 	Let $(X, \CX, \mu,T)$ be an ergodic system (not necessarily of finite order) and  $f\in L^\infty(\mu)$ be a 1-bounded function  such that  $\nnorm{f}_{s+1,T}	>0$ for some $s\in\N$. Then there exist $r\in\N$ and eigenfunctions $(\chi_\uh)_{\uh \in \N^{s-1}}\subseteq \CE(T)$, such that
		\begin{equation}\label{E:11}
		\liminf_{H\to \infty}\E_{\uh\in [H]^{s-1}}\int \Delta_{s-1, T; \uh} f\cdot \chi_\uh\, d\mu>0,
		\end{equation}
		and
		\begin{equation}\label{E:22}
		\E(\chi|\CK_{rat})=\E(\chi| \CK_{r})
		\end{equation}
		for all $\chi$ of the form $\chi=\psi_1\cdots \psi_k$, where $k\in \N$ and $\psi_1,\ldots, \psi_k\in \{\chi_\uh,\bar{\chi}_\uh,\uh \in \N^{s-1}  \}$.
	\end{proposition}
	\begin{remark}
		The important point is that $r$ does not depend on $k\in \N$ and the choice of $\psi_1,\ldots, \psi_k$ (on the other hand,  $r$ has to depend on the function $f$).  Note also that we do not and cannot claim that \eqref{E:22} holds for every $\chi\in \CE(T)$, we can only do this if the system  $(X, \CX, \mu,T)$ is of finite order.
		
		For $s=1$, the conclusion states  that there exists $r\in \N$ such that
		$\int f\cdot \chi\, d\mu>0$ for some $\chi\in \CE(T)$ with 	$\E(\chi^k|\CK_{rat})=\E(\chi^k| \CK_{r})$ for all $k\in \N$.
	\end{remark}

	\begin{proof}
		 Suppose first that the system has  finite order. By Theorem~\ref{T:HostKra}, it is an inverse limit of systems with finite rational spectrum, meaning that there exists a sequence of $T$-invariant sub-$\sigma$-algebras $(\CX_n)_{n\in\N}$, such that $\CX=\bigvee_{n=1}^\infty \CX_n$ and each system $(X,\CX_n,\mu,T)$ has finite rational spectrum. Let  $\nnorm{f}_{s+1,T}\geq a$ for some $a\in (0,1)$.
		Then letting $\CF:=L^\infty(X, \CX_n,\mu)$, for appropriate $n\in \N$, we get that there exist $r\in \N$ such that
		\begin{equation}\label{E:ratr}
		\E(g|\Krat)=\E(g|\CK_r) \   \text{ for every } \ g\in \CF,
		\end{equation}
		 and such that there exists $f_0\in \CF\cap L^\infty(\mu)$ with
		\begin{equation}\label{E:ff0}
		\norm{f-f_0}_{L^{2^{s+1}}(\mu)}< \frac{1}{2^{s-1}} (a/2)^{2^{s+1}}.
	\end{equation}
	Without loss of generality, we can assume that $f_0$ is 1-bounded since $f$ is. From the bound $\nnorm{\cdot}_{s,T}\leq \norm{\cdot}_{L^{2^s}(\mu)}$ and \eqref{E:ff0} (note that the upper bound is $<a/2$) we deduce that 		  $\nnorm{f_0}_{s+1,T}> a/2$. The estimate \eqref{inductive U^2 inverse} implies that there exist eigenfunctions $(\chi_\uh)_{\uh\in\N^{s-1}}\subseteq \CE(T)\cap \CF$ for which
$$
		 	\liminf_{H\to \infty}\E_{\uh\in [H]^{s-1}}\int \Delta_{s-1, T; \uh} f_0\cdot \chi_\uh\, d\mu\geq (a/2)^{2^{s+1}}.
		$$
		By the telescoping identity, the estimate \eqref{E:ff0} and 1-boundedness of all the functions, we have for every $\uh\in \N^{s-1}$ that
		\begin{align*}
		    \abs{\int\Delta_{s-1, T; \uh}f_0 \cdot \chi_\uh\, d\mu - \int\Delta_{s-1, T; \uh}f \cdot \chi_\uh\, d\mu}\leq 2^{s-1}\norm{f_0-f}_{L^1(\mu)}< (a/2)^{2^{s+1}}.
		\end{align*}
		We deduce from this that
			 $$
		\liminf_{H\to \infty}\E_{\uh\in [H]^{s-1}}\int \Delta_{s-1, T; \uh} f\cdot \chi_\uh\, d\mu>0.
		$$
		Moreover, since $\chi_\uh\in  \CF$ and the subalgebra  $\CF$ is closed under  conjugation and satisfies \eqref{E:ratr},  we have that \eqref{E:22} holds for all functions $\chi$ of the needed form.
		
We consider now the general ergodic system.			Since   $\nnorm{f}_{s+1,T}= \nnorm{\tilde{f}}_{s+1,T}$, where $\tilde{f}:=\E(f|\CZ_s)$,
		we get that $\nnorm{\tilde{f}}_{s+1,T}	>0$.
				Since the factor $(X,\CZ_s, \mu,T)$ is a finite order system, we deduce from the previous case applied to this system that  there exist $r\in\N$ and $(\chi_\uh)_{\uh \in \N^{s-1}}\subseteq \CE(T)$ such that
				\begin{equation}\label{E:11'}
					\liminf_{H\to \infty}\E_{\uh\in [H]^{s-1}}\int \Delta_{s-1, T; \uh} \tilde{f}\cdot \chi_\uh\, d\mu>0,
				\end{equation}
				and
				$$
					\E(\chi|(\CK_{rat}\cap \CZ_s))=\E(\chi| (\CK_r\cap \CZ_s))
				$$
				 holds	for all $\chi$ of the form $\chi=\psi_1\cdots \psi_k$ where $k\in \N$ and $\psi_1,\ldots, \psi_k\in \{\chi_\uh,\bar{\chi}_\uh,\uh \in \N^{s-1}  \}$.
			 Lastly, since the limit in  \eqref{E:11'} remains unchanged if  we  replace $\tilde{f}$ with $f$ and  since  $\Krat\subset \CZ_s$ for every $s\in \N$, the claims \eqref{E:11} and \eqref{E:22} are satisfied.
	\end{proof}

\subsection{Relationship between global and local good properties}\label{SS:globallocal}
Our last goal is to prove  Lemma~\ref{L:necsuf} that links global and local versions of the good properties used in the statements of our main results.		

We start with a convenient characterisation of the projection $\E(\chi|\Krat(T))$ of $\chi\in\CE(T)$ and its orthogonal complement $\chi- \E(\chi|\Krat(T))$, which will also be used elsewhere. 	Let $\chi$ be a nonergodic eigenfunction of $T$ with eigenvalue $\lambda$.
For each $q\in\R$ and $A\subseteq \R$, let
$$
E_q := \{x\in X\colon \lambda(x) = e(q)\} \quad \textrm{and} \quad E_A := \bigcup_{q\in A}E_q.
$$
 Since the sets $E_A$ are unions of level sets of a $T$-invariant function, they are themselves $T$-invariant. Consequently, if for a Borel subset $A$ of $\R$ we set
$$
\chi_A := \chi\cdot {\bf 1}_{E_A}\quad \textrm{and} \quad \lambda_A := \lambda\cdot {\bf 1}_{E_A},
$$
then $\chi_A$ is an eigenfunction with eigenvalue $\lambda_A$. With these definitions, we can relate the eigenfunctions $\E(\chi|\Krat(T))$ and $\chi - \E(\chi|\Krat(T))$ to $\chi_\Q$ and $\chi_{\R\setminus\Q}$.
\begin{lemma}\label{rational part of eigenfunctions}
    Let $(X, \CX, \mu, T)$ be a system and $\chi\in\CE(T)$. Then $\chi_\Q = \E(\chi|\Krat(T))$ and $\chi_{\R\setminus\Q}=\chi - \E(\chi|\Krat(T)) $. In particular, both  $\E(\chi|\Krat(T))$ and $\chi - \E(\chi|\Krat(T))$ are nonergodic eigenfunctions of $T$.
\end{lemma}

\begin{proof}
We prove the lemma by showing that $\chi_\Q$ is in $\Krat(T)$ while $\chi_{\R\setminus\Q}$ is orthogonal to $\Krat(T)$. These two claims and the identity $\chi = \chi_\Q + \chi_{\R\setminus\Q}$ imply the lemma.

For $q\in\Q$, we pick $r\in\N$ such that $r q\in\Z$. Then $T^r({\bf 1}_{E_q}\cdot \chi) = {\bf 1}_{E_q}\cdot e(q)^r\cdot \chi = {\bf 1}_{E_q}\cdot \chi$, and so the function ${\bf 1}_{E_q}\cdot \chi$ is $T^r$-invariant. It follows that $\chi_\Q = \sum_{q\in \Q\cap[0,1)} {\bf 1}_{E_q}\cdot\chi$ is in $\Krat(T)$.

We proceed to the second part; we shall show that if $\tilde{\chi}\in \CE(T)\cap \CI(T^r)$ for some $r>0$, then $\int \chi_{\R\setminus\Q} \cdot \tilde{\chi}\, d\mu = 0$, and hence $\E(\chi_{\R\setminus\Q}|\Krat(T)) = 0$. Let $\lambda$ be the eigenvalue of $\chi$, so that $\lambda_{\R\setminus\Q} = \lambda \cdot {\bf 1}_{E_{\R\setminus\Q}}$ is the eigenvalue of $\chi_{\R\setminus\Q}$. Importantly, for $\mu$-a.e. $x\in X$, the value $\lambda_{\R\setminus\Q}(x)$ is either 0 or $e(q_x)$ for  $q_x\notin\Q$.

For every $n\in\N$, we have
\begin{align*}
    \int \chi_{\R\setminus\Q} \cdot \tilde{\chi}\,d\mu = \int T^{rn}\chi_{\R\setminus\Q} \cdot T^{rn}\tilde{\chi}\,d\mu = \int \lambda_{\R\setminus\Q}^{rn} \cdot \chi_{\R\setminus\Q} \cdot \tilde{\chi}\,d\mu,
\end{align*}
and so
\begin{align*}
    \int \chi_{\R\setminus\Q} \cdot \tilde{\chi}\,d\mu = \lim_{N\to\infty}\int \chi_{\R\setminus\Q} \cdot \tilde{\chi}\cdot \E_{n\in [N]} \lambda_{\R\setminus\Q}^{rn}.
\end{align*}
Evaluating the limit pointwise and using the bounded convergence theorem, we obtain
\begin{align}\label{correlation with rational eigenfunction}
    \int \chi_{\R\setminus\Q} \cdot \tilde{\chi}\,d\mu = \int \chi_{\R\setminus\Q} \cdot \tilde{\chi} \cdot {\bf 1}_{\{\lambda_{\R\setminus\Q}^r = 1\}}\, d\mu.
\end{align}
The fact that for $\mu$-a.e. $x\in X$, the value $\lambda_{\R\setminus\Q}$ is 0 or $e(q_x)$ for $q_x\notin\Q$ implies that $\lambda_{\R\setminus\Q}^r \neq 1$ holds $\mu$-almost everywhere, and so the integral \eqref{correlation with rational eigenfunction} vanishes. This completes the proof.
\end{proof}

		\begin{proof}[Proof of Lemma~\ref{L:necsuf}]
			The claim \eqref{i:semcon} is simply a restatement of the definition of being good for seminorm control. For the claims \eqref{i:equisimple}-\eqref{i:ratsimple}, the reverse implication follows easily  by defining for $j\in [\ell]$ the transformations   $T_j\colon \T\to \T$ and $f_j\colon \T \to\C$ by $T_j t:=t+t_j \pmod{1}$ and $f_j(t):=e(t)$.

			We  prove the forward implication in \eqref{i:irrsimple} (similarly for \eqref{i:equisimple}).  For $j\in[\ell]$ we decompose $\chi_j=\chi_{j,T_j,\Q}+\chi_{j,T_j,\R\setminus\Q}$ and use multilinearity to expand $\E_{n\in[N]}\prod_{j\in[\ell]} T_j^{a_j(n)}\chi_j$ into a sum of $2^\ell$ averages.
			If we substitute at least one of the functions $\chi_j$ with $\chi_{j,T_j,\R\setminus\Q}$
			in $\E_{n\in[N]}\prod_{j\in[\ell]} T_j^{a_j(n)}\chi_j$,
			then our irrational equidistribution assumption implies that  this average
			 converges pointwise $\mu$-a.e. to $0$, and using the bounded convergence theorem, it converges to $0$ in $L^2(\mu)$.  Hence, it suffices to verify \eqref{E:Krat1} when the average $\E_{n\in[N]}\prod_{j\in[\ell]} T_j^{a_j(n)}\chi_j$ is replaced with the average
			 $\E_{n\in[N]}\prod_{j\in[\ell]} T_j^{a_j(n)}\chi_{j,T_j,\Q}$. By Lemma~\ref{rational part of eigenfunctions}, we have  $\chi_{j,T_j,\Q}=\E(\chi_j|\Krat(T_j))$ for every $j\in [\ell]$, so in this case \eqref{E:Krat1} holds trivially.

			The forward implication in \eqref{i:ratsimple} follows by approximating the functions $f_j\in  \Krat(T_j)$ in $L^2(\mu)$ by linear combinations
			of eigenfunctions  in $\Krat(T_j)$ and using our rational convergence assumption.
		\end{proof}
		
		\section{Preparation for the main results}\label{S:lemmas}
\subsection{Elementary inequalities}
Throughout the paper, we use two inequalities presented below.
\begin{lemma}\label{VDC}
Let $H\in \N$, and $(v_n)_{n\in\N}$ be 1-bounded elements of a Hilbert space. Then
\begin{equation}
    \label{VDC1}
    \norm{\E_{n\in[N]}v_n}_{L^2(\mu)}^2 \leq 2\Re\brac{\E_{n, h\in[N]}{\bf 1}_{[N]}(n+h)\langle v_n, v_{n+h}\rangle} + \frac{1}{N}.
\end{equation}
\end{lemma}
\begin{proof}
We expand the square $\norm{\E_{n\in[N]}v_n}_{L^2(\mu)}^2 = \E_{n,m\in[N]}\langle v_n, v_m\rangle$, and then we partition $[N]^2$ into the sets $A_1, A_2, A_3$, which consist of the tuples $(n,m)$ satisfying  $n<m$, $n>m$, and $n=m$ respectively. The inequality follows by substituting $m = n+h$ on $A_1$ and $n=m+h$ on $A_2$ as well as using the 1-boundedness of $v_n$ on $A_3$.
\end{proof}

The following lemma allows us to pass from averages of sequences $a_{\uh-\uh'}$ to averages of sequences $a_{\uh}$.
\begin{lemma}\label{difference sequences}
Let $(a_\uh)_{\uh\in\N^s}$ be a sequence of nonnegative real numbers. Then
\begin{align*}
    \E_{\uh, \uh'\in[H]^s}\, a_{\uh-\uh'} \leq \E_{\uh\in[H]^s}\, a_{\uh}
\end{align*}
for every $H\in\N$.
\end{lemma}
\begin{proof}
We observe that
\begin{align*}
    \E_{\uh, \uh'\in[H]^s}\, a_{\uh-\uh'} = \E_{\uh\in[H]^s}\prod_{i\in [s]}\left(1-\frac{h_i}{H}\right) \cdot a_{\uh},
\end{align*}
and the claim follows from the fact that $0\leq \left(1-\frac{h}{H}\right)\leq 1$ for $h\in [H]$ and the  nonnegativity of $a_{\uh}$.
\end{proof}

\subsection{A consequence of the spectral theorem}
We record a fact used in the proof of Theorems~\ref{T:main2} and \ref{T:polies2}.

		\begin{lemma}\label{L:a1l}
			Let $a_1,\ldots,a_\ell\colon \N\to \Z$ be good  for irrational equidistribution.
			Then  for every system $(X,\CX,\mu,T_1,\ldots, T_\ell)$ equation  \eqref{E:Kratchar} holds when all but one of the functions $f_1,\ldots, f_\ell\in L^\infty(\mu)$ are rational $T_j$-eigenfunctions.
		\end{lemma}
		\begin{remark}
	The assumption is satisfied when  $a_1,\ldots,a_\ell\colon \N\to \Z$
	are linearly independent polynomials with zero constant terms.
	\end{remark}
		\begin{proof}
			Suppose that for some $m\in [\ell]$,
			the functions $f_j$ are rational $T_j$-eigenfunctions for $j\in[\ell]$, $j\neq m$.
			It suffices to show that if  $\E(f_m|\Krat(T_m))=0$, then
			$$
			\lim_{N\to\infty}\E_{n\in[N]}\,  e\Big(\sum_{j\in[\ell], j\neq m}a_j(n)t_j\Big)\cdot T_m^{a_m(n)}f_m=0
			$$
			in $L^2(\mu)$			whenever $t_j\in \spec_\Q(T_j)$.
			Using the spectral theorem,\footnote{Recall that the spectral theorem gives that if $(X,\CX,\mu,T)$ is a system and $f\in L^2(\mu)$, then there exists a positive bounded measure $\sigma$ such that $\int T^n f\cdot \bar{f}\, d\mu=\int e(nt)\, d\sigma(t)$ for every $n\in\N$. Then for any choice of $c_1,\ldots, c_N\in \C$ we have  $\norm{\sum_{n=1}^N\, c_n\cdot T^n f}_{L^2(\mu)}=\norm{\sum_{n=1}^N\, c_n\cdot e(nt)}_{L^2(\sigma(t))}$.}  we get that it suffices to show that
			\begin{equation}\label{E:specpos}
				\lim_{N\to\infty}\norm{\E_{n\in[N] }\,  e\Big(\sum_{j\in[\ell]}a_j(n)t_j\Big)}_{L^2(\sigma(t_m))}=0,
			\end{equation}
			where  $\sigma$ is the spectral measure of  $f_m$ with respect to the system $(X,\CX,\mu,T_m)$.  Since $\E(f_m|\Krat(T_m))=0$, the measure $\sigma$ does not have point masses on $\Q\cap [0,1)$. Moreover, since the collection of sequences $a_1,\ldots, a_\ell$ is good for irrational equidistribution, the following holds: if $t_m\in [0,1)\setminus \Q$, then
			$$
			\lim_{N\to\infty}\E_{n\in[N] }\,  e\Big(\sum_{j\in[\ell]}a_j(n)t_j\Big)=0
			$$
			for all $t_j\in [0,1)$, $j\in[\ell]$, $j\neq m$. Using the last two facts and the bounded convergence theorem we get that  \eqref{E:specpos} holds. This completes the proof.
		\end{proof}
		
		\subsection{Improving weak convergence to norm convergence}
		Let $(X, \CX, \mu, T)$ be a system and $\mu = \int \mu_x\ d\mu(x)$ be the ergodic decomposition with respect to $T$. In our arguments, we frequently deal with functions $f$ defined by weak limits of averages $\E_{n\in[N_k]}\, f_{n,k}$.  One of the challenges that we face is that an average $\E_{n\in[N_k]}\, f_{n,k}$ converging weakly to $f$ in the space $L^2(\mu)$ may not do so in the space $L^2(\mu_x)$. To address this (and other) issues, we develop techniques allowing us to pass from weak convergence of a sequence of averages to a norm convergence of a related sequence of averages. The starting point is the following elementary lemma from functional analysis. {Although variants of the next statement  are classical, we do not have a reference for the exact result we need, so we quickly give its simple proof.}
		
			\begin{lemma}  \label{L:verystrong} Let $\CF=\{(x_{l,k})_{k\in\N}:\ l\in \N\}$ be a  countable collection of  $1$-bounded sequences of elements on an inner product space $X$ such that the weak limit $x_l:=\lim_{k\to\infty} x_{l,k}$ exists for every $l\in \N$. Then  every sequence $N_k\to\infty$  has a subsequence $N'_k\to\infty$ (chosen independently  of $l\in \N$) such that for every subsequence $(N''_k)$ of $(N_k')$ we have
				$$
				\lim_{K\to\infty} \E_{k\in[K]}\,  x_{l,N''_k}= x_l
				$$
			strongly for every $l\in \N$.
		\end{lemma}
		\begin{remark}
		We will apply this when $X=L^2(\mu)$.
		The additional property about $N_k''$ will be  used in the proof of Proposition~\ref{P:dual replacement} and the strengthening it gives is crucially used later on Section~\ref{SS:Nk'} in order to find a  subsequence $N_k'$ that works simultaneously  for  properties~\eqref{i:conv1} and \eqref{i:conv2} there.
	\end{remark}
	\begin{proof} We can assume that $x_l=0$ for every $l\in \N$. Furthermore, using a diagonal argument we can assume that $l\in \N$ is fixed. So let $l\in\N$ be fixed.
		
		Since $\lim_{k\to\infty} \langle x_{l,N_k},x\rangle= 0$ for every  $x\in X$  we can choose  inductively a subsequence $(N'_k)$ of $(N_k)$ such that
		$$
		\max_{i\in[k-1]}	|\langle x_{l,N'_k},x_{l,N'_i}\rangle|\leq \frac{1}{2^k} \, \text{ for every } k\in \N.
		$$
		Then  for every    subsequence $(N''_k)$ of $(N_k')$, we have
		$$
		\max_{i\in[k-1]}	|\langle x_{l,N''_k},x_{l,N''_i}\rangle|\leq \frac{1}{2^k} \, \text{ for every } k\in \N.
		$$
		Using this estimate and expanding the square below, we deduce that
		$$
			\norm{\E_{k\in[K]}\,  x_{l,N''_k}}^2\leq \frac{3}{K} \, \text{ for every } K\in \N.
		$$
		The result follows.
	\end{proof}

The following consequence of the previous result shows that  rather general mean convergence results for multiple ergodic averages hold  once an extra averaging is introduced. Only  part~\eqref{i:VS2} will be used later on, part~\eqref{i:VS1} is for expository reasons only to showcase the general principle used.
		\begin{corollary}\label{C:verystrong}  Let $T_1,\ldots, T_\ell$ be  measure preserving transformations (not necessarily commuting) on a probability space
			$(X,\CX,\mu)$ and $a_1,\ldots, a_\ell\colon \N\to \Z$ be sequences.
			\begin{enumerate}
				\item\label{i:VS1} Every sequence $N_k\to \infty$ has a
				subsequence $N_k'\to\infty$ such that the following holds:  for all    $f_1,\ldots, f_\ell\in L^\infty(\mu)$, if we let
				\begin{equation}\label{E:AN}
					A_N(f_1,\ldots, f_\ell):=\E_{n\in[N]} \, T_1^{a_1(n)}f_1\cdots T_\ell^{a_\ell(n)}f_\ell
				\end{equation}
				for $N\in\N$, then 
				the limit
				$$
			\lim_{K\to\infty}	\E_{k\in [K]} 	\, A_{N_k'}(f_1,\ldots, f_\ell)
				$$
				exists in $L^2(\mu)$.
				
				\item\label{i:VS2} Let   $\CF$ be a countable collection of  functions in $L^\infty(\mu)$ and for some $m\in[\ell]$, let  $\mu=\int \mu_{m,x}\, d\mu$ be the ergodic decomposition  of $\mu$ with respect to $T_m$. Then
				every sequence $N_k\to \infty$ has a
				subsequence $N_k'\to\infty$ such that for some  $L_K\to\infty$ the following holds:  if   	$A_N(f_1,\ldots, f_\ell)$ is as in \eqref{E:AN},
				then for $\mu$-a.e. $x\in X$, the limit
				$$
			\lim_{K\to\infty}	\E_{k\in [L_K]} 	\, A_{N_k'}(f_1,\ldots, f_\ell)
				$$
				exists in $L^2(\mu_{m,x})$ for all $f_j\in \CF$, $j\in[\ell]$, $j\neq m$, and $f_m\in L^\infty(\mu_{m,x})$.
			\end{enumerate}
		\end{corollary}
	\begin{remarks}
		$\bullet$ It is important that $N_k'$ and $L_K$ are chosen independently of the functions in $\CF$ and the (uncountably many) functions $f_m\in L^\infty(\mu_{m,x})$, $x\in X$.
		
		$\bullet$ Note that in part (ii), the transformations $T_j$, $j\in[\ell]$, $j\neq m$, may  not preserve the measures $\mu_{m,x}$.		
	\end{remarks}
		\begin{proof}
			Without loss of generality we can assume that all functions involved are  $1$-bounded.
			
			We prove \eqref{i:VS1}. 	We consider a countable collection of  $1$-bounded functions $\mathcal{G}$ that is dense in $L^2(\mu)$  in the set of all measurable $1$-bounded functions.  Using weak compactness and a diagonal argument, we have that every sequence $N_k$ has a subsequence $\tilde{N}_k$ such that the weak limit $\lim_{k\to \infty} A_{\tilde{N}_k}(f_1,\ldots, f_\ell)$
			exists for all  $f_1,\ldots, f_\ell\in \mathcal{G}$. Using Lemma~\ref{L:verystrong}
			(for $x_{f_1,\ldots, f_\ell, k}:=A_{\tilde{N}_k}(f_1,\ldots, f_\ell)$, $f_1,\ldots, f_\ell\in \mathcal{G}$), we get that $\tilde{N}_k$ has a  subsequence $N_k'$ such that  
			the limit
			$
			\lim_{K\to\infty}\E_{k\in [K]} 	\, A_{N_k'}(f_1,\ldots, f_\ell)
			$
			exists in $L^2(\mu)$ for all $f_1,\ldots, f_\ell\in \mathcal{G}$. Using the denseness property of $\mathcal{G}$, the fact that $T_1,\ldots, T_\ell$ preserve the measure $\mu$, and an approximation argument, we deduce the asserted convergence.
			
			We prove \eqref{i:VS2}.
				Since the collection $\CF$ is countable, using a diagonal argument it suffices to get the conclusion for fixed $f_j$, $j\in[\ell]$, $j\neq m$. So we fix  $f_j$, $j\in[\ell]$, $j\neq m$, and let
				$A_N(f_m):=	A_N(f_1,\ldots,f_m,\ldots, f_\ell)$, $N\in\N$,  where 	$A_N(f_1,\ldots, f_\ell)$ is as in \eqref{E:AN}.  Recall that our standing assumption is that the probability space $(X,\CX,\mu)$ is regular and hence $X$ can be given the structure of a compact metric space such that $\CX$ is its Borel $\sigma$-algebra.  Let $\CF'$ be a countable dense subset of $(C(X), \norm{\cdot}_\infty)$. The fact that the measure $\mu_{m,x}$ is regular for all $x\in X$ implies that $\CF'$ is dense in $L^2(\mu_{m,x})$ for all  $x\in X$.
				
				 Using weak-compactness and a diagonal argument, we get that there exists a subsequence $\tilde{N}_k$ such that the averages $A_{\tilde{N}_k}(f_m)$ converge weakly for every  $f_m\in \CF'$. Using			
				 Lemma~\ref{L:verystrong} (for $x_{f_m, k}:=A_{\tilde{N}_k}(f_m)$, $f_m\in \CF'$)
				 we get that  there exists a subsequence $N'_k\to\infty$ of $\tilde{N}_k$ (and hence of $N_k$) such that the averages
				$\E_{k\in [K]} 	\, A_{N'_k}(f_m)$
				converge (strongly) in $L^2(\mu)$ as $K\to\infty$ for every $f_m\in \CF'$.
				Recall  that if  a sequence of $1$-bounded functions  converges in $L^2(\mu)$ to a function $F$, then some subsequence  converges to $F$ pointwise for $\mu$-a.e. $x\in X$, and hence, by the bounded convergence theorem, in $L^2(\mu_{m,x})$ for $\mu$-a.e. $x\in X$.
				Using these facts and   a diagonal argument,
				we get that there exists a subsequence $L_K\to\infty$
					such  that for $\mu$-a.e. $x\in X$, the averages $\E_{k\in [L_K]} 	\, A_{N'_k}(f_m)$
				converge in $L^2(\mu_{m,x})$ for every $f_m\in \CF'$.
				Since $\CF'$ is dense in $L^2(\mu_{m, x})$ for every $x\in X$ and $T_m$ is $\mu_{m,x}$-preserving,  we deduce that the previous convergence property holds for every  $f_m\in L^\infty(\mu_{m, x})$.\footnote{Note that we cannot carry out a similar approximation argument with respect to the other functions $f_j$, $j\neq m$,  because the transformations $T_j$, $j\neq m$,  do not necessarily  preserve the measure $\mu_{m,x}$. This is the reason why  we stick to a countable collection of functions $\CF$ when $j\neq m$.}   This completes the proof.			
			 	\end{proof}
	
\subsection{Introducing averaged functions}
An important insight from the finitary works on the polynomial Szemer\'edi theorem by Peluse and Prendiville \cite{P19a, P19b, PP19} is that while studying averages $\E_{n\in[N]} \,  T_1^{a_1(n)}f_1\cdots T_\ell^{a_\ell(n)}f_\ell$, it is useful to replace one of the functions $f_m$ by a more structured function $\tilde{f}_m$. The additional properties encoded in the function $\tilde{f}_m$ can then be used to obtain new information about the original average. The starting point in arguments which exploit this idea further is the following result; we present a variant best suited to our needs in the proofs of Theorem \ref{T:local-main2} (\eqref{E:ftilde} is needed there) and Proposition \ref{P:smoothing pairwise} (a straightforward variant including dual functions will be used).
\begin{proposition}\label{P:dual replacement}
	Let $a_1,\ldots, a_\ell\colon \N\to \Z$ be sequences, $(X, \CX, \mu,T_1,\ldots, T_\ell)$ be a system, and  $f_1,\ldots, f_\ell\in L^\infty(\mu)$ be $1$-bounded functions, such that
	\begin{equation}\label{E:positivea}
		\limsup_{N\to\infty}\norm{\E_{n\in[N]} \,  T_1^{a_1(n)}f_1\cdots T_\ell^{a_\ell(n)}f_\ell}_{L^2(\mu)}>0.
	\end{equation}
	Let also $m\in[\ell]$. There exists a sequence of integers $N_k\to\infty$  and $1$-bounded functions $g_k\in L^\infty(\mu)$  such that if
	\begin{equation}\label{E:averaged}
		A_{N_k}:=  \E_{n\in [N_k]} \, T_m^{-a_m(n)} g_k\cdot  \prod_{j\in [\ell], j\neq m}T_j^{a_j(n)}T_m^{-a_m(n)}\bar{f}_j, \quad k\in\N,
	\end{equation}
then $A_{N_k}$ converges weakly to some $1$-bounded  function, and if we set
\begin{equation}\label{E:ftildeweak}
	\tilde{f}_m:=\lim_{k\to\infty}\, A_{N_k},
\end{equation}
where the limit is a weak limit, then additionally
\begin{equation}\label{E:ftilde}
	\tilde{f}_m=\lim_{K\to\infty}\E_{k\in [K]}\, A_{N'_k} \, \text{ strongly in } L^2(\mu)\, \text{ for every subsequence } (N_k') \text{ of } (N_k),
\end{equation}
 and
 	\begin{equation}\label{E:notzero1}
 	\limsup_{N\to \infty} \norm{\E_{n\in[N]} \,   T_m^{a_m(n)}\tilde{f}_m \cdot \prod_{j\in [\ell],j\neq m}T_j^{a_j(n)}f_j}_{L^2(\mu)}>0.
 \end{equation}
\end{proposition}
\begin{proof}
		Equation~\eqref{E:positivea} implies that there exists  $M_k\to\infty$ such that  the following limit exists and we have
	$$
\lim_{k\to\infty}\norm{	\E_{n\in [M_k]} \,  T_1^{a_1(n)}f_1\cdots T_\ell^{a_\ell(n)}f_\ell}^2_{L^2(\mu)}>0.
	$$
	We expand   the square and compose with $T_m^{-a_m(n)}$ in order to get that
	\begin{equation}\label{E:f1fk}
		\lim_{k\to\infty} \int \bar{f}_m\cdot\Big(\E_{n\in [M_k]} \,  T_m^{-a_m(n)}g_k\cdot  \prod_{j\in [\ell], j\neq m}T_j^{a_j(n)}T_m^{-a_m(n)}\bar{f}_j\Big)\, d\mu>0
	\end{equation}
	for some $1$-bounded functions $g_k\in L^\infty(\mu)$, which in \eqref{E:f1fk} happen to take the form
	\begin{equation}\label{E:fk1}
			g_k=\E_{n\in [M_k]} \,  T_1^{a_1(n)}f_1\cdots T_\ell^{a_\ell(n)}f_\ell, \quad k\in\N.
	\end{equation}
	Using weak compactness, we get that there exists  a subsequence  $(M_k')$ of $(M_k)$
	{and $1$-bounded functions $g_k\in L^\infty(\mu)$ (perhaps different than those defined by \eqref{E:fk1})}, such that  the sequence
$$
	x_k := \E_{n\in [M_k']} \,  T_m^{-a_m(n)}g_k\cdot  \prod_{j\in [\ell], j\neq m}T_j^{a_j(n)}T_m^{-a_m(n)}\bar{f}_j, \quad k\in\N,
$$
	converges weakly in $L^2(\mu)$ to a function $\tilde{f}_m$. Using Lemma~\ref{L:verystrong} for the single sequence $x_k$, we get that  there exists  a subsequence $(N_k)$ of $(M_k')$ such that for every subsequence $(N_k')$ of $(N_k)$  if $A_N$ is as in \eqref{E:averaged}, then \eqref{E:ftilde} holds. Note that then  identity \eqref{E:ftildeweak} holds for weak convergence since $N_k$ is a subsequence of $M_k'$.

	Using \eqref{E:f1fk} for the subsequence $N_k$ of $M_k'$, we
	 conclude that $\int f_m\cdot \tilde{f}_m\, d\mu>0$ where  $\tilde{f}_m$ is defined as in \eqref{E:ftildeweak}.
	 Using this, the Cauchy-Schwarz inequality, the fact that $N_k'$ is a subsequence of $N_k$,  and \eqref{E:ftildeweak}, we deduce
	$$
	0<\int |\tilde{f}_m|^2\, d\mu =\lim_{k\to\infty}\E_{n\in [N'_k]}\,  \int \tilde{f}_m \cdot  T_m^{-a_m(n)}\bar{g}_k\cdot  \prod_{j\in [\ell], j\neq m}T_j^{a_j(n)}T_m^{-a_m(n)}f_j\, d\mu
	$$
{for some	$1$-bounded functions $g_k\in L^\infty(\mu)$.}
	Composing with $T_m^{a_m(n)}$, we get
	$$
	\lim_{k\to\infty}\int  \Big(\E_{n\in [N'_k]}\,  T_m^{a_m(n)}\tilde{f}_m \cdot \prod_{j\in [\ell], j\neq m}T_j^{a_j(n)}f_j\Big) \cdot \bar{g}_k\, d\mu>0.
	$$
	Using the Cauchy-Schwarz inequality again, we deduce that
	$$
	\limsup_{k\to\infty}\norm{\E_{n\in [N'_k]}\,  T_m^{a_m(n)}\tilde{f}_m \cdot \prod_{j\in [\ell], j\neq m}T_j^{a_j(n)} f_j }_{L^2(\mu)}>0
	$$
	as required. This completes the proof.
\end{proof}

\subsection{Dual-difference interchange}
Throughout the article, we use on numerous occasions a dual-difference interchange argument presented below. At various times, we need different forms of this argument. They follow from the same proof, therefore we present them in a single statement below.

\begin{proposition}[Dual-difference interchange]\label{dual-difference interchange}
	Let $(X, \CX, \mu,T_1, \ldots, T_\ell)$ be a system,  $s,$ $s'\in \N$, $\b_1, \ldots, \b_{s+1}\in\Z^\ell$ be vectors,
	$(f_{n,k})_{n,k\in\N}\subseteq L^\infty(\mu)$ be 1-bounded, and $f\in L^\infty(\mu)$ be defined by
	\begin{align*}
        f:=\, \lim_{k\to\infty}\E_{n\in [N_k]}\, f_{n,k},
	\end{align*}
	for some $N_k\to\infty$, where the average is assumed to converge weakly. If
	\begin{align}\label{seminorm of conditional expectation}
	    \nnorm{f}_{\b_1, \ldots, \b_s, \b_{s+1}^{\times s'}}>0,
	\end{align}
	then
	\begin{align*}
            \liminf_{H\to\infty}\E_{\uh, \uh'\in [H]^s}\limsup_{k\to\infty} \E_{n\in[N_k]} \int\Delta_{\b_1, \ldots, \b_s; \uh-\uh'}f_{n,k} \cdot u_{\uh, \uh'} \, d\mu >0
    \end{align*}
    for 1-bounded functions $u_{\uh, \uh'}\in L^\infty(\mu)$ with the following properties:
	\begin{enumerate}
	    \item\label{it: ddi 1} if $s'=1$, then $u_{\uh, \uh'}$ are invariant under $T^{\b_{s+1}}$;
        \item\label{it: ddi 2} if $s'=2$, then we can take $u_{\uh, \uh'}$ to be nonergodic eigenfunctions of $T^{\b_{s+1}}$;

        or
        \item\label{it: ddi s'} for any $s' \in\N$, we can alternatively take  $u_{\uh, \uh'}:= \prod_{j\in[2^{s'}]}\CD_{j,\uh,\uh'}$ for 1-bounded dual functions $\CD_{j, \uh, \uh'}$ of $T^{\b_{s+1}}$ of level $s'$.
	\end{enumerate}
\end{proposition}
We note two alternative  conclusions in the cases $s'=1, 2$ coming from the fact that for different purposes, we might be interested in different versions of Proposition \ref{dual-difference interchange}, and that the inverse results for seminorms of degree 1 and 2 can be expressed in two different forms. For example, we use both the cases \eqref{it: ddi 1} and \eqref{it: ddi s'} of Proposition \ref{dual-difference interchange} in the proof of Proposition \ref{P:smoothing pairwise} as well as its special case Proposition \ref{smoothing of n, n^2, n^2+n}. We do not use the case \eqref{it: ddi 2} as stated explicitly; however, in Section \ref{SS:Pm-1} we carry out a variant of this argument for functions of the form $f:=\, \lim_{K\to\infty}\E_{k\in[K]}\E_{n\in [N_k]}\, f_{n,k}$, in which the eigenfunctions $u_{\uh, \uh'}$ are shown to satisfy extra properties.

For the proof of Proposition \ref{dual-difference interchange}, we need the following version of the Gowers-Cauchy-Schwarz inequality. Its proof proceeds identically to the single-transformation version given in \cite[Lemma~4.3]{Fr21}.
\begin{lemma}[Twisted Gowers-Cauchy-Schwarz inequality]\label{L:GCS} Let $(X, \CX, \mu,T_1, \ldots, T_\ell)$ be a system, $s\in \N$, $\b_1, \ldots, \b_s\in\Z^\ell$ be vectors, and the functions $(f_\ueps)_{\ueps\in\{0,1\}^s}, (u_\uh)_{\uh\in\N^s} \subseteq L^\infty(\mu)$ be 1-bounded. Then for every $H\in \N$, we have
\begin{align*}
&\Big|\E_{\uh\in [H]^s}\,  \int \prod_{\ueps\in \{0,1\}^s} T^{\b_1 \eps_1 h_1+\cdots+\b_s \eps_s h_s} f_\ueps\cdot u_{\uh}\, d\mu\Big|^{2^s}\\
&\leq \E_{\uh,\uh'\in [H]^s}\,  \int \Delta_{\b_1, \ldots, \b_s; \uh-\uh'}f_{\underline{1}} \cdot T^{-(\b_1 h_1'+\cdots +\b_s h'_s)}\Big(\prod_{\ueps\in \{0,1\}^s}\mathcal{C}^{|\ueps|}u_{\uh^\ueps}\Big)\, d\mu.
\end{align*}
\end{lemma}

\begin{proof}[Proof of Proposition \ref{dual-difference interchange}]
Using the assumption \eqref{seminorm of conditional expectation}, we use respectively the identities \eqref{invariant inverse},  \eqref{inductive U^2 inverse}, and \eqref{dual inverse}, to deduce that
\begin{align*}
\lim_{H\to\infty}\E_{\uh\in [H]^s} \int  \Delta_{{\b_1}, \ldots, {\b_{s}}; \uh}f \cdot u_\uh\, d\mu > 0
\end{align*}
for 1-bounded functions $u_\uh$, $\uh\in \N^s$, which we can take to be:
\begin{itemize}
    \item $u_\uh := \overline{\E(\Delta_{{\b_1}, \ldots, {\b_{s}}; \uh}f|\CI(T^{\b_{s+1}}))}$ if $s' = 1$;
    \item $u_\uh\in\CE(T^{\b_{s+1}})$ if $s'=2$;

    or
    \item $u_\uh := \CD_{s', T^{\b_{s+1}}} (\Delta_{{\b_1}, \ldots, {\b_{s}}; \uh^\ueps}f)$ for $s'\in\N$.
\end{itemize}

From the fact that  $f=\lim_{k\to\infty}\E_{n\in [N_k]}\, f_{n,k}$, where the averages converge weakly, we deduce that the following expression is positive
	$$
    \liminf_{H\to\infty} \lim_{k\to\infty}\E_{n\in [N_k]}\, \E_{\uh\in  [H]^s}\,  \int \prod_{\ueps\in \{0,1\}^s\setminus \{\underline{1}\}} T^{\b_1 \eps_1 h_1+ \cdots + \b_s \eps_s h_s} \CC^{|\ueps|} f \cdot T^{\b_1 h_1+\cdots+\b_s h_s} f_{n,k}\cdot u_\uh\, d\mu.
	$$
For fixed $k,n, H\in \N$, we apply  Lemma~\ref{L:GCS} with  $f_\ueps := \CC^{|\ueps|}f$ for $\ueps\in\{0,1\}^s\setminus\{1\}$ and $f_{\underline{1}}:= f_{n,k}$, obtaining
$$
    \liminf_{H\to\infty} \limsup_{k\to\infty}\E_{n\in [N_k]}\, \E_{\uh,\uh'\in  [H]^s}\,  \int\Delta_{\b_1, \ldots, \b_s; \uh-\uh'}f_{n,k} \cdot u_{\uh, \uh'} \, d\mu > 0,
$$	
where for $\uh,\uh'\in \N^s$ we have
\begin{align*}
     u_{\uh, \uh'} := T^{-(\b_1 h_1'+\cdots+\b_s h'_s)}\Big(\prod_{\ueps\in \{0,1\}^s}\mathcal{C}^{|\ueps|}u_{\uh^\ueps}\Big).
\end{align*}
If the functions $u_\uh$ are $T^{\b_{s+1}}$-invariant, then so are
$u_{\uh, \uh'}$. If $u_\uh\in\CE(T^{\b_{s+1}})$, then $u_{\uh, \uh'}\in\CE(T^{\b_{s+1}})$ as well. Lastly, if $u_{\uh} := \CD_{s', T^{\b_{s+1}}} (\Delta_{{\b_1}, \ldots, {\b_{s}}; \uh}f)$, then
\begin{align*}
     u_{\uh, \uh'} &= \prod_{\ueps\in\{0,1\}^{s'}}\mathcal{C}^{|\ueps|} T^{-(\b_1 h_1'+\cdots+\b_s h'_s)}\CD_{s', T^{\b_{s+1}}}(\Delta_{{\b_1}, \ldots, {\b_{s}}; \uh^\ueps}f)\\
     &= \prod_{\ueps\in\{0,1\}^{s'}}\CD_{s', T^{\b_{s+1}}} \Big(\mathcal{C}^{|\ueps|} T^{-(\b_1 h_1'+\cdots+\b_s h'_s)} \Delta_{{\b_1}, \ldots, {\b_{s}}; \uh^\ueps}f\Big).
\end{align*}
The result follows from the fact that the limsup of a sum is at most the sum  of the limsups.
\end{proof}

\subsection{Removing low complexity functions}

In the final stage of the degree lowering argument in Section~\ref{S:degree lowering}, we need the following result which roughly speaking states that if a positivity property holds when  the average of $\Delta_{\b_1, \ldots, \b_s; \uh}f$ is twisted by a product of low complexity functions, then these functions can be removed. Lemma \ref{L:lower} is a generalisation of \cite[Lemma 3.4]{Fr21}; although its statement differs somewhat from the result in \cite{Fr21}, where we just assume that the vectors $\b_1 = \cdots = \b_s$ are the same and $b_{j, \uh, H}$ are sequences of numbers rather than functions, the proof requires only marginal changes so we omit it.
\begin{lemma}[Removing low-complexity functions]\label{L:lower}
Let $\ell, s\in\N$,  $(X,\CX, \mu, T_1, \ldots, T_\ell)$ be a  system,  $f\in L^\infty(\mu)$ be a function and $\b_1, \ldots, \b_s\in\Z^\ell$ be vectors.
For $j\in [s]$, $H\in\N$ and $\uh\in [H]^s$, let  $b_{j,\uh, H}\in L^\infty(\mu)$ be 1-bounded functions such that 
the sequence of functions $\uh\mapsto b_{j,\uh, H}$ does not depend on the variable $h_j$.
If
\begin{align}\label{positive limit in L:lower}
    \limsup_{N\to\infty}  	\norm{\E_{\uh\in [H]^s}\, \prod_{j\in [s]} b_{j,\uh, H}\cdot \Delta_{\b_1, \ldots, \b_s; \uh}f}_{L^2(\mu)}>0,
\end{align}
then $\nnorm{f}_{\b_1, \ldots, \b_s} > 0$.
\end{lemma}
We will apply Lemma \ref{L:lower} in Section \ref{SS:lower bound} with $\b_1 = \cdots = \b_{s-1} := \be_m$ and $\b_s := r\be_m$ for some $r\in \N$.

More generally, we could replace $\Delta_{\b_1, \ldots, \b_s; \uh}f$ in \eqref{positive limit in L:lower} by $\prod_{\ueps\in\{0,1\}^s} T^{\b_1 \eps_1 h_1 + \cdots + \b_s \eps_s h_s}f_\ueps$ for 1-bounded functions $f_\ueps$, in which case the conclusion $\nnorm{f_{\underline{1}}}_{\b_1, \ldots, \b_s}>0$ would follow. From this standpoint, Lemma \ref{L:lower} can be viewed as yet another refinement of the Gowers-Cauchy-Schwarz inequality \eqref{E:GCS}. We do not need this strengthening for our arguments, though, therefore we present the result in a simpler version that actually gets applied.

	\section{Proof of  Theorem~\ref{T:main2} and  \ref{T:local-main2} - Reduction to  a  degree lowering property}\label{S:joint ergodicity}

	In this section we prove Theorem~\ref{T:local-main2},  which
 as we remarked before, implies  Theorem~\ref{T:main2}.
 We omit the proof of  Theorem~\ref{T:local-main1} since it  is similar, in fact simpler, as it does not require the use of Proposition~\ref{P:finiterational}.

\subsection{Preparation for induction}
 In order to prove Theorem~\ref{T:local-main2},  it will be more convenient to prove the following statement.
\begin{proposition}\label{P:Main}
	Let $a_1,\ldots, a_\ell\colon \N\to \Z$ be good for seminorm control and   irrational equidistribution for the system  $(X, \CX, \mu,T_1,\ldots,T_\ell)$, and suppose that equation \eqref{E:Kratchar} holds when all but one of the functions are rational $T_j$-eigenfunctions.
	Then for every $m\in \{0,1,\ldots, \ell\}$, the following property holds:

	$(P_m)$ If   $(X, \CX, \mu, T_1,\ldots, T_\ell)$ is a system and  $f_1,\ldots,f_\ell\in L^\infty(\mu)$ are  such that   $f_j\in \CE(T_j)$ for $j=m+1,\ldots, \ell$, then
	\begin{equation}\label{E:zeroo}
		\lim_{N\to\infty}\E_{n\in[N]}\,   T_1^{a_1(n)}f_1\cdots T_\ell^{a_\ell(n)}f_\ell=0
	\end{equation}
 in  $L^2(\mu)$  if $\E(f_j|\Krat(T_j))=0$ for some $j\in [\ell]$.
\end{proposition}
Note that the   implication (i)$\implies$(ii) of Theorem~\ref{T:local-main2} follows immediately  from the case  $m=\ell$. In
what follows, we are going to show that property $(P_m)$ holds
by (finite) induction on $m\in \{0,1,\ldots, \ell\}$.
Since the collection of sequences $a_1,\ldots, a_\ell$ is good for irrational equidistribution, we get by
 Lemma~\ref{L:necsuf}~\eqref{i:irrsimple}  that  property $(P_0)$ holds.

So we fix $m\in [\ell]$ and assume that property $(P_{m-1})$ holds.
We are going to show that property $(P_m)$ holds. We briefly sketch how $(P_m)$ is derived from $(P_{m-1})$ before delving into the full proof. Suppose that \eqref{E:zeroo} fails. Then it also fails when $f_m$ is replaced by a more structured function $\tilde{f}_m$ defined as in \eqref{E:ftildeweak}. The assumption that the sequences $a_1, \ldots, a_m$ are good for seminorm control implies that $\nnorm{\tilde{f}_m}_{s+1, T_m}>0$ for some $s\in\N$ (Proposition \ref{P:fmreplace}). Then the degree lowering argument shows that $\nnorm{\tilde{f}_m}_{s, T_m}>0$ as long as $s\geq 2$. This result, stated as Proposition \ref{P:degreelowering}, is the most technically engaging part of the proof, and its proof is carried out in detail in Section \ref{S:degree lowering}. Iterating, we deduce that $\nnorm{\tilde{f}_m}_{2, T_m}>0$. We then derive the property $(P_m)$ using this fact, the definition of $\tilde{f}_m$ and the property $(P_{m-1})$.

\subsection{Property $(P_{m-1})$ for the ergodic components}\label{S:degreelower}
We will deduce from property $(P_{m-1})$ and
 Corollary~\ref{C:verystrong}~(ii) the following result that is better suited for our purposes since it provides information for the ergodic components of the system (with respect to the transformation $T_m$). It will be used in Section \ref{SS:Nk'} while carrying out the degree lowering argument.
\begin{proposition}\label{P:Hx}
 Let $(X, \CX, \mu,T_1,\ldots, T_\ell)$ be a system and for some $m\in[\ell]$ let   $\mu=\int \mu_{m,x}\, d\mu$ be  the ergodic decomposition with respect to the transformation $T_m$.  Let also $a_1,\ldots,a_\ell\colon \N\to \Z$ be sequences  and $ \CF$ be a countable subset of $L^\infty(\mu)$.   Then every sequence $N_k\to\infty$ has a subsequence $N_k'\to\infty$ such that for some $L_K\to\infty$,
 the following properties hold:
 \begin{enumerate}
\item\label{i:P1}	For    $\mu$-a.e. $x\in X$,   the limit
	\begin{equation}\label{E:f123}
		\lim_{K\to\infty}\E_{k\in [L_K]} \E_{n\in[N_k']} \,   T_1^{a_1(n)}f_1\cdots  T_\ell^{a_\ell(n)}f_\ell
	\end{equation}
	exists in $L^2(\mu_{m,x})$ for all  $f_j\in \CF$, $j\in [\ell]$, $j\neq m$, and $f_m\in L^\infty(\mu_{m,x})$.
	
	\item \label{i:P2}	If furthermore property $(P_{m-1})$ of Proposition~\ref{P:Main}
	holds, then  for $\mu$-a.e. $x\in X$, the limit in \eqref{E:f123}  is  $0$  in $L^2(\mu_{m,x})$ for all  $f_j\in \CF$, $j\in [\ell]$, $j\neq m$, with
	$f_j\in \CE(T_j)$ for $j=m+1,\ldots, \ell$,  and  $f_m\in \CE(T_m,\mu_{m,x})$ that satisfies  $\E(f_m|\Krat(T_m,\mu_{m,x}))=0$.
	\end{enumerate}
\end{proposition}
\begin{remark}
	It is important that the subsequence $N_k'$  and $L_K$ can be  chosen independently of the functions $f_m\in \CE(T_m,\mu_{m,x})$ with $\E(f_m|\Krat(T_m,\mu_{m,x}))=0$ (as $x$ varies this is an uncountable collection of functions so we do not get this independence for free). Note also that for $j\neq m$ the transformation $T_j$  may not preserve the measures $\mu_{m,x}$ and this makes it difficult to prove this result  by working on every ergodic component separately.
	\end{remark}
\begin{proof}
	 Let $N_k\to\infty$ be a sequence.  The existence of a subsequence $N_k'\to \infty $ of $N_k$ and $L_K\to\infty$
	  so that
	 property~(i) is satisfied follows from Corollary~\ref{C:verystrong}~(ii). 	 We claim that this $N_k'$ and $L_K\to\infty$ also satisfies  property~$(ii)$.
	
Let $\{\chi_l\colon  l\in\N\}$, be a relative orthonormal basis of nonergodic eigenfunctions for $\CZ_1(\mu,T_m)$ given by Proposition \ref{P:basis}.\footnote{The orthonormality is not needed  for this  argument, only their density properties will be used.}  Then  $\chi_l\in \CE(T_m)$ and $\chi_l -\E(\chi_l|\Krat(T_m,\mu))\in \CE(T_m)$,  and this difference is orthogonal to $\Krat(T_m,\mu)$.
We deduce from Property $(P_{m-1})$   that
$$
	\lim_{N\to \infty} \norm{\E_{n\in[N]}\,  T_m^{a_m(n)}\big(\chi_l -\E(\chi_l|\Krat(T_m,\mu))\big)\cdot \prod_{j\in [\ell], j\neq m} T_j^{a_j(n)}f_j }_{L^2(\mu)}=0
$$
for every $l\in \N$ and  $f_j\in\CF$, $j\in [\ell]$, $j\neq m$,  with
$f_j\in \CE(T_j)$ for $j=m+1,\ldots, \ell$.  Since linear combinations of  functions  of the form $\{\psi\cdot \chi_l \colon \, l\in \N, \psi \in \CI(T_m)\}$  are dense in $L^2(\CZ_1(T_m,\mu))$,  using linearity we deduce that
$$
\lim_{N\to \infty} \norm{\E_{n\in[N]}\,  T_m^{a_m(n)}\big(f_m -\E(f_m|\Krat(T_m,\mu))\big)\cdot \prod_{j\in [\ell],j\neq m} T_j^{a_j(n)}f_j }_{L^2(\mu)}=0
$$
for all $f_m\in L^\infty(\CZ_1(T_m,\mu))$  and $f_j\in\CF$, $j\in [\ell]$, $j\neq m$,  with
$f_j\in \CE(T_j)$ for $j=m+1,\ldots, \ell$. By Lemma~\ref{L:dense} there exists a countable  set of functions $\CA$ in $L^\infty(\CZ_1(T_m,\mu))$ such that
  \begin{equation}\label{E:dense}
  	\text{ for  } \mu\text{-a.e. }  x\in X  \text{ the set } \CA \text{ is dense in } L^2(\CZ_1(T,\mu_{m,x})).
 \end{equation}

The mean convergence property for  $N_k'$ and $L_K\to\infty$ gives that  for $\mu$-a.e. $x\in X$, we have
 \begin{equation}\label{E:Pm-1}
	\lim_{K\to\infty}\norm{\E_{k\in [L_K]} \E_{n\in[N_k']}\,  T_m^{a_m(n)}\big(f_m -\E(f_m|\Krat(T_m,\mu))\big)\cdot \prod_{j\in [\ell], j\neq m} T_j^{a_j(n)}f_j  }_{L^2(\mu_{m,x})}=0
\end{equation}
	  for all $f_m\in \CA$ and   $f_j\in \CF$, $j\in [\ell]$, $j\neq m$,  with
	 $f_j\in \CE(T_j)$ for $j=m+1,\ldots, \ell$.
	  Recall that for every system $(X, \CX, \mu,T)$ with ergodic decomposition $\mu=\int \mu_x\, d\mu$ and $f\in L^2(\mu)$, we have
	 $\E(f|\Krat(T,\mu))=\E(f|\Krat(T,\mu_x))$ for $\mu$-a.e. $x\in X$.\footnote{One way to see this is to note that by the pointwise ergodic theorem (applied first for the system $(X, \CX, \mu,T)$ and then for the systems $(X, \CX, \mu_x,T)$) , we have for $\mu$-a.e. $x'\in X$ that
	 	$\E(f|\Krat(T,\mu))(x')=\lim_{r\to\infty}\lim_{N\to\infty}\E_{n\in[N]}f(T^{r!n}x')=
	 	\E(f|\Krat(T,\mu_x))(x')$ for $\mu$-a.e. $x\in X$.}
	 	As a consequence, in  \eqref{E:Pm-1} we can replace the measure $\mu$ with $\mu_{m,x}$ in the conditional expectation, and we get for $\mu$-a.e. $x\in X$ that
	   \begin{equation}\label{E:Pm-1'}
	  	\lim_{K\to\infty}\norm{\E_{k\in [L_K]} \E_{n\in[N_k']}\,  T_m^{a_m(n)}\big(f_m -\E(f_m|\Krat(T_m,\mu_{m,x}))\big)\cdot \prod_{j\in [\ell], j\neq m} T_j^{a_j(n)}f_j  }_{L^2(\mu_{m,x})}=0
	  \end{equation}
    for all $f_m\in \CA$ and $f_j\in \CF$, $j\in [\ell]$, $j\neq m$,  with
  $f_j\in \CE(T_j)$ for $j=m+1,\ldots, \ell$.
  Using \eqref{E:dense} and an approximation argument in $L^2(\mu_{m,x})$, we deduce that  for $\mu$-a.e. $x\in X$, we have
  \begin{equation}
  	\lim_{K\to\infty}\norm{\E_{k\in [L_K]} \E_{n\in[N_k']}\,  T_m^{a_m(n)}\big(f_m -\E(f_m|\Krat(T_m,\mu_{m,x}))\big)\cdot \prod_{j\in [\ell], j\neq m} T_j^{a_j(n)}f_j  }_{L^2(\mu_{m,x})}=0
  \end{equation}
  for all  $f_m\in L^\infty(\CZ_1(T_m,\mu_x))$ and $f_j\in \CF$, $j\in [\ell]$, $j\neq m$,  with
  $f_j\in \CE(T_j)$ for $j=m+1,\ldots, \ell$. Since $	\CE(T_m,\mu_{m,x})\subset \CZ_1(T_m,\mu_{m,x})$, this immediately implies the asserted statement.
\end{proof}

\subsection{Replacing $f_m$ with an averaged function}\label{SS:replace}

Using Proposition \ref{P:dual replacement}, we deduce the following result.
\begin{proposition}\label{P:fmreplace}
	Let $a_1,\ldots, a_\ell\colon \N\to \Z$ be sequences that are good for seminorm control   for the system $(X, \CX, \mu,T_1,\ldots, T_\ell)$. Then  there exists $s\in\N$ such that the following holds: if $m\in [\ell]$  and   \eqref{E:positivea} holds for some $f_1,\ldots, f_\ell\in L^\infty(\mu)$ with   $f_j\in \CE(T_j)$ for $j=m+1,\ldots, \ell$,  then
	\begin{equation}\label{E:notzero2}
		\nnorm{\tilde{f}_m}_{s+1,T_m}>0
	\end{equation}
	where $\tilde{f}_m$ is given by \eqref{E:ftildeweak} (as a weak limit)
	and  satisfies \eqref{E:ftilde}.
\end{proposition}

\subsection{Reduction to a degree lowering property}\label{SS:newgoal}
Our new goal is to establish the following statement:

\begin{proposition}\label{P:degreelowering}
	Let $a_1,\ldots, a_\ell\colon \N\to \Z$ be good for seminorm control and irrational equidistribution for the system $(X, \CX, \mu,T_1,\ldots, T_\ell)$.
	Suppose  that
	property $(P_{m-1})$  of Proposition~\ref{P:Main} holds for some $m\in [\ell]$ and  $\tilde{f}_m$ is given by  \eqref{E:ftildeweak} and  satisfies \eqref{E:ftilde} where all functions involved are $1$-bounded and  $f_j\in \CE(T_j)$ for $j=m+1,\ldots, \ell$.  Lastly,    suppose that $\nnorm{\tilde{f}_m}_{s+1,T_m}>0$ for some $s\geq 2$.  Then   $\nnorm{\tilde{f}_m}_{s,T_m}>0$.
\end{proposition}
Before proceeding to the proof of Proposition~\ref{P:degreelowering} we show how it can be used to prove
   Proposition~\ref{P:Main} (and hence Theorem~\ref{T:main2}).
\begin{proof}[Proof of Proposition~\ref{P:Main} assuming Proposition~\ref{P:degreelowering}]
Since  $a_1,\ldots, a_\ell$ are good for irrational equidistribution we get by
 Lemma~\ref{L:necsuf}~(ii) that property $P_0$ holds.
	So  we fix $m\in [\ell]$  and we assume that property $(P_{m-1})$ holds.
	Our goal is to show that property $(P_m)$ also holds.

Arguing by contradiction, suppose  that \eqref{E:zeroo} does not hold, or equivalently,  that 	
	\eqref{E:positivea} holds for some $f_1, \ldots, f_\ell\in L^\infty(\mu)$ where $f_j\in \CE(T_j)$ for $j=m+1,\ldots, \ell$. It suffices to show that then  $\E(f_j| \CK_{rat}(T_j))\neq 0$ for $j\in [\ell]$.

	Since the sequences $a_1,\ldots, a_\ell$ are good for seminorm control, by combining  Proposition~\ref{P:fmreplace} and   \eqref{E:positivea}, we deduce  that if $\tilde{f}_m$ is as in \eqref{E:ftildeweak}, then
	$\nnorm{\tilde{f}_m}_{s+1,T_m}>0$ for some $s\in \N$. Iterating the conclusion of Proposition~\ref{P:degreelowering} exactly $s-1$ times, we get $\nnorm{\tilde{f}_m}_{2,T_m}>0$. This implies using Proposition \ref{U^2 inverse} that
	$$
	\int \tilde{f}_m\cdot \chi_m\, d\mu>0
	$$
	for some $\chi_m\in \mathcal{E}(T_m)$. We use  the definition \eqref{E:ftildeweak} of $\tilde{f}_m$, compose with  $T_m^{a_m(n)}$,  use that the transformations $T_1,\ldots, T_\ell$ commute,  and then the  Cauchy-Schwarz inequality. We deduce  that
	$$
	\limsup_{k\to\infty}\norm{\E_{n\in[N_k]} \,T_m^{a_m(n)}\chi_m \prod_{j\in [\ell], j\neq m} T_j^{a_j(n)}f_j}_{L^2(\mu)}>0.
	$$
	Since $\chi_m\in \CE(T_m)$ and $f_j\in \CE(T_j)$, $j=m+1,\ldots, \ell$,  property $(P_{m-1})$ applies and gives  that $\E(f_j| \CK_{rat}(T_j))\neq 0$  for   $j\in [\ell]$, $j\neq m$.

		It remains to show that $\E(f_m| \Krat(T_m))\neq 0$. To do this, we first note
	that we have established the $L^2(\mu)$-identity
\eqref{E:zeroo} for all
	$f_1,\ldots, f_\ell\in L^\infty(\mu)$  that satisfy  $\E(f_j| \Krat(T_j))= 0$  for  some   $j\in [\ell]$, $j\neq m$ and $f_{j}\in \CE(T_j)$ for $j=m+1, \ldots, \ell$.
		We claim that
		under the same restrictions on $f_1,\ldots, f_\ell$, we have  the identity
	\begin{equation}\label{E:Krat}
		\lim_{N\to\infty}\norm{ \E_{n\in[N]}\, \prod_{j\in[\ell]} T_j^{a_j(n)}f_j -  \E_{n\in[N]}\,    T_m^{a_m(n)}f_m \prod_{j\in [\ell], j\neq m}  T_j^{a_j(n)}\E(f_j|\Krat(T_j))}_{L^2(\mu)}=0.
	\end{equation}
	To see this, we first  note that if $(X, \CX, \mu,T)$ is a system and $f\in \CE(T)$, then  by Lemma~\ref{rational part of eigenfunctions} we have  $f-\E(f|\Krat(T))\in  \CE(T)$ and trivially  $\E\big(f-\E(f|\Krat(T))|\Krat(T)\big)=0$. Hence,
 identity \eqref{E:zeroo} implies that \eqref{E:Krat}  holds if  for $j\in [\ell], j\neq m$, we replace at least one of the functions  $f_j\in L^\infty(\mu)$ with  $f_j-\E(f_j|\Krat(T))$.  Using this property, writing   $f_j=\E(f_j|\Krat(T_j))+(f_j-\E(f_j|\Krat(T_j)))$ for $j\neq m$, and expanding the product, we deduce  that \eqref{E:Krat} holds.
Combining 	\eqref{E:positivea} and \eqref{E:Krat} we get that \eqref{E:positivea} holds for some $f_j\in L^\infty(\mu)$
	with $f_j\in \Krat(T_j)$ for $j\in [\ell]$, $j\neq m$.

Using $L^2(\mu)$-approximation and linearity, we conclude that \eqref{E:positivea} holds for some rational $T_j$-eigenfunctions $f_j$, $j\in [\ell]$, $j\neq m$.
 Our assumption that \eqref{E:Krat} holds when all but one of the functions are $T_j$-rational eigenfunctions  implies that $\E(f_m|\Krat(T_m))\neq 0$, completing the proof.
	\end{proof}
We remark that in order  to prove Theorem~\ref{T:local-main1} (which implies Theorem~\ref{T:main1}), there is a small difference in the last step where we want to prove the result for $\ell=1$, since this case is not covered by our assumption as was the case with Theorem~\ref{T:local-main2} . The case $\ell=1$ follows from the $\ell=1$ case
of the corresponding result in \cite[Theorem~1.1]{Fr21} (and is done by a further degree lowering argument).

		\section{Proof of  Theorems~\ref{T:main2} and \ref{T:local-main2} - Proof of  the  degree lowering property}\label{S:degree lowering}

The goal of this subsection is to  complete the proof of Theorem~ \ref{T:local-main2} by proving Proposition~\ref{P:degreelowering}. The proof goes as follows. Starting with the assumption $\nnorm{\tilde{f}_m}_{s+1,T_m}> 0$, we deduce that $\nnorm{\tilde{f}_m}_{s+1,T_m, \mu_{m, x}}> 0$ on a set $A$ of positive measure. We then use Proposition \ref{P:finiterational} to deduce that for $x\in A$ and a positive lower density set of values $\uh\in\N^{s-1}$, the multiplicative derivative $\Delta_{s-1, T_m; \uh}\tilde{f}_m$ correlates with an eigenfunction $\chi_{m,\uh,x}$ whose eigenvalue is either irrational or rational with denominator bounded by $r_x$. Using a version of the dual-difference interchange argument and invoking property $(P_{m-1})$ (in the version given by Proposition \ref{P:Hx}), we deduce that for many values $\uh$, the eigenfunction  $\chi_{m,\uh,x}$ is a product of low complexity functions and a term invariant under $T_m^{r_x!}$. We then use Lemma \ref{L:lower} to conclude that $\nnorm{\tilde{f}_m}_{s, T_m, \mu_x}>0$ for $x\in A$, and the claim $\nnorm{\tilde{f}_m}_{s,T_m}> 0$ follows from the formula \eqref{E:seminonerg} and the fact that $A$ has positive measure.
\subsection{Ergodic decomposition and inverse theorem}
Let
$$
\mu=\int \mu_{m,x}\, d\mu
$$
be the ergodic decomposition of $\mu$ with respect to the transformation $T_m$.
Then the measures $ \mu_{m,x}$ are $T_m$-invariant (but not necessarily $T_j$-invariant for $j\neq m$) and the systems $(X, \CX, \mu_{m,x}, T_m)$ are ergodic for $\mu$-a.e. $x\in X$.

Since $\nnorm{\tilde{f}_m}_{s+1,T_m}> 0$,  there exists a $\mu$-measurable  $T$-invariant set $A$ with $\mu(A)>0$ such that for every $x\in A$, we have
$$
 \nnorm{\tilde{f}_m}_{s+1,T_m, \mu_{m, x}}> 0.
$$
Using Proposition~\ref{P:finiterational},
we get that for every $x\in A$, there exist $r_x=r_{\mu_{m,x}}\in \N$
and  $\chi_{m,\uh,x}=\chi_{m,\uh,\mu_{m,x}}\in \CE(T_m,\mu_{m,x})$    such that\footnote{It is important that $r_x$ does not depend on  $\uh,\uh'\in \N$ and $r_x, \chi_{m,\uh,\uh',x}$ depend only on the measure  $\mu_{m,x}$ and not on $x$ itself.}
 \begin{equation}\label{E:rx}
 \E(\chi_{m,\uh,\uh',x}| \Krat(T_m,\mu_{m,x}))= \E(	\chi_{m,\uh,\uh',x}|\CK_{r_x} (T_m,\mu_{m,x}))
 \end{equation}
for all $\uh,\uh'\in \N^{s-1}$ and  $x\in X$, where $\chi_{m,\uh,\uh',x}$  are as in \eqref{E:hh'} below and
\begin{equation}\label{E:positive1}
\liminf_{H\to\infty}\E_{\uh\in [H]^{s-1}}\int \Delta_{s-1, T_m; \uh}\tilde{f}_m\cdot \chi_{m,\uh,x}\, d\mu_{m,x}>0.\footnote{We do not claim here that the map $x\mapsto \chi_{m,h,x}$ or the map $x\mapsto r_x$ can be chosen to be  $\mu$-measurable.}
\end{equation}
We deduce that
there exist $a=a_{\mu_{m,x}}>0$ and a subset $\Lambda=\Lambda_{\mu_{m,x}}$ of $\N^{s-1}$ with positive lower density, such that
\begin{equation}\label{E:positiveLambda}
\int \Delta_{s-1, T_m; \uh}\tilde{f}_m\cdot \chi_{m,\uh,x}\, d\mu_{m,x}>a
\end{equation}
for all $\uh\in \Lambda $.

\subsection{Passing information   to the ergodic components - Choice of  $(N_k')$}\label{SS:Nk'}
Recall that we work under the assumption that property $(P_{m-1})$ of Proposition~\ref{P:Main} holds for some $m\in[\ell]$. Let $N_k$ be a sequence that defines $\tilde{f}_m$ as in \eqref{E:ftildeweak}, so that
\begin{align*}
    \tilde{f}_m=\lim_{k\to\infty}\E_{n\in[N_k]}\, f_{n,k}
\end{align*}
weakly, and
\begin{equation}\label{E:fnk}
  	f_{n,k}:=T_m^{-a_m(n)}\bar{g}_k\cdot  \prod_{j\in [\ell], j\neq m}T_j^{a_j(n)}T_m^{-a_m(n)}\bar{f}_j.
  \end{equation}
We consider the following countable collection of $1$-bounded functions in $L^\infty(\mu)$:
   $$
\CF:= \{ \Delta_{s-1, T_m; \uh} f_j:\, j\in [\ell],\ j\neq m, \ \uh\in \Z^{s-1}\}.
$$

We claim that  we can  pick a subsequence $N_k'\to \infty$ of
 $N_k$ and $L_K\to\infty$ such that the following properties hold:
  \begin{enumerate}
  		\item \label{i:conv1}	For    $\mu$-a.e. $x\in X$,  we have\footnote{In order to deduce this property weak convergence in $L^2(\mu)$ does not suffice,  we need mean convergence, which is one of the reasons why we use the extra averaging over $k$ that enables to deduce  mean convergence. There are also several other places in the argument when mean convergence is needed.}
  		\begin{equation}\label{E:tildefm}
  		\tilde{f}_m=\lim_{K\to\infty} \E_{k\in [L_K]}\E_{n\in[N_k']}\, f_{n,k} \ \text{  in } L^2(\mu_{m,x}).
  	\end{equation}

 	\item	\label{i:conv2} For    $\mu$-a.e. $x\in X$,   the limit
 	$$
 		\lim_{K\to\infty}\E_{k\in [L_K]} \E_{n\in[N_k']} \,   T_1^{a_1(n)}f'_1\cdots  T_\ell^{a_\ell(n)}f'_\ell\,    \text{	exists in } L^2(\mu_{m,x})
 	$$
 	 for all  $f'_j\in\CF$, $j\in [\ell]$, $j\neq m$, and $f'_m\in L^\infty(\mu_{m,x})$,
 	and it   is  $0$ if in addition we have
 	$f'_j\in \CE(T_j)$ for $j=m+1,\ldots, \ell$  and  $f'_m\in \CE(T_m,\mu_{m,x})$ satisfies  $\E(f'_m|\Krat(T_m,\mu_{m,x}))=0$.
 \end{enumerate}
 To see this, we  first use both parts of  Proposition~\ref{P:Hx} to pick a subsequence  $N_k'\to\infty$ of $N_k$  and $L_K\to\infty$, such that property~\eqref{i:VS2} holds. Using Proposition \ref{P:dual replacement}, and specifically the identity \eqref{E:ftilde}, we obtain \eqref{E:tildefm} but with convergence taking place in $L^2(\mu)$ rather than $L^2(\mu_{m, x})$. By passing to a subsequence $L_K'\to\infty$ of $L_K$ we get (as in the proof of   Corollary~\ref{C:verystrong}~\eqref{i:VS2}) that
 property~\eqref{i:conv1} holds (now with  convergence taking place in $L^2(\mu_{m,x})$) with $L_K'$ in place of $L_K$.
 Since $L_K'$ is a subsequence of $L_K$, then also  property~\eqref{i:conv2} holds with $L_K'$ in place of $L_K$.
 After renaming $L_K'$ as $L_K$ we get both claimed properties.

\subsection{Dual-difference interchange and use of $(P_{m-1})$}\label{SS:Pm-1}
Combining the identity
$$
\Delta_{s-1, T_m; \uh}\tilde{f}_m=\prod_{\ueps\in \{0,1\}^s}\mathcal{C}^{|\ueps|}T_m^{\ueps\cdot \underline{h}}\tilde{f}_m 
$$
with
\eqref{E:positiveLambda} and  \eqref{E:tildefm},
we deduce   that  for $\mu$-a.e. $x\in A$, we have
\begin{multline*}
\liminf_{H\to\infty} \lim_{K\to\infty}\E_{k\in[L_K]}\E_{n\in [N'_k]} \\
\E_{\uh\in  [H]^{s-1}}\,  \int \prod_{\ueps\in \{0,1\}^s\setminus \{\underline{1}\}}\mathcal{C}^{|\ueps|}T_m^{\ueps\cdot \underline{h}}\tilde{f}_m \cdot T_m^{h_1+\cdots+h_{s-1}}f_{n,k}\cdot \chi_{m,\uh,x}\cdot  {\bf 1}_\Lambda(\underline{h})\, d\mu_{m,x}>0.
\end{multline*}

We will use   Lemma~\ref{L:GCS} with $s-1$ in place of $s$, for $T:=T_m$, $\mu:=\mu_{m,x}$,
 $f_{\underline{1}}:=f_{n,k}$,   $f_\ueps:=\mathcal{C}^{|\ueps|}\tilde{f}_m $ for $\ueps\in \{0,1\}^{s-1}\setminus \underline{1}$, and
$g_{\underline{h}}:=\chi_{m,\uh,x}\cdot {\bf 1}_\Lambda(\underline{h})$, in order to upper bound the average over $\uh,\uh'$.  We deduce that\footnote{We cannot claim that the limit as $K\to\infty$ below exists because $f_{n,k}$ depend on the functions $g_k$ and we have not bothered to treat mean convergence in such cases. This is not going to create any trouble   for us since the functions $g_k$ will  be eliminated on the next step.}
	\begin{multline*}
	\liminf_{H\to\infty} \limsup_{K\to\infty}\E_{k\in [L_K]}\E_{n\in [N'_k]}\\
	\E_{\uh,\uh'\in  [H]^{s-1}}\,  \int\Delta_{s-1, T_m; \uh-\uh'}f_{n,k}\cdot T_m^{-|\uh|}\Big(\prod_{\ueps\in \{0,1\}^{s-1}}\big(\mathcal{C}^{|\ueps|}\chi_{m,\uh^\ueps,x}\cdot  {\bf 1}_\Lambda(\uh^\ueps)\big)\Big)\, d\mu_{m,x}>0;
\end{multline*}
we recall here that $\uh^\ueps := (h_1^{\eps_1}, \ldots, h_s^{\eps_s})$, where $h_j^0:=h_j$ and $h_j^1:=h_j'$ for $j\in[s]$.
Let
	 \begin{equation}\label{E:fihh'}
	 f_{j,\uh,\uh'}:=\Delta_{s-1, T_m; \uh-\uh'}f_j\quad \textrm{for}\quad  j\in [\ell],\,  j\neq m, \,  \uh,\uh'\in [H]^{s-1}.
	 \end{equation}
	 Note that since $T_j,T_m$ commute and $f_j\in \CE(T_j,\mu)$ for $j=m+1,\ldots, \ell$, we have that $f_{j,\uh,\uh'}\in \CE(T_j,\mu)$.
	 For all $\uh,\uh'\in\N^{s-1}$ and $x\in A$, let also
\begin{equation}\label{E:hh'}
	\chi_{m,\uh,\uh',x}:=\prod_{\ueps\in \{0,1\}^{s-1}}\mathcal{C}^{|\ueps|}\chi_{m,\uh^\ueps,x}\in \CE(T_m,\mu_{m,x}),\footnote{We used that $\chi\in \CE(T,\mu)$ implies $\chi\in \CE(T,\mu_x)$ for $\mu$-a.e. $x\in X$.}
\end{equation}
and  $\Lambda'=\Lambda'_{\mu_{m,x}}$ be defined by
$$
\Lambda':=\{(\uh,\uh')\in \N^{2(s-1)} \colon \uh^\ueps\in \Lambda \text{ for all } \ueps \in \{0,1\}^{s-1}\}.
$$
Using the previous facts and the form of $f_{n,k}$ from \eqref{E:fnk}, we deduce after composing with $T_m^{a_m(n)}$ and  applying  the Cauchy-Schwarz inequality that for $\mu$-a.e. $x\in A$, we have
\begin{multline*}
\liminf_{H\to\infty}\E_{\uh,\uh'\in[H]^{s-1}}{\bf 1}_{\Lambda'}(h,h')\\
\lim_{K\to\infty}\norm{\E_{k\in[L_K]}\E_{n\in [N'_k]} \, T_m^{a_m(n)}\chi_{m,\uh,\uh',x} \prod_{j\in [\ell], j\neq m} T_j^{a_j(n)}f_{j,\uh,\uh'}}_{L^2(\mu_{m,x})}>0,
\end{multline*}
where the limit $\lim_{K\to\infty}$ exists
for all $\uh,\uh'\in \N^{s-1}$ by property~\eqref{i:conv2} of Section~\ref{SS:Nk'}, the definition of the functions $f_{j,\uh,\uh'}$  in \eqref{E:fihh'},   and the choice of $\CF$.
Since ${\bf 1}_{\Lambda}(\uh,\uh')\leq {\bf 1}_\Lambda(\uh)$ and $\Lambda$ has positive lower density,
we deduce that for $\mu$-a.e. $x\in A$, we have
\begin{multline*}
\liminf_{H\to\infty}\E_{\uh'\in [H]^{s-1}} \E_{\uh\in\Lambda\cap [H]^{s-1}}\\
\lim_{K\to\infty}\norm{\E_{k\in[L_K]}\E_{n\in [N_k']} \, T_m^{a_m(n)}\chi_{m,\uh,\uh',x}  \prod_{j\in [\ell], j\neq m} T_j^{a_j(n)}f_{j,\uh,\uh'}}_{L^2(\mu_{m,x})}>0.
\end{multline*}

Using property~\eqref{i:conv2} of Section~\ref{SS:Nk'} (indirectly we use property $(P_{m-1})$ here) we get that for $\mu$-a.e. $x\in A$, the limits $\lim_{K\to\infty}$ below exist for all $\uh,\uh'\in \N^{s-1}$ and   the following limit is positive\footnote{The existence of the limits $\lim_{K\to\infty}$  is also crucial here, we would run into trouble  if we had $\limsup_{K\to\infty}$ before since    on the next step we could only apply  the result about characteristic factors  on some fixed well chosen subsequence of $L_K$.}
\begin{multline*}
\liminf_{H\to\infty}\E_{\uh'\in [H]^{s-1}} \E_{\uh\in\Lambda\cap [H]^{s-1}}\\
\lim_{K\to\infty}\norm{\E_{k\in[L_K]}\E_{n\in [N_k']} \,  T_m^{a_m(n)}\E(\chi_{m,\uh,\uh',x}|\Krat(T_m, \mu_{m,x}))  \prod_{j\in [\ell], j\neq m} T_j^{a_j(n)}f_{j,\uh,\uh'} }_{L^2(\mu_{m,x})}.
\end{multline*}
Using the pigeonhole principle, we pick for $\mu$-a.e. $x\in A$ and all $H\in\N$ an element $\uh'_H = \uh'_{H, x}\in [H]^{s-1}$ for which the following expression is positive
\begin{multline*}
	\liminf_{H\to\infty}\E_{\uh\in\Lambda\cap [H]^{s-1}}\\
	\lim_{K\to\infty}\norm{\E_{k\in[L_K]}\E_{n\in [N_k']} \,  T_m^{a_m(n)}\E(\chi_{m,\uh,\uh'_H,x}|\Krat(T_m,\mu_{m,x}))  \prod_{j\in [\ell], j\neq m} T_j^{a_j(n)}f_{i,\uh,\uh'_H} }_{L^2(\mu_{m,x})}.
\end{multline*}

Recall  that by \eqref{E:rx}, we have  $\E(\chi_{m,\uh,\uh',x}|\Krat(T_m,\mu_{m,x}))=\E(\chi_{m,\uh,\uh',x}|
\CK_{r_x}(T_m,\mu_{m,x}))$ for all $\uh,\uh'\in \N^{s-1}$ and $x\in A$. Since $\chi_{m,\uh,\uh',x}\in \CE(T_m,\mu_{m,x})$ and (ergodic) eigenfunctions with different eigenvalues are orthogonal, we deduce that $\E(\chi_{m,\uh,\uh'_H,x}|\Krat(T_m,\mu_{m,x}))=0$ unless the $T_m$-eigenfunction  $\chi_{m,\uh,\uh'_H,x}$ is $T^{r_x}$-invariant.
We  conclude from this that for $\mu$-a.e. $x\in A$ and every $H\in\N$,  there exist    
sets $\Lambda_H:=\Lambda_{H,\mu_{m,x}}\subset \Lambda\cap [H]^{s-1}$ with
\begin{equation}\label{E:LN}
	\liminf_{H\to\infty}\frac{|\Lambda_H|}{H^{s-1}}>0
\end{equation}
and such that $T_m^{r_x}\chi_{m, \uh,\uh'_H,x} = \chi_{m, \uh,\uh'_H,x}$ holds $\mu_{m, x}$-a.e. for all $\uh\in \Lambda_H$ and $H\in \N$.

\subsection{Extracting a lower bound on $\nnorm{\tilde{f}_m}_{s, T_m}$}\label{SS:lower bound}
We now combine all the statements to deduce the claim $\nnorm{\tilde{f}_m}_{s, T_m}>0$. From the definition of $\chi_{m, \uh, \uh'_H, x}$ it follows that
\begin{align}\label{decomposition of eigenfunction}
    \chi_{m, \uh, x} = \prod_{{\ueps}\in \{0,1\}^{s-1}\setminus \underline{0}}\CC^{|\ueps|+1}\chi_{m, \uh_H^\ueps, x} \cdot \chi_{m, \uh, \uh'_H, x}
\end{align}
for every $\uh\in\N^{s-1}$ and $H\in \N$, where $\uh^\ueps_H = (h^{\eps_1}_{1, H}, \ldots, h^{\eps_{s-1}}_{s-1, H})$ and
\begin{align*}
     h^{\eps_j}_{j, H} := \begin{cases} h_j,\; & \eps_j = 0\\ h'_{j,H},\; &\eps_j = 1.
     \end{cases}
\end{align*}
For each $H\in\N$, $\uh\in\N^{s-1}$, and $j\in[s-1]$, we set
\begin{align}\label{product of eigenfunctions}
    \tilde{\chi}_{j, \uh, H, x} := \prod\limits_{\substack{\ueps\in\{0,1\}^{s-1}, \eps_j = 1,\\ \eps_1 = \ldots = \eps_{j-1} = 0}}\CC^{|\ueps|+1}\chi_{m, \uh_H^\ueps, x},
\end{align}
so that
\begin{align}\label{decomposition of eigenfunction 2}
    \chi_{m, \uh, x} = \prod_{j\in[s-1]} \tilde{\chi}_{j, \uh, H, x} \cdot \chi_{m, \uh, \uh'_H, x}
\end{align}
for all $H\in\N$ and $\uh\in N^{s-1}$.
The definition \eqref{product of eigenfunctions} implies that the functions $\tilde{\chi}_{j, \uh, H, x}$ are independent of the variable $h_j$. Consequently, for $\mu$-a.e. $x\in A$ and every $\uh\in\Lambda_H$ and $H\in\N$, the eigenfunction $\chi_{m, \uh, x}$ can be expressed as a product of $s-1$ functions $\tilde{\chi}_{1, \uh, H, x}, \ldots, \tilde{\chi}_{s-1, \uh, H, x}$, each of which depends on at most $s-2$ entries of $\uh$, and the function ${\chi}_{m, \uh, \uh'_H, x}$, which is $T_m^{r_x}$-invariant $\mu_{m,x}$-almost everywhere. Hence, for many values $\uh$, the eigenfunctions $\chi_{m, \uh, x}$ are products of low complexity functions and a function invariant under $T_m^{r_x}$. This fact will prove crucial in obtaining the claimed result.

From the estimate \eqref{E:positiveLambda} and the inclusion $\Lambda_H\subseteq\Lambda$, we infer that
\begin{align*}
    \int \Delta_{s-1, T_m; \uh}\tilde{f}_m\cdot \chi_{m,\uh,x}\, d\mu_{m,x}>a
\end{align*}
for all $\uh\in\Lambda_H$ and $H\in\N$, and so
\begin{align}\label{E:liminfEigens}
    \liminf_{H\to\infty}\E_{\uh\in\Lambda_H}\int \Delta_{s-1, T_m; \uh}\tilde{f}_m\cdot \prod_{j\in[s-1]} \tilde{\chi}_{j, \uh, H, x} \cdot\chi_{m, \uh, \uh'_H, x}\, d\mu_{m,x}>0.
\end{align}
At this step, we want to use the invariance property of $\chi_{m, \uh, \uh'_H, x}$ to get rid of this function. Letting $r=r_x$ for simplicity, we compose each integral with $T_m^{rh'_s}$, introduce extra averaging over $h'_s\in[H]$, and recall that $T_m^{rh'_s} \chi_{m, \uh, \uh'_H, x}= \chi_{m, \uh, \uh'_H, x}$  for $\mu_{m,x}$-a.e. $x\in X$, so that
\begin{align}\label{E:liminfEigens2}
    \liminf_{H\to\infty}\E_{\uh\in\Lambda_H}\int \chi_{m, \uh, \uh'_H, x} \cdot  \E_{h'_s\in[H]} \Delta_{s-1, T_m; \uh}\brac{T_m^{r h'_s}\tilde{f}_m}\cdot \prod_{j\in[s-1]} T_m^{r h'_s}\tilde{\chi}_{j, \uh, H, x}\, d\mu_{m,x}>0.
\end{align}
Applying the triangle inequality inside the integral, extending the summation by positivity from $\Lambda_H$ to all of $[H]^{s-1}$ and invoking the estimate \eqref{E:LN} as well as the fact that $\Lambda$ has  positive lower density, we conclude that
\begin{align}\label{E:liminfEigens3}
    \liminf_{H\to\infty}\E_{\uh\in[H]^{s-1}}\int \Big| \E_{h'_s\in[H]} \Delta_{s-1, T_m; \uh}\brac{T_m^{r h'_s}\tilde{f}_m}\cdot \prod_{j\in[s-1]} T_m^{r h'_s}\tilde{\chi}_{j, \uh, H, x}\Big|\, d\mu_{m,x}>0.
\end{align}
We subsequently apply the Cauchy-Schwarz inequality and Lemma \ref{VDC} (specifically, its first part \eqref{VDC1}) to \eqref{E:liminfEigens3}, so that
\begin{multline*}
        \liminf_{H\to\infty}\Big|\E_{(\uh, h_s, h_s')\in[H]^{s+1}}\, {\bf 1}_{[H]}(h_s+h_s')\\
    \int \Delta_{s, T_m; (\uh, r h_s)}\brac{T_m^{r h_s'}\tilde{f}_m}\cdot \prod_{j\in[s-1]} \Delta_{1, T_m; r h_s}\brac{ T_m^{r h_s'}\tilde{\chi}_{j, \uh, H, x}}\, d\mu_{m,x}\Big|>0.
\end{multline*}
Composing each integral with $T_m^{-r h_s'}$, we arrive at the inequality
\begin{multline*}
    \liminf_{H\to\infty}\Big|\E_{(\uh, h_s, h_s')\in[H]^{s+1}}\, {\bf 1}_{[H]}(h_s+h_s')\\
    \int \Delta_{s,T_m; (\uh, r h_s)}\tilde{f}_m\cdot \prod_{j\in[s-1]} \Delta_{1, T_m; r h_s}\tilde{\chi}_{j, \uh, H, x}\, d\mu_{m,x}\Big|>0.
\end{multline*}
Using the pigeonhole principle, we pick for each $H\in\N$ an element $h'_s = h'_{s, H}\in[H]$ that maximises the average
\begin{align*}
    \Big|\E_{(\uh, h_s)\in[H]^{s}}\, {\bf 1}_{[H]}(h_s+h_s')\int \Delta_{s,T_m; (\uh, r h_s)}\tilde{f}_m\cdot \prod_{j\in[s-1]} \Delta_{1, T_m; r h_s}\tilde{\chi}_{j, \uh, H, x}\, d\mu_{m,x}\Big|
\end{align*}
so that
\begin{align}\label{E:liminfEigens4}
    \liminf_{H\to\infty}\Big|\E_{(\uh, h_s)\in[H]^{s}}\, {\bf 1}_{[H]}(h_s+h_{s,H}')\int \Delta_{s,T_m; (\uh, r h_s)}\tilde{f}_m\cdot \prod_{j\in[s-1]} \Delta_{1, T_m; r h_s}\tilde{\chi}_{j, \uh, H, x}\, d\mu_{m,x}\Big|>0.
\end{align}

Letting
\begin{align*}
    b_{j, \uh, h_s, H, x} := \begin{cases}  \Delta_{1, T_m; r h_s}\tilde{\chi}_{j, \uh, H, x} \cdot {\bf 1}_{[H]}(h_s+h'_{s, H}),\; &j=1\\
    \Delta_{1, T_m; r h_s}\tilde{\chi}_{j, \uh, H, x},\; &2\leq j\leq s-1\\
    1,\; &j =s,
    \end{cases}
\end{align*}
we can express \eqref{E:liminfEigens4} as
\begin{align}\label{E:liminfEigens5}
    \liminf_{H\to\infty}\Big|\E_{(\uh, h_s)\in[H]^{s}}\int \Delta_{s,T_m; (\uh, r h_s)}\tilde{f}_m\cdot \prod_{j\in[s]} b_{j, \uh, h_s, H, x}\, d\mu_{m,x}\Big|>0.
\end{align}
The crucial property is that each function $b_{j, \uh, h_s, H, x}$ is independent of the variable $h_j$. Noting that $\Delta_{s, T_m; (\uh, rh_s)} f = \Delta_{\be_m^{\times (s-1)}, \be_m^r; (\uh, h_s)}f$ and applying Lemmas \ref{L:lower} and \ref{L:seminorm of power} to \eqref{E:liminfEigens5}, we deduce that $\nnorm{\tilde{f}_m}_{s, T_m, \mu_{m,x}}>0$ for $\mu$-a.e. $x\in A$. Using this bound and the fact that $A$ has positive measure, we conclude from \eqref{E:seminonerg} that $\nnorm{\tilde{f}_m}_{s,T_m, \mu}>0$, as required. This completes the proof of Proposition~\ref{P:degreelowering}.

\section{Seminorm control and smoothing for pairwise independent polynomials}\label{S:smoothing}
In this section, we prove Theorem~\ref{T:polies0}. The converse direction is standard and easy to prove, so we focus on the forward direction: if the polynomials $p_1, \ldots, p_\ell\in\Z[n]$ are pairwise independent, then some Host-Kra factors are characteristic for ergodic averages with commuting transformations and iterates $p_1(n), \ldots, p_\ell(n)$.

If $T_1 = \cdots = T_\ell$, then only essential distinctness of the polynomials is needed, and the claimed result was established in \cite{HK05b,Lei05c} following  a variant of the standard PET argument in \cite{Be87a}. The problem is that when we deal with multiple transformations, a PET argument alone can only give a control of an average
\begin{align}\label{original average}
	\E_{n\in[N]}\, T_1^{p_1(n)}f_1 \cdots T_\ell^{p_\ell(n)}f_\ell
\end{align}
by box seminorms $\nnorm{\cdot}_{\b_1, \ldots, \b_{s+1}}$ with potentially distinct vectors $\b_1, \ldots, \b_{s+1}$. For instance, for the average
\begin{align}\label{E: n^2 n^2+n}
	\E_{n\in[N]}\, T_1^{n^2}f_1\cdot T_2^{n^2+n}f_2,
\end{align}
the usual PET argument gives control by the box seminorm $\nnorm{f_2}_{\be_2^{\times s_1}, (\be_2-\be_1)^{\times s_2}}$ for some $s_1, s_2>0$.
This is not good enough for applications such as Theorem \ref{T:polies2}. The relevant factors for box seminorms are more complicated than the Host-Kra factors for a single transformation, and we do not have a comparable structure theory for them. Neither can we obtain Theorem \ref{T:polies2} by performing the degree lowering argument directly with box seminorms without running into all sorts of difficulties. Therefore it is essential for our applications to strengthen the control by box seminorms into a control by Gowers-Host-Kra seminorms. This is achieved step-by-step. We first pass from controlling \eqref{original average} by a seminorm $\nnorm{f_\ell}_{\b_1, \ldots, \b_{s+1}}$ to a control by $\nnorm{f_\ell}_{\b_1, \ldots, \b_s, \be_\ell^{\times s'}}$ for some (possibly large) value $s'\in\N$. Then we repeat the argument $s$ more times to control \eqref{original average} by the seminorm $\nnorm{f_\ell}_{\be_\ell^{s''}}=\nnorm{f_\ell}_{s'', T_\ell}$ for some $s''\in\N$. Finally, we use an approximation argument based on Proposition \ref{dual decomposition} and induction to obtain a control of \eqref{original average} by Gowers-Host-Kra seminorms of the other terms.

As mentioned above, the starting point is to obtain a control of the average \eqref{original average} by some box seminorm. For this purpose, we use the following result of Donoso, Ferr\'e-Moragues, Koutsogiannis, and Sun, reformulated in a language that better suits our needs.
\begin{proposition}[{\cite[Theorem 2.5]{DFMKS21}}]\label{PET bound}
	Let $d, \ell\in\N$, $\p_1, \ldots, \p_\ell\in\Z[n]^\ell$ be essentially distinct  polynomials that have the form  $\p_j(n) = \sum_{i=0}^d \a_{ji} \, n^i$, $j\in [\ell]$,  and $(X, \CX, \mu, T_1, \ldots, T_\ell)$ be a system. Let  also $\p_0:=0$ and
	$$
	d_{\ell j} := \deg(\p_\ell - \p_j),\quad j=0,\ldots, \ell-1.
	$$
	 Then  there exists $s\in\N$ (depending only on $d, \ell$) and vectors
	\begin{align}\label{E:coefficientvectors 0}
		\b_1, \ldots, \b_{s+1} \in \{\a_{\ell d_{\ell j}}-\a_{j d_{\ell j}}\colon \; j = 0, \ldots, \ell-1\},
	\end{align}
	independent of the system, such that for all 1-bounded functions $f_1, \ldots, f_\ell\in L^\infty(\mu)$ we have
	\begin{align*}
		\lim_{N\to\infty}\norm{\E_{n\in[N]}\prod_{j\in[\ell]}T^{\p_j(n)}f_j}_{L^2(\mu)} = 0
	\end{align*}
	whenever $\nnorm{f_\ell}_{{\b_1}, \ldots, {\b_{s+1}}}=0$.
\end{proposition}
\begin{remark}
	This extends the results about characteristic factors from \cite{HK05b, Lei05c} that cover the case $T_1=\cdots=T_\ell$.
\end{remark}
The vectors $\b_1, \ldots, \b_{s+1}$ appearing in \eqref{E:coefficientvectors 0} are precisely the leading coefficients of the polynomials $\p_\ell, \p_\ell-\p_1, \ldots, \p_\ell - \p_{\ell-1}$; in particular, if $\p_\ell$ is essentially distinct from $\p_1, \ldots, \p_{\ell-1}$,\footnote{This means that $\p_\ell-\p_j$ is nonconstant  $j\in[\ell-1]$, which under our standing assumption that all polynomials have zero constant term is equivalent to distinctness.} then the vectors $\b_1, \ldots, \b_{s+1}$ are nonzero. For instance, taking $\p_1(n): = (n^2, 0) =  n^2\be_1$ and $\p_2(n): = (0, n^2+n)=(n^2+n)\be_2$ as in the average \eqref{E: n^2 n^2+n}, the vectors $\b_1, \ldots, \b_{s+1}$ would come from the set $\{\be_2, \be_2 -\be_1\}$.



Passing from a control by $\nnorm{f_\ell}_{\b_1, \ldots, \b_{s+1}}$ to a control by $\nnorm{f_\ell}_{\b_1, \ldots, \b_s, \be_\ell^{\times s'}}$ follows a two-step ``ping-pong'' strategy. Using the control by $\nnorm{f_\ell}_{\b_1, \ldots, \b_{s+1}}$, we first pass to an auxiliary control by some seminorm $\nnorm{f_i}_{\b_1, \ldots, \b_s, \be_i^{\times s_1}}$ for some $i\in [\ell-1]$, and then we use this auxiliary control to go back and control the average \eqref{original average} by $\nnorm{f_\ell}_{\b_1, \ldots, \b_s, \be_\ell^{\times s'}}$. We call the two steps outlined above \emph{ping} and \emph{pong}.

\subsection{Seminorm smoothing for the family $\{n, n^2, n^2+n\}$}
To better illustrate our techniques and motivate some arguments made in the general case, we first prove the following seminorm estimate for the pairwise independent family $n, n^2, n^2+n$.
\begin{proposition}\label{Iterated smoothing of n, n^2, n^2+n}
	There exists $s\in\N$ with the following property: for every system $(X, \CX, \mu, T_1, T_2, T_3)$ and 1-bounded functions $f_1, f_2, f_3\in L^\infty(\mu)$, the average
	\begin{align}\label{n, n^2, n^2+n}
		\E_{n\in[N]}\, T_1^n f_1 \cdot T_2^{n^2}f_2 \cdot T_3^{n^2+n}f_3
	\end{align}
	converges to 0 in $L^2(\mu)$ whenever $\nnorm{f_3}_{s, T_3}=0$.
\end{proposition}

By Proposition \ref{PET bound}, there exist vectors $\b_1, \ldots, \b_{s+1}\in\{\be_3, \be_3-\be_2\}$ with the property that if $\nnorm{f_3}_{{\b_1}, \ldots, {\b_{s+1}}} = 0$, then the limit of \eqref{n, n^2, n^2+n} is 0. The key then is the following seminorm smoothing argument, which allows to replace the vectors $\b_1, \ldots, \b_{s+1}$ by (possibly many copies of) $\be_3$.

\begin{proposition}[Seminorm smoothing for $n, n^2, n^2+n$]\label{smoothing of n, n^2, n^2+n}
	Let $s\in\N$, ${\b_1}, \ldots, {\b_{s+1}}\in\{\be_3, \be_3-\be_2\}$ and $(X,\CX, \mu,T_1, T_2, T_3)$ be a system. Suppose that for all 1-bounded functions  $f_1, f_2, f_3\in L^\infty(\mu)$, the equality $\nnorm{f_3}_{{\b_1}, \ldots, {\b_{s+1}}}= 0$ implies that
	\begin{align}\label{n, n^2, n^2+n vanishes}
		\lim_{N\to\infty}\E_{n\in[N]}\, T_1^n f_1 \cdot T_2^{n^2}f_2 \cdot T_3^{n^2+n}f_3 = 0.
	\end{align}
	Then \eqref{n, n^2, n^2+n vanishes} also holds under the assumption $\nnorm{f_3}_{{\b_1}, \ldots, {\b_{s}}, \be_3^{\times s'}}=0$ for some $s'\in\N$. Moreover, the value $s'$ is independent of the system $(X,\CX, \mu,T_1, T_2, T_3)$ or the functions $f_1, f_2, f_3$.
\end{proposition}
An iterative application of Proposition \ref{smoothing of n, n^2, n^2+n} gives Proposition \ref{Iterated smoothing of n, n^2, n^2+n}.

The content of Proposition \ref{smoothing of n, n^2, n^2+n} is that we can ``smooth out'' the seminorm, gradually replacing the potentially distinct vectors $\b_1, \ldots, \b_{s+1}$ by (possibly many copies of) the same vector $\be_3$ which corresponds to the transformation $T_3$ acting on $f_3$.

\begin{proof}[Proof of Proposition \ref{smoothing of n, n^2, n^2+n}]
	If $\b_{s+1} = \be_3$, there is nothing to prove, so we assume that $\b_{s+1} = \be_3 - \be_2$.
	
	Suppose that
	\begin{align}\label{average with n, n^2, n^2+n is positive}
		\lim_{N\to\infty}\norm{\E_{n\in[N]}\, T_1^n f_1 \cdot T_2^{n^2}f_2 \cdot T_3^{n^2+n}f_3}_{L^2(\mu)}>0
	\end{align}
	(the limit in \eqref{average with n, n^2, n^2+n is positive} and similar limits below exist by~\cite{Wal12}).
	Our argument, which we call a ``{\em ping-pong}'' argument,  proceeds in two steps. In the first (\textit{ping}) step, we prove the auxiliary result that
	\begin{align}\label{auxilliary estimate for n, n^2, n^2+n}
		\nnorm{f_2}_{{\b_1}, \ldots, {\b_{s}}, \be_2^{\times s_1}}>0
	\end{align}
	for some $s_1\in\N$. This is achieved by reducing the problem of getting a seminorm control for the original average \eqref{n, n^2, n^2+n} to a similar problem but for an average of the form
	\begin{align}\label{n, n^2, n^2+n aux}
		\E_{n\in[N]}\, T_1^n g_1 \cdot T_2^{n^2}g_2 \cdot T_2^{n^2+n}g_3.
	\end{align}
	This average is simpler (in a  suitable sense to be explained later) than the original average, and obtaining seminorm estimates for \eqref{n, n^2, n^2+n aux} can be accomplished by invoking directly Proposition \ref{PET bound}. In the second (\textit{pong}) step, we use the auxiliary estimate \eqref{auxilliary estimate for n, n^2, n^2+n} to replace the function $f_2$ in the original average by a dual function. The claim
	$\nnorm{f_3}_{{\b_1}, \ldots, {\b_{s}}, \be_3^{\times s'}}>0$ then follows by invoking Proposition \ref{strong PET bound}, a corollary of Proposition \ref{PET bound}, to handle averages of the form
	\begin{align}\label{n, n^2, n^2+n aux 2}
		\E_{n\in[N]}\, T_1^n g_1 \cdot \prod_{j\in [L]} \CD_{j}(n^2) \cdot T_3^{n^2+n}g_3,
	\end{align}
	where $\CD_1, \ldots, \CD_L\in\FD_{s'}$ are sequences coming from dual functions (see Section \ref{SS:dual}).

	\smallskip

	\textbf{Step 1 (\textit{ping}): Obtaining control by a seminorm of $f_2$.}
	\smallskip
	
	From \eqref{average with n, n^2, n^2+n is positive} and Proposition \ref{P:dual replacement} we deduce that
	\begin{align*}
		\lim_{N\to\infty}\norm{\E_{n\in[N]}\, T_1^n f_1 \cdot T_2^{n^2}f_2 \cdot T_3^{n^2+n}\tilde{f}_3}_{L^2(\mu)}>0
	\end{align*}
	for some function
	\begin{equation*}
		\tilde{f}_3:=\lim_{k\to\infty} \E_{n\in [N_k]}
		\, T_3^{-(n^2+n)}g_k \cdot T_3^{-(n^2+n)} T_1^n \bar{f}_1 \cdot T_3^{-(n^2+n)} T_2^{n^2}\bar{f}_2,
	\end{equation*}
	where the limit is a weak limit. Then our assumption gives
	$$
	\nnorm{\tilde{f}_3}_{{\b_1}, \ldots, {\b_{s+1}}}>0.
	$$
	Subsequently, Proposition \ref{dual-difference interchange} implies that
	\begin{multline*}
		\liminf_{H\to\infty}\E_{\uh,\uh'\in [H]^{s}}\\
		\lim_{N\to\infty}\norm{ \E_{n\in[N]} \, T_1^n (\Delta_{{\b_1}, \ldots, {\b_{s}}; \uh-\uh'} f_1)\cdot T_2^{n^2} (\Delta_{{\b_1}, \ldots, {\b_{s}}; \uh-\uh'} f_2)\cdot T_3^{n^2+n}u_{\uh,\uh'}}_{L^2(\mu)}>0
	\end{multline*}
		for some 1-bounded functions $u_{\uh,\uh'}$, $\uh,\uh'\in \N^s$,  invariant under $T^{\b_{s+1}}$. This invariance property and the fact that  $T^{\b_{s+1}} = T_3 T_2^{-1}$ imply that
	\begin{equation}\label{invariance n, n^2, n^2+n}
		T_3 u_{\uh,\uh'}=T_2 u_{\uh,\uh'},
	\end{equation}
	and so
	\begin{multline*}
		\liminf_{H\to\infty}\E_{\uh,\uh'\in [H]^{s}}\\
		\lim_{N\to\infty}\norm{ \E_{n\in[N]} \, T_1^n (\Delta_{{\b_1}, \ldots, {\b_{s}}; \uh-\uh'} f_1)\cdot T_2^{n^2} (\Delta_{{\b_1}, \ldots, {\b_{s}}; \uh-\uh'} f_2)\cdot T_2^{n^2+n}u_{\uh,\uh'}}_{L^2(\mu)}>0.
	\end{multline*}
	Thus, we have replaced $T_3$ in the average by $T_2$. We note now that all the terms with highest-degree polynomials involve the same transformation $T_2$. Then there exist $\varepsilon>0$ and a subset $B\subseteq \N^{2s}$ of positive lower density such that
	\begin{align*}
		\lim_{N\to\infty}\norm{\E_{n\in[N]} \, T_1^n (\Delta_{{\b_1}, \ldots, {\b_{s}}; \uh-\uh'} f_1)\cdot T_2^{n^2} (\Delta_{{\b_1}, \ldots, {\b_{s}}; \uh-\uh'} f_2)\cdot T_2^{n^2+n}u_{\uh,\uh'}}_{L^2(\mu)}>\varepsilon
	\end{align*}
	for all $(\uh, \uh')\in B$. These averages take the form
	\begin{align*}
		\E_{n\in[N]}\, T_1^n g_1 \cdot T_2^{n^2}g_2 \cdot T_2^{n^2+n}g_3,
	\end{align*}
	and we know from Proposition \ref{PET bound} that they are controlled by $\nnorm{g_2}_{s_1, T_2}$ for some $s_1\in\N$.\footnote{{For pairwise dependent polynomial iterates, our argument fails  at this step. For example, suppose we want to get seminorm control for the  averages $\E_{n\in[N]} T_1^nf_1\cdot T_2^nf_2$.  Our argument gives such control assuming we have a similar control for the averages $\E_{n\in[N]} T_1^ng_1\cdot T_1^ng_2$. But these averages cannot be   controlled by Host-Kra seminorms of $g_1$ or $g_2$ with respect to the transformation $T_1$ since for $g_2:=\bar{g}_1$, the average converges to $\E(|g_1|^2|\CI(T_1))$, which can be nonzero even if $\nnorm{g_1}_{s, T_1} = 0$ for all $s\in\N$.}} Using this fact and Proposition \ref{P:Us}, we deduce that for all $(\uh, \uh')\in B$, we have the lower bound
	\begin{align*}
		\nnorm{\Delta_{{\b_1}, \ldots, {\b_{s}}; \uh-\uh'} f_2}_{s_1, T_2} > \delta
	\end{align*}
	for some $\delta>0$ independent of  $(\uh, \uh')$.
	Consequently,
	\begin{align}\label{positivity of differences}
		\liminf_{H \to\infty}\E_{\uh,\uh'\in [H]^{s}} \nnorm{\Delta_{{\b_1}, \ldots, {\b_{s}}; \uh-\uh'} f_2}_{s_1, T_2} >0.
	\end{align}
	Together with Lemma \ref{difference sequences}, the inductive formula for seminorms \eqref{inductive formula}  and H\"older inequality, the inequality \eqref{positivity of differences} implies that
	$$
	\nnorm{f_2}_{{\b_1}, \ldots, {\b_{s}}, \be_2^{\times s_1}}>0.
	$$
	Hence, the seminorm $\nnorm{f_2}_{{\b_1}, \ldots, {\b_{s}}, \be_2^{\times s_1}}$ controls the average \eqref{n, n^2, n^2+n}.
	
	Starting with a control of \eqref{n, n^2, n^2+n} by a seminorm of $f_3$, we have arrived at a control by a seminorm of $f_2$. Since among the vectors $\b_1, \ldots, \b_s$ we may find the vector $\be_3$ corresponding to the transformation $T_3$ and not involving $T_2$ in any way, this seminorm control is not particularly useful as an independent result. However, as we explain shortly, it proves to be of importance as an intermediate result, used for obtaining our claimed control of \eqref{n, n^2, n^2+n} by $\nnorm{f_3}_{\b_1, \ldots, \b_s, \be_3^{\times s'}}$.

	\smallskip
	
	\textbf{Step 2 (\textit{pong}): Obtaining control by a seminorm of $f_3$.}
	\smallskip

	To get the claim that $\nnorm{f_3}_{{\b_1}, \ldots, {\b_{s}}, \be_3^{\times s'}}$ controls the average for some $s'\in\N$, we repeat the procedure once more  with $f_2$ in place of $f_3$. From \eqref{average with n, n^2, n^2+n is positive} and Proposition \ref{P:dual replacement}, it follows that
	\begin{align*}
		\lim_{N\to\infty}\norm{\E_{n\in [N]} \, T_1^n f_1 \cdot T_2^{n^2}\tilde{f}_2 \cdot T_3^{n^2+n} f_3}_{L^2(\mu)}>0
	\end{align*}
	for some function
	\begin{equation*}
		\tilde{f}_2:=\lim_{k\to\infty} \E_{n\in [N_k]}
		\, T_2^{-n^2}g_k\cdot T_2^{-n^2} T_1^{n}\bar{f}_1 \cdot T_2^{-n^2} T_3^{n^2+n}\bar{f}_3,
	\end{equation*}
	where the limit is a weak limit. Then the previous result gives
	$$
	\nnorm{\tilde{f}_2}_{{\b_1}, \ldots, {\b_{s}}, \be_2^{\times s_1}}>0.
	$$
	Previously, we applied Proposition~\ref{dual-difference interchange}~(i)  to get rid of ${\b_{s+1}}$. This time, our goal is to get rid of $\be_2^{\times s_1}$.  We apply Proposition~\ref{dual-difference interchange} again, but this time we use its part (iii) since we want to get rid of all the $s_1$ vectors $\be_2$ at once.\footnote{If we used the first part of Proposition \ref{dual-difference interchange} like in the \textit{ping} step, we would have to deal with difference functions $\Delta_{{\b_1}, \ldots, {\b_{s}}, \be_2^{\times (s_1-1)}; \uh-\uh'} f_j$. Then we could only control the original average in terms of the seminorm $\nnorm{f_3}_{{\b_1}, \ldots, {\b_{s}}, \be_2^{\times (s_1-1)}, \be_3^{\times s'}}$ which is less  useful because of the presence of vectors $\be_2$.}
	This gives
	\begin{multline*}
		\liminf_{H\to\infty}\E_{\uh,\uh'\in [H]^{s}}\\
		\lim_{N\to\infty}\norm{ \E_{n\in[N]} \, T_1^n (\Delta_{{\b_1}, \ldots, {\b_{s}}; \uh-\uh'} f_1)\cdot \CD_{\uh,\uh'}(n^2) \cdot T_3^{n^2+n} (\Delta_{{\b_1}, \ldots, {\b_{s}}; \uh-\uh'} f_3)}_{L^2(\mu)}>0,
	\end{multline*}
	where
	\begin{align*}
		\CD_{\uh,\uh'}(n):=T_2^n\prod_{\ueps\in\{0,1\}^{s_1}}\CC^{|\ueps|}\CD_{s', T_2}(\Delta_{{\b_1}, \ldots, {\b_{s}}; \uh^{\ueps}}\tilde{f}_2).
	\end{align*}
	Thus, $\CD_{\uh,\uh'}$ is a product of $2^{s_1}$ elements of $\FD_{s'}$. Invoking Proposition \ref{strong PET bound}, a version of Proposition \ref{PET bound} that allows us to ignore the dual term $\CD_{\uh,\uh'}$, we deduce that every average
	\begin{align*}
		\E_{n\in[N]} \, T_1^n (\Delta_{{\b_1}, \ldots, {\b_{s}}; \uh-\uh'} f_1)\cdot \CD_{\uh,\uh'}(n^2) \cdot T_3^{n^2+n} (\Delta_{{\b_1}, \ldots, {\b_{s}}; \uh-\uh'} f_3)
	\end{align*}
	is controlled by $\nnorm{\Delta_{{\b_1}, \ldots, {\b_{s}}; \uh-\uh'} f_3}_{s', T_3}$ for some $s'\in\N$ that depends only on the vectors $\b_1, \ldots, \b_s$. We combine this insight with Proposition \ref{P:Us} and a pigeonholing argument like in the \textit{ping} step\footnote{Specifically, we invoke Proposition \ref{P:Us} for averages of the form $\E_{n\in[N]}\,  T_1^{n}g_1 \cdot \prod_{j\in [2^{s_1}]} T_2^{n^2} g'_j \cdot T_3^{n^2+n}g_3$, where $g'_j$ is $\CZ_{s_1}(T_2)$-measurable for each $j\in[2^{s_1}]$. One can show that an average like this is qualitatively controlled by $\nnorm{g_3}_{s', T_3}$ by approximating functions $g'_1, \ldots, g'_{2^{s_1}}$ by linear combinations of dual functions using Proposition \ref{dual decomposition} and then applying Proposition \ref{strong PET bound}.} to conclude that
	\begin{align}\label{positivity of differences 2}
		\liminf_{H \to\infty}\E_{\uh,\uh'\in [H]^{s}} \nnorm{\Delta_{{\b_1}, \ldots, {\b_{s}}; \uh-\uh'} f_3}_{s', T_3} >0.
	\end{align}
	Together with Lemma \ref{difference sequences}, the inductive formula \eqref{inductive formula}, and H\"older inequality, the inequality \eqref{positivity of differences 2} gives us
	$$
	\nnorm{f_3}_{{\b_1}, \ldots, {\b_{s}}, \be_3^{\times s'}}>0,
	$$
	which is precisely what we claimed to prove.
\end{proof}

\subsection{Introducing the formalism for longer families}\label{S:formalism pairwise independent}
The idea for handling longer families is to reduce an arbitrary average to an average of smaller ``type''. For instance, if $p_1, p_2, p_3\in\Z[n]$ have the same degree, an arbitrary average
\begin{align*}
	\E_{n\in[N]}\, T_1^{p_1(n)}f_1\cdot T_2^{p_2(n)}f_2 \cdot T_3^{p_3(n)}f_3,
\end{align*}
has type $(1, 1, 1)$ (in the sense that each of $T_1, T_2, T_3$ occurs exactly once) and can be reduced to an average of type $(2, 1, 0)$ (meaning that $T_1$ occurs twice, $T_2$ occurs once, $T_3$ does not occur), or $(1, 2, 0)$ ($T_1, T_2, T_3$ occur once, twice and zero times respectively). We then show that an average of type $(2, 1, 0)$ reduces to an average of type $(3, 0, 0)$ ($T_1$ occurs three times), which can be directly controlled using Proposition \ref{strong PET bound}.

We now introduce a handy formalism for the induction scheme in the proof of Theorem~\ref{T:polies0}. Specifically, we conceptualise important properties of averages
\begin{align}\label{general average}
	\E_{n\in [N]} \,\prod_{j\in[\ell]}T_{\eta_j}^{p_j(n)}f_j \cdot \prod_{j\in[L]}\CD_{j}(q_j(n)).
\end{align}
For an average \eqref{general average},
we let $\ell$ be its \emph{length}, $d := \max\limits_{j\in[\ell]}\deg p_j$ be its \emph{degree}, and $\eta:=(\eta_1,\ldots, \eta_\ell)$ be its \emph{indexing tuple}.
Furthermore, we define
\begin{align*}
	\FL := \{j\in[\ell]\colon \ \deg p_j = d\}
\end{align*}
to be the set of indices corresponding to polynomials of maximum degree.
We may denote $\FL = \FL(p_1, \ldots, p_\ell)$
to signify dependence on the specific polynomial family.

The induction scheme is guided by the principle that while proving Theorem~\ref{T:polies0} for a given average, we invoke the result for ``simpler'' averages. This relative simplicity is formalised by the following notion of type.
We let the \emph{type} of \eqref{general average} to be the tuple $w := (w_1, \ldots, w_\ell)$. Each entry $w_t$ is defined by
\begin{align*}
	w_t := |\{j\in \FL\colon \ \eta_t = j\}| = |\{j\in[\ell]\colon \ \eta_t = j,\ \deg p_t = d\}|,
\end{align*}
and it represents the number of times the transformation $T_t$ appears in the average with a polynomial iterate of maximum degree. For instance, the average
\begin{align*}
	\E_{n\in[N]}\, T_1^{n^3}f_1\cdot T_1^{n^2}f_2 \cdot T_2^{n^3 + n}f_3\cdot T_1^n f_4 \cdot T_2^{2 n^3 + n}f_5 \cdot \CD_{1}(n^4)\cdot \CD_{2}(n^6)
\end{align*}
has length 5, degree 3, indexing tuple (1, 1, 2, 1, 2)  and type $(1,2)$ because the transformation $T_1$ appears only once with the polynomial of maximum degree (at the index $j=1$) while $T_2$ has an iterate of maximum degree twice, at $j=3, 5$. It does not matter that the polynomials in dual terms have higher degrees since our notion of type completely ignores how many dual terms we have and what polynomials they involve.

To organise the induction scheme, we need to define an ordering on types $w$ with a fixed length $|w|:= w_1 + \cdots + w_{\ell}$. There are several good ways of doing so; we choose the most efficient one. Let $\supp(w) = \{t\in[\ell]:\ w_t>0\}$. For distinct integers $m,i\in [\ell]$ with $m\in\supp(w)$, we define the type operation
\begin{align*}
	(\sigma_{mi}w)_t := \begin{cases} w_t,\; &t \neq m,i\\
		w_m-1,\; &t = m\\
		w_i + 1,\; &t = i.
	\end{cases}.
\end{align*}
For instance, $\sigma_{12}(2,3, 7) = (1, 4, 7)$.
Letting $w':=\sigma_{mi}w$, we set $w' < w$ if $w_m\leq w_i$. In particular, $(1, 4, 7)<(2, 3, 7)$ in the example above. We note that the tuple $(2, 3, 7)$ of higher type has smaller variance than the tuple $(1, 4, 7)$ of smaller type; this is a consequence of the fact that while passing from $(2,3,7)$ to $(1,4,7)$, we decrease by 1 the smallest nonzero value 2.
This observation carries forward more generally: if the condition $w_m\leq w_i$ is satisfied, then an easy computation shows that $w'$ has strictly higher variance than $w$, or equivalently ${w'_1}^2 + \cdots + {w'_\ell}^2 > w_1^2 + \cdots + w_\ell^2$. Thanks to this fact, we can extend the ordering $<$ to all tuples $\N_0^\ell$ of fixed length $|w|$ by transitivity, and so for two type tuples
$w, w'\in\N_0^{\ell}$, we let  $w'<w$ if there exist types $w_0, \ldots, w_r$ with $w_0 = w$, $w_r = w'$, such that for every $l= 0, \ldots, r-1$, we have $w_{l+1} = \sigma_{mi} w_l$ for distinct $m,i\in\FL$ with $w_{lm}\leq w_{li}$. For instance, this ordering induces the following chains of types:
\begin{align*}
	(4, 0, 0) < (3,1,0) < (2, 2, 0) < (2, 1, 1) \quad \textrm{and}\quad
	(0, 4, 0) < (1, 3, 0) < (2, 2, 0) < (2, 1, 1),
\end{align*}
and we could use either sequence of steps in our argument.

The ordering $<$ on types organises the way in which we pass from averages of indexing tuple $\eta$ to averages of indexing tuple $\eta'$. After applying the type operations finitely many types to a type $w'$, we arrive at a type $w$ with the property that $w_t = 0$ for all but one value $t\in\FL$. These types $w$ will serve as the basis for our induction procedure because they allow us to directly invoke Proposition \ref{strong PET bound}, therefore we call them \emph{basic}. For instance, $(4, 0, 0)$ and $(0, 4, 0)$ in the example above are basic types but $(3, 1, 0), (1, 3, 0), (2,2,0), (2,1,1)$ are not.\footnote{As indicated before, we could define the ordering on types in a different way. For instance, we could define $w':= \sigma_{mi} w$ if $m<i$ (or if $m>i$) and extend it by transitivity to all tuples of fixed length. The choice of the order is ultimately a matter of taste; ours is the most efficient one in that it allows for the swiftest passage to tuples of basic type.}

\subsection{The induction scheme}
Theorem~\ref{T:polies0} follows from the result below upon setting $L:=0$ and letting $\eta$ be the identity map.
\begin{proposition}\label{Host Kra characteristic pairwise independent 2}
	Let $d, \ell, L\in\N$, $\eta\in[\ell]^\ell$ be an indexing tuple and  $p_1, \ldots, p_\ell, q_1, \ldots, q_L\in\Z[n]$ be polynomials of degree at most $d$ such that $p_1, \ldots, p_\ell$ are pairwise independent. Then there exists an integer  $s\in\N$ (depending only on $d,\ell,L$) such that for all systems  $(X,\CX, \mu,T_1, \ldots, T_\ell)$, 1-bounded functions $f_1, \ldots, f_\ell\in L^\infty(\mu)$, and sequences of dual functions $\CD_{1}, \ldots, \CD_L\in\FD_d$, we have
	\begin{align}\label{general average vanishes}
		\lim_{N\to\infty}\norm{\E_{n\in [N]} \,\prod_{j\in[\ell]}T_{\eta_j}^{p_j(n)}f_j \cdot \prod_{j\in[L]}\CD_{j}(q_j(n))}_{L^2(\mu)} = 0
	\end{align}
	whenever $\nnorm{f_j}_{s, T_{\eta_j}} = 0$.
\end{proposition}
In the seminorm smoothing argument, we use the following soft quantitative version which follows by combining Propositions \ref{Host Kra characteristic pairwise independent 2} and \ref{P:Us}.
\begin{corollary}\label{C:Host Kra characteristic pairwise independent 2}
	Let $d, \ell, L\in\N$, $\eta\in[\ell]^\ell$ be an indexing tuple and $p_1, \ldots, p_\ell, q_1, \ldots, q_L\in\Z[n]$ be polynomials  such that $p_1, \ldots, p_\ell$ are pairwise independent. Let $s\in\N$ be as in Proposition \ref{Host Kra characteristic pairwise independent 2} and  $(X,\CX, \mu,T_1, \ldots, T_\ell)$ be a system. Then for every $\varepsilon>0$ there exists $\delta>0$ such that for all 1-bounded functions $f_1, \ldots, f_\ell\in L^\infty(\mu)$ and sequences of dual functions $\CD_{1}, \ldots, \CD_L\in\FD_d$, we have
	\begin{align}\label{general average vanishes quant}
		\lim_{N\to\infty}\norm{\E_{n\in [N]} \,\prod_{j\in[\ell]}T_{\eta_j}^{p_j(n)}f_j \cdot \prod_{j\in[L]}\CD_{j}(q_j(n))}_{L^2(\mu)} <\varepsilon
	\end{align}
	whenever $\nnorm{f_j}_{s, T_{\eta_j}} <\delta$ for some $j\in[\ell]$.
\end{corollary}
\begin{remarks}\begin{itemize}
    \item Since $\CD_{j}(q_j(n))$ has the form $T_{\pi_j}^{q_j(n)}g_j$ for some $\pi_j\in[\ell]$ and $g_j\in L^\infty(\mu)$, the previous limit exists by \cite{Wal12}. The same comment applies for all limits involving dual sequences that appear in the remaining of this section.
    \item The parameter $\delta$ in Corollary \ref{C:Host Kra characteristic pairwise independent 2} is completely ineffective since it is given by the ineffective Proposition \ref{P:Us}. Therefore, it may depend not only on $\varepsilon$, but also on $d,\ell, L, \eta, s$, polynomials, and the system. Crucially, it does not depend on the functions $f_j$ nor the dual sequences $\CD_j$.
\end{itemize}
\end{remarks}
\begin{proof}
	We prove Corollary \ref{C:Host Kra characteristic pairwise independent 2} for fixed $d, \ell, L, \eta$ by assuming Proposition \ref{Host Kra characteristic pairwise independent 2} for the same parameters.
	
	Fix $\pi\in[\ell]^L$. We first prove the following claim: for all  $f_1, \ldots, f_\ell, g_1, \ldots, g_L\in L^\infty(\mu)$ with $\CZ_d(T_{\pi_j})$-measurable   $g_j$ for each $j\in[L]$, we have
	\begin{align}\label{general average vanishes 11}
		\lim_{N\to\infty}\norm{\E_{n\in [N]} \,\prod_{j\in[\ell]}T_{\eta_j}^{p_j(n)}f_j \cdot  \prod_{j\in[L]}T^{q_j(n)}_{\pi_j}g_j}_{L^2(\mu)} = 0
	\end{align}
	whenever $\nnorm{f_j}_{s, T_{\eta_j}} = 0$ for some $j\in[\ell]$.
	
	Fix $f_1, \ldots, f_\ell, g_1, \ldots, g_L$ and suppose that \eqref{general average vanishes 11} fails. Using Proposition \ref{dual decomposition} and the pigeonhole principle, we deduce that there exist dual functions $g_1', \ldots, g_L'$ of $T_{\pi_1}, \ldots, T_{\pi_L}$ of degree $d$ such that
	\begin{align*}
		\lim_{N\to\infty}\norm{\E_{n\in [N]} \,\prod_{j\in[\ell]}T_{\eta_j}^{p_j(n)}f_j \cdot  \prod_{j\in[L]}T^{q_j(n)}_{\pi_j}g_j'}_{L^2(\mu)} > 0.
	\end{align*}
	Setting $\CD_j(n) := T_{\pi_j}^{n}g'_j$ and using Proposition \ref{Host Kra characteristic pairwise independent 2}, gives that $\nnorm{f_j}_{s, T_{\eta_j}}>0$ for all $j\in[\ell]$, and so the claim follows.
	
	Combining the claim above with Proposition \ref{P:Us} for $\CY_1 = \cdots = \CY_\ell = \CX$ and $\CY_{\ell+1} = \cdots = \CY_{\ell+L} =  \CZ_d(T_{\pi_j})$, we deduce the following: for every $\varepsilon>0$ there exists $\delta_\pi>0$ such that for all 1-bounded functions $f_1, \ldots, f_\ell, g_1, \ldots, g_L\in L^\infty(\mu)$ with $g_j\in L^\infty(\CZ_d(T_{\pi_j}))$, we have
	\begin{align*}
		\lim_{N\to\infty}\norm{\E_{n\in [N]} \,\prod_{j\in[\ell]}T_{\eta_j}^{p_j(n)}f_j \cdot  \prod_{j\in[L]}T^{q_j(n)}_{\pi_j}g_j}_{L^2(\mu)} < \varepsilon
	\end{align*}
	whenever $\nnorm{f_j}_{s, T_{\eta_j}} <\delta_\pi$ for some $j\in[\ell]$.
	
	Corollary \ref{C:Host Kra characteristic pairwise independent 2} follows by taking
	$\delta:= \min\{\delta_\pi\colon \, \pi\in[\ell]^L\}$ and recalling that dual functions of $T_j$ of order $d$ are $\CZ_d(T_j)$-measurable, hence for every $j\in[L]$, the sequence $\CD_{j}(q_j(n))$ has the form $\CD_j(q_j(n)) = T_{\pi_j}^{q_j(n)}g_j$ for some $\pi_j\in[\ell]$ and a 1-bounded $\CZ_d(T_j)$-measurable  function $g_j\in L^\infty(\mu)$.
\end{proof}

Proposition \ref{Host Kra characteristic pairwise independent 2} will be deduced from the following result.
\begin{proposition}\label{P:iterated smoothing pairwise}
	Let $d, \ell, L\in\N$, $\eta\in[\ell]^\ell$ be an indexing tuple and   $p_1, \ldots, p_\ell, q_1, \ldots, q_L\in\Z[n]$ be polynomials of degree at most $d$ such  that $p_1, \ldots, p_\ell$ are pairwise independent. Let $(X, \CX, \mu, T_1, \ldots, T_\ell)$ be a system, $w$ be the type of the average \eqref{general average}, and $m\in[\ell]$ be such that $w_{\eta_m} = \min\limits_{t\in\supp(w)} w_t$.
	Then
	there exist $s\in\N$ (depending only on $d,\ell,L$) such that for all 1-bounded functions $f_1, \ldots, f_\ell\in L^\infty(\mu)$ and sequences of functions $\CD_{1}, \ldots, \CD_L\in\FD_d$, we obtain \eqref{general average vanishes} whenever $\nnorm{f_m}_{s, T_{\eta_m}} = 0$.
\end{proposition}
For instance, if $w := (1,2,1,0,0,3,0)$, then \eqref{general average vanishes} follows whenever $\nnorm{f_m}_{s, T_{\eta_m}}= 0$ for $m\in[7]$ satisfying $\eta_m = 1$ or $\eta_m = 3$.

Proposition \ref{P:iterated smoothing pairwise} is a straightforward consequence of Proposition \ref{strong PET bound}, followed by an iterated application of the smoothing result given below.
\begin{proposition}  \label{P:smoothing pairwise}
	Let $d, \ell, L, s\in\N$, $\eta\in[\ell]^\ell$ be an indexing tuple, and $p_1, \ldots, p_\ell$, $q_1, \ldots, q_L\in\Z[n]$ be polynomials  of degree at most $d$ such that $p_1, \ldots, p_\ell$ are pairwise independent. Let $(X, \CX, \mu, T_1, \ldots, T_\ell)$ be a system, $w$ be the type of the average \eqref{general average}, and $m\in[\ell]$ be such that $w_{\eta_m} = \min\limits_{t\in\supp(w)} w_t$.
	Then for each vectors  $\b_1, \ldots, \b_{s+1}$ satisfying \eqref{E:coefficientvectors} there exists $s'\in\N$ (depending only on $d,\ell,L,s$) with the following property: for all 1-bounded functions $f_1, \ldots, f_\ell\in L^\infty(\mu)$ and sequences of functions $\CD_{1}, \ldots, \CD_L\in\FD_d$, if $\nnorm{f_m}_{{\b_1},\ldots, {\b_{s+1}}}= 0$ implies \eqref{general average vanishes}, then so does $\nnorm{f_m}_{{\b_1},\ldots, {\b_{s}}, \be_{\eta_m}^{\times s'}}= 0$.
\end{proposition}

We prove Propositions \ref{Host Kra characteristic pairwise independent 2}, \ref{P:iterated smoothing pairwise}, and \ref{P:smoothing pairwise},  as well as Corollary \ref{C:Host Kra characteristic pairwise independent 2}, by induction on the length $\ell\in\N$, and for each fixed $\ell$ we further induct on type using the ordering $<$. Specifically, for fixed length $\ell$ and basic types $w$, the three aforementioned propositions follow trivially from Proposition \ref{strong PET bound}. This also covers the base case $\ell = 1$, since then the average has the basic type $(1)$. For an average of length $\ell$ and a non-basic type $w$, we prove Proposition \ref{P:smoothing pairwise} by assuming Corollary \ref{C:Host Kra characteristic pairwise independent 2} for length $\ell-1$ as well as for length $\ell$ and type $w'<w$.
We then infer Proposition \ref{Host Kra characteristic pairwise independent 2} for length $\ell$ and type $w$ by combining this result for length $\ell-1$ with Proposition \ref{P:iterated smoothing pairwise} for length $\ell$ and type $w$. Finally, Corollary \ref{C:Host Kra characteristic pairwise independent 2} for length $\ell$ and type $w$ follows directly from Proposition \ref{Host Kra characteristic pairwise independent 2} for $\ell$ and $w$ as proved before, and Proposition \ref{P:iterated smoothing pairwise} for length $\ell$ and type $w$ follows from Proposition \ref{P:smoothing pairwise} for $\ell$ and $w$ by a straightforward argument sketched below the statement of Proposition \ref{P:iterated smoothing pairwise}.

We illustrate the aforementioned induction scheme by analysing how Proposition \ref{Host Kra characteristic pairwise independent 2} is derived for the average
\begin{align}\label{type 111}
	\lim_{N\to\infty} \E_{n\in[N]}\, T_1^{n^3-n}f_1 \cdot T_2^{n^3}f_2 \cdot T_3^{n^3+n}f_3
\end{align}
for $1$-bounded functions $f_1,f_2,f_3\in L^\infty(\mu)$.
This average has type $(1, 1, 1)$ because each of the transformations $T_1, T_2, T_3$ has a polynomial iterate of degree 3 exactly once. By Proposition \ref{strong PET bound}, this average is controlled by the seminorm $\nnorm{f_3}_{{\b_1}, \ldots, {\b_{s+1}}}$ for some vectors $\b_1, \ldots, \b_{s+1}\in\{\be_3, \be_3-\be_2, \be_3-\be_1\}$. The goal of Proposition \ref{P:smoothing pairwise} is to show that \eqref{type 111} is also controlled by the seminorm $\nnorm{f_3}_{{\b_1}, \ldots, {\b_{s}}, \be_3^{\times s'}}$ for some $s'\in\N$. Applying this result iteratively, we deduce that \eqref{type 111} is controlled by the seminorm $\nnorm{f_3}_{s'', T_3}$ for some $s''\in\N$, and this is the content of Proposition \ref{P:iterated smoothing pairwise}. An analogous argument gives control of \eqref{type 111} by the seminorms $\nnorm{f_1}_{s'', T_1}$ and $\nnorm{f_2}_{s'', T_2}$ (for possibly different values $s''$).

The proof of Proposition \ref{P:smoothing pairwise} for \eqref{type 111} proceeds in the same two steps as the proof of Proposition \ref{smoothing of n, n^2, n^2+n}. Suppose for the sake of concreteness that $\b_{s+1} = \be_3-\be_2$. In the \textit{ping} step, we show that the average is controlled by  $\nnorm{f_2}_{{\b_1}, \ldots, {\b_{s}}, \be_2^{\times s_1}}$ for some $s_1\in\N$. To show this, we assume for the sake of contradiction that
\begin{align*}
	\lim_{N\to\infty} \norm{\E_{n\in[N]}\, T_1^{n^3-n}f_1 \cdot T_2^{n^3}f_2 \cdot T_3^{n^3+n}f_3}_{L^2(\mu)}>0.
\end{align*}
We then replace $f_3$ by a suitably defined function $\tilde{f}_3$ using Proposition \ref{P:dual replacement}. Subsequently, we apply Proposition \ref{dual-difference interchange}~(i) like in the proof of Proposition \ref{smoothing of n, n^2, n^2+n}. Having done so, we arrive at
\begin{align*}
	\liminf_{H\to\infty}\E_{\uh,\uh'\in [H]^{s}} \lim_{N\to\infty}\norm{ \E_{n\in[N]} \, T_1^{n^3-n} f_{1, \uh, \uh'}\cdot T_2^{n^3} f_{2, \uh, \uh'} \cdot T_3^{n^3+n}u_{\uh,\uh'}}_{L^2(\mu)}>0
\end{align*}
for $f_{j, \uh, \uh'} := \Delta_{{\b_1}, \ldots, {\b_{s}}; \uh-\uh'} f_j$ and 1-bounded functions $u_{\uh, \uh'}$ invariant under $T_3 T_2^{-1}$. The invariance property of $u_{\uh, \uh'}$ gives
\begin{align}\label{averaged 123}
	\liminf_{H\to\infty}\E_{\uh,\uh'\in [H]^{s}} \lim_{N\to\infty}\norm{ \E_{n\in[N]} \, T_1^{n^3-n} f_{1, \uh, \uh'}\cdot T_2^{n^3} f_{2, \uh, \uh'} \cdot T_2^{n^3+n}u_{\uh,\uh'}}_{L^2(\mu)}>0.
\end{align}
Hence, for a set $B$ of $(\uh, \uh')$ of positive lower density, the averages inside \eqref{averaged 123} are bounded away from 0 by some $\varepsilon>0$. Each of these averages takes the form
\begin{align}\label{type 120}
	\lim_{N\to\infty} \E_{n\in[N]}\, T_1^{n^3-n}g_1 \cdot T_2^{n^3}g_2 \cdot T_2^{n^3+n}g_3
\end{align}
for some $1$-bounded $g_1,g_2,g_3\in L^\infty(\mu)$. They have type $(1, 2, 0)$, which is smaller than the type $(1, 1, 1)$ of the original average \eqref{type 111} because $(1, 2, 0) = \sigma_{32}(1,1,1)$. At this point, we could invoke Proposition \ref{Host Kra characteristic pairwise independent 2} inductively to assert that \eqref{type 120} is controlled by $\nnorm{g_2}_{s_1, T_2}$, but we need something stronger. We apply Corollary \ref{C:Host Kra characteristic pairwise independent 2}, the quantitative version of Proposition \ref{Host Kra characteristic pairwise independent 2}, to conclude that for $(\uh, \uh')\in B$, we have the estimate $\nnorm{f_{2, \uh, \uh'}}_{s_1, T_2}>\delta$ for some $\delta>0$ and $s_1\in\N$.
Arguing in the same way as in the proof of Proposition \ref{smoothing of n, n^2, n^2+n}, we infer from this that \eqref{type 111} is controlled by $\nnorm{f_2}_{{\b_1}, \ldots, {\b_{s}}, \be_2^{\times s_1}}$.

In the \textit{pong} step, we use the newly established control by $\nnorm{f_2}_{{\b_1}, \ldots, {\b_{s}}, \be_2^{\times s_1}}$ to replace $f_2$ in the original average by $\tilde{f}_2$ using Proposition \ref{P:dual replacement}. An application of Proposition \ref{dual-difference interchange} gives
\begin{align}\label{averaged 103}
	\liminf_{H\to\infty}\E_{\uh,\uh'\in [H]^{s}} \lim_{N\to\infty}\norm{ \E_{n\in[N]} \, T_1^{n^3-n} f_{1, \uh, \uh'}\cdot \prod_{j\in[2^s]}\CD_{j, \uh, \uh'}(n^3) \cdot T_3^{n^3+n}f_{3, \uh,\uh'}}_{L^2(\mu)}>0
\end{align}
for some $\CD_{1, \uh, \uh'}, \ldots, \CD_{2^s, \uh, \uh'}\in\FD_{s_1}$. Each of the averages inside \eqref{averaged 103} takes the form
\begin{align}\label{type 101}
	\lim_{N\to\infty} \E_{n\in[N]}\, T_1^{n^3-n}g_1 \cdot \prod_{j\in[2^s]}\CD_{j}(n^3)\cdot T_3^{n^3+n}g_3
\end{align}
for some $1$-bounded $g_1,g_2,g_3\in L^\infty(\mu)$.
Such averages have length 2, and so we once again quote Corollary \ref{C:Host Kra characteristic pairwise independent 2} inductively\footnote{Specifically, we invoke it for averages of the form $\E_{n\in[N]} T_1^{n^3-n}g_1 \cdot \prod_{j\in [L]} T_2^{n^3} g'_j \cdot T_3^{n^3+n}g_3$, where $g'_j$ is $\CZ_{s_1}(T_2)$-measurable for each $j\in[L]$. Corollary \ref{C:Host Kra characteristic pairwise independent 2}  for these averages follows from Proposition \ref{Host Kra characteristic pairwise independent 2} for averages of length 2 of the form \eqref{type 101}.} to conclude that
for a positive lower density set of tuples $(\uh, \uh')$, we have $\nnorm{f_{3, \uh, \uh'}}_{s', T_3}>\delta$ for some $s'\in\N$ and $\delta>0$.
Arguing as in the proof of Proposition \ref{smoothing of n, n^2, n^2+n}, we infer from this that $\nnorm{f_3}_{{\b_1}, \ldots, {\b_{s}}, \be_3^{\times s'}}$ controls \eqref{type 120}.

To sum up, in the proof of Proposition \ref{P:smoothing pairwise} for averages
\begin{align*}
	\lim_{N\to\infty} \E_{n\in[N]}\, T_1^{n^3-n}f_1 \cdot T_2^{n^3}f_2 \cdot T_3^{n^3+n}f_3,
\end{align*}
we invoke Corollary \ref{C:Host Kra characteristic pairwise independent 2} for averages of the form
\begin{align}\label{type 120 2}
	\lim_{N\to\infty} \E_{n\in[N]}\, T_1^{n^3-n}f_1 \cdot T_2^{n^3}f_2 \cdot T_2^{n^3+n}f_3
\end{align}
in the \textit{ping} step, and we do this for averages of the form
\begin{align}\label{type 101 2}
	\lim_{N\to\infty} \E_{n\in[N]}\, T_1^{n^3-n}f_1\cdot\prod_{j\in[2^s]}\CD_{j}(n^3) \cdot T_3^{n^3+n}f_3
\end{align}
in the \textit{pong} step. The procedure remains the same for a more general average \eqref{general average} of type $w$. In the \textit{ping} step of the proof of Proposition \ref{P:smoothing pairwise} for an average  \eqref{general average}, we invoke Corollary \ref{C:Host Kra characteristic pairwise independent 2} for an average of length 3 and type $\sigma_{\eta_m\eta_i}w$ for some $\eta_i\in\FL\setminus\{\eta_m\}$, where $m\in\FL$ is an index such that
$
w_{\eta_m}=\min\limits_{t\in\supp(w)}w_t.
$ This new average has polynomial iterates $p'_1, \ldots, p'_\ell$ which are still pairwise independent. The fact that the pairwise independence of the polynomials is preserved guarantees that we can apply Proposition \ref{strong PET bound} without running the risk that some of the vectors $\b_1, \ldots, \b_{s+1}$ are 0. In the \textit{pong} step, we use Corollary \ref{C:Host Kra characteristic pairwise independent 2} for averages of length $\ell-1$ with iterates $p_1, \ldots, p_{m-1}, p_{m+1}, \ldots, p_\ell$ .

\subsection{Proofs}

We now prove Proposition \ref{P:smoothing pairwise} for averages of length $\ell$ and type $w$ assuming Corollary \ref{C:Host Kra characteristic pairwise independent 2} for averages of length $\ell-1$ as well as length $\ell$ and type $w'$ strictly smaller than $w$.
\begin{proof}[Proof of Proposition \ref{P:smoothing pairwise}]
	We split the proof into two parts. The first part  covers  the case where the vector $w$ is basic and can be dealt immediately using Proposition \ref{PET bound}.  This can be thought of as the basis of the induction. The complementary non-basic case will be dealt using the  ping-pong strategy described above. This can be thought of as the induction step.
	
	\smallskip
	\textbf{Base case: $w$ is basic.}
	\smallskip
	
	Suppose first that the vector $w$ is basic, i.e. $w_t = 0$ for all but a single value $t \in[\ell]$. That means that all the polynomials of maximum degree are iterates of the same transformation $T_t$. Proposition \ref{PET bound} implies in this case that
	\begin{align*}
		\nnorm{f_m}_{{\b_1}, \ldots, {\b_{s+1}}} = \nnorm{f_m}_{{a_1}\be_t, \ldots, {a_{s+1}}\be_t}
	\end{align*}
	for some  $a_1, \ldots, a_{s+1}\in \Z$ such that $\b_i = a_i \be_t$ satisfy the condition \eqref{E:coefficientvectors}. The assumption that the polynomials $p_1, \ldots, p_\ell$ are essentially distinct implies that $a_1, \ldots, a_{s+1}\neq 0$. Lemma \ref{L:seminorm of power} gives the inequality
	\begin{align*}
		\nnorm{f_m}_{{a_1}\be_t, \ldots, {a_{s+1}}\be_t}\ll_{a_{s+1}}\nnorm{f_m}_{{a_1}\be_t, \ldots, {a_{s}}\be_t, \be_t} = \nnorm{f_m}_{{\b_1}, \ldots, {\b_{s}}, \be_t},
	\end{align*}
	which implies Proposition \ref{P:smoothing pairwise} in this case.
	
	\smallskip
	
	\textbf{General case: $w$ is not basic.}
	
	\smallskip
	
	We prove Proposition \ref{P:smoothing pairwise} for $\eta$ of type $w$ by assuming that Proposition \ref{Host Kra characteristic pairwise independent 2} holds for $\ell-1$ as well as $\ell$ and all types $w'<w$. For simplicity of notation, we assume $m = \ell$.
	
	By Proposition \ref{strong PET bound} and the assumption that the polynomials $p_1, \ldots, p_\ell$ are essentially distinct, the vector $\b_{s+1}$ is nonzero and equals $\b_{s+1} = b_\ell \be_{\eta_\ell} - b_i \be_{\eta_i}$ for some $b_\ell, b_i\in\Z$ and $i\in\{0, \ldots, \ell-1\}$. If $i=0$, then $b_\ell \neq 0$ by the fact that $\b_{s+1}$ is nonzero, and the claim follows from the bound
	\begin{align}\label{scaling seminorm}
		\nnorm{f_\ell}_{{\b_1}, \ldots, {\b_{s}}, b_\ell \be_{\eta_\ell}}\ll_{b_\ell} \nnorm{f_\ell}_{{\b_1}, \ldots, {\b_{s}}, \be_{\eta_\ell}}
	\end{align}
	given by Lemma \ref{L:seminorm of power}.
	If $\eta_i = \eta_\ell$, then the fact $\b_{s+1} \neq \mathbf{0}$ implies that $b_\ell - b_i\neq 0$, and the claim follows once again from \eqref{scaling seminorm}. We therefore assume that $\b_{s+1} = b_\ell \be_{\eta_\ell} - b_i \be_{\eta_i}$ for some $i\in \FL$ such that $\eta_\ell\neq \eta_i$, in which case also $b_\ell, b_i\neq 0$. The proof of
	Proposition~\ref{P:smoothing pairwise} in this case follows the same two-step strategy as the proof of Proposition~\ref{smoothing of n, n^2, n^2+n}. We first obtain the control of \eqref{general average} by $\nnorm{f_i}_{{\b_1}, \ldots, {\b_{s}}, \be_{\eta_i}^{\times s_1}}$ for some $s_1\in\N$ depending only on $d,\ell,L,s$. This is accomplished by using the control by $\nnorm{f_\ell}_{{\b_1}, \ldots, {\b_{s+1}}}$, given by assumption, for an appropriately defined function $\tilde{f_\ell}$ in place of $f_\ell$. Subsequently, we repeat the procedure by applying the newly established control by $\nnorm{f_i}_{{\b_1}, \ldots, {\b_{s}}, \be_{\eta_i}^{\times s_1}}$ for a function $\tilde{f}_i$ in place of $f_i$. This gives us the claimed result.
	
	\smallskip
	
	\textbf{Step 1 (\textit{ping}): Obtaining control by a seminorm of $f_i$.}
	\smallskip
	
	Suppose that
	\begin{align}\label{general average is positive}
		\lim_{N\to\infty}\norm{\E_{n\in [N]} \,\prod_{j\in[\ell]}T_{\eta_j}^{p_j(n)}f_j \cdot \prod_{j\in[L]}\CD_{j}(q_{j}(n))}_{L^2(\mu)}>0.
	\end{align}
	Then by a straightforward variant of Proposition~\ref{P:dual replacement} we have
	\begin{align*}
		\lim_{N\to\infty}\norm{\E_{n\in [N]} \,\prod_{j\in[\ell-1]}T_{\eta_j}^{p_j(n)}f_j \cdot T_{\eta_\ell}^{p_\ell(n)}\tilde{f}_\ell \cdot \prod_{j\in[L]}\CD_{j}(q_{j}(n))}_{L^2(\mu)}>0
	\end{align*}
	for some function
	\begin{equation*}
		\tilde{f}_\ell:=\lim_{k\to\infty} \E_{n\in [N_k]}
		\, T_{\eta_\ell}^{-p_\ell(n)}g_k\cdot \prod_{j\in [\ell-1]}T_{\eta_\ell}^{-p_\ell(n)} T_{\eta_j}^{p_j(n)}\bar{f}_j \cdot \prod_{j\in[L]}T_{\eta_\ell}^{-p_\ell(n)}\overline{\CD_{j}(q_{j}(n))},
	\end{equation*}
	where the limit is a weak limit. Then our assumption gives
	$$
	\nnorm{\tilde{f}_\ell}_{{\b_1}, \ldots, {\b_{s+1}}}>0.
	$$
	By Proposition \ref{dual-difference interchange}, we get the estimate
	\begin{multline*}
		\liminf_{H\to\infty}\E_{\uh,\uh'\in [H]^{s}}\\
		\lim_{N\to\infty}\norm{\E_{n\in[N]} \, \prod_{j\in[\ell-1]} T_{\eta_j}^{p_j(n)} (\Delta_{{\b_1}, \ldots, {\b_s}; \uh-\uh'} f_j)\cdot T_{\eta_\ell}^{p_\ell(n)}u_{\uh,\uh'}\cdot \prod_{j\in [L]} \CD'_{j, \uh, \uh'}(q_{j}(n))}_{L^2(\mu)}>0,
	\end{multline*}
	where $u_{\uh,\uh'}$ are 1-bounded and invariant under $T^{\b_{s+1}}$, and for every $j\in[L]$,
	\begin{align*}
		\CD'_{j, \uh, \uh'}(n) := \Delta_{{\b_1}, \ldots, {\b_s}; \uh-\uh'}\CD_{j}(n)
	\end{align*}
	is a product of $2^s$ elements of $\FD_d$.
	
	We recall that  $b_\ell \be_{\eta_\ell} - b_i \be_{\eta_i}$ for some $i\in \FL$ such that $\eta_\ell \neq \eta_i$ and $b_\ell, b_i\neq 0$. It follows from the invariance property of $u_{\uh, \uh'}$ that
	\begin{equation}\label{invariance prop pairwise}
		T_{\eta_\ell}^{b_\ell}u_{\uh,\uh'}=T_{\eta_i}^{b_i}u_{\uh,\uh'},
	\end{equation}
	where we use the identity $T_{\eta_\ell}^{b_\ell}=T_{\eta_i}^{b_i}T^{\b_{s+1}}$.
	Using the triangle inequality we deduce that
		\begin{align*}
		&\liminf_{H\to\infty}\E_{\uh,\uh'\in [H]^{s}}\E_{r\in [b_\ell]} \\
		&
		\lim_{N\to\infty}\norm{\E_{n\in[N]} \prod_{j\in[\ell-1]} T_{\eta_j}^{p_j(b_\ell n + r)} (\Delta_{{\b_1}, \ldots, {\b_s}; \uh-\uh'} f_j)\cdot T_{\eta_\ell}^{p_\ell(b_\ell n + r)}u_{\uh,\uh'}\cdot \prod_{j\in [L]} \CD'_{j, \uh, \uh'}(q_{j}(b_\ell n+r))}_{L^2(\mu)}
	\end{align*}
is positive.
After replacing $\liminf\limits_{H\to\infty}$ with $\limsup\limits_{H\to\infty}$, using  \eqref{invariance prop pairwise}, and the pigeonhole principle, we get that for some $r_0\in [b_\ell]$, we have
	\begin{multline}\label{E:positive13}
		\limsup_{H\to\infty}\E_{\uh,\uh'\in [H]^{s}} \lim_{N\to\infty}\norm{\E_{n\in[N]} \, \prod_{j\in[\ell]} T_{\eta'_j}^{p'_j(n)} f_{j, \uh, \uh'}\cdot \prod_{j\in [L]} \CD'_{j, \uh, \uh'}(q'_{j}(n))}_{L^2(\mu)}>0,
	\end{multline}
	where
	\begin{align*}
		p'_j(n) &:= \begin{cases} p_j(b_\ell n+r_0) - p_j(r_0),\; &j\in[\ell-1]\\
			\frac{b_i}{b_\ell}(p_\ell(b_\ell n+r_0)-p_\ell(r_0)), &j = \ell,
		\end{cases}\\
		f_{j, \uh, \uh'} &:= \begin{cases} \Delta_{{\b_1}, \ldots, {\b_s}; \uh-\uh'} T_{\eta_j}^{p_j(r_0)} f_j,\; &j\in[\ell-1]\\
			T_{\eta_\ell}^{p_\ell(r_0)} u_{\uh,\uh'},\; &j=\ell,
		\end{cases}\\	
		\eta'_j &:= \begin{cases} \eta_j,\; &j\in[\ell-1]\\
			\eta_i,\; &j = \ell,
		\end{cases}
	\end{align*}
	and $q'_j(n):=q_j(b_\ell n+r_0)$ for $j\in[L]$.
	
	Note that the polynomials $p'_1, \ldots, p'_\ell$ are pairwise independent. We make a small digression here regarding the necessity of this assumption. It is crucial for us in order to be able to prove Propositions \ref{Host Kra characteristic pairwise independent 2} and \ref{P:smoothing pairwise} inductively that at each step of the induction scheme, the polynomials $p_1, \ldots, p_\ell$ are essentially distinct. While essential distinctness of $p_1, \ldots, p_\ell$ might not carry forward to the essential distinctness of $p'_1, \ldots, p'_\ell$ (because essential distinctness is not preserved under scaling), pairwise independence is preserved this way, hence we make this assumption on the polynomials $p_1, \ldots, p_\ell$ in the first place. Moreover, it does not restrict generality since it is  a necessary assumption in the statement of Theorem~\ref{T:polies0}.

	We go back to the proof now. By assumption, $w_{\eta_\ell}$ minimises $(w_t)_{t\in\supp(w)}$ and $\eta_\ell\neq\eta_i$. Moreover, the structure of the set \eqref{E:coefficientvectors} implies that $w_{\eta_i}>0$, and so $w_{\eta_i}\geq w_{\eta_\ell}$. Therefore, the type $w'=\sigma_{\eta_\ell\eta_i}w$ of the averages in  \eqref{E:positive13} is strictly smaller than $w$. We inductively apply Corollary \ref{C:Host Kra characteristic pairwise independent 2} to each average in \eqref{E:positive13}. Together with the pairwise independence of the new polynomials $p_1', \ldots, p'_\ell$, this allows us to conclude that there exist $s_1\in\N$ (depending only on $d,\ell,L,s$), $\varepsilon>0$, and a set $B\subset\N^{2s}$ of positive upper density, such that for every $(\uh, \uh')\in B$ we have
	$$\nnorm{\Delta_{{\b_1}, \ldots, {\b_s}; \uh-\uh'} f_i}_{s_1, T_{\eta_i}} \geq \varepsilon.$$
	Hence,
	\begin{align*}
		\limsup_{H\to\infty}\E_{\uh, \uh'\in[H]^s}\nnorm{\Delta_{{\b_1}, \ldots, {\b_s}; \uh-\uh'} f_i}_{s_1, T_{\eta_i}} > 0.
	\end{align*}
	Then, using Lemma \ref{difference sequences} and then \eqref{inductive formula} together with  H\"older's inequality,  we have that
	$$
	\nnorm{f_i}_{{\b_1}, \ldots, {\b_s}, \be_{\eta_i}^{\times s_1}}>0,
	$$
	and so the seminorm $\nnorm{f_i}_{{\b_1}, \ldots, {\b_s}, \be_{\eta_i}^{\times s_1}}$ controls the average \eqref{general average}.
	
	\smallskip
	
	\textbf{Step 2 (pong): Obtaining control by a seminorm of $f_\ell$.}
	
	\smallskip
	
	To get the claim that $\nnorm{f_\ell}_{{\b_1}, \ldots, {\b_s}, \be_{\eta_\ell}^{\times s'}}$ controls the average for some $s'\in\N$, we repeat the procedure once more with $f_i$ in place of $f_\ell$. From \eqref{general average is positive}  and a straightforward variant of Proposition~\ref{P:dual replacement}
	it follows that
	\begin{align*}
		\lim_{N\to\infty}\norm{\E_{n\in [N]} \,\prod_{j\in[\ell], j\neq i}T_{\eta_j}^{p_j(n)}f_j \cdot T_{\eta_i}^{p_i(n)}\tilde{f_i} \cdot \prod_{j\in [L]} \CD_{j}(q_{j}(n))}_{L^2(\mu)}>0
	\end{align*}
	for some function
	\begin{equation*}
		\tilde{f}_i:=\lim_{k\to\infty} \E_{n\in [N_k]}
		\, T_{\eta_i}^{-p_i(n)}g_k\cdot \prod_{j\in [\ell], j\neq i}T_{\eta_i}^{-p_i(n)} T_{\eta_j}^{p_j(n)}\bar{f}_j \cdot \prod_{j\in [L]} T_{\eta_i}^{-p_i(n)} \overline{\CD_{j}(q_j(n))},
	\end{equation*}
	where the limit is a weak limit. Then the previous result gives
	$$
	\nnorm{\tilde{f}_i}_{{\b_1}, \ldots, {\b_s}, \be_{\eta_i}^{\times s_1}}>0.
	$$
	Proposition \ref{dual-difference interchange} implies that
	\begin{align*}
		&\liminf_{H\to\infty}\E_{\uh,\uh'\in [H]^{s}}\\
		&\lim_{N\to\infty}\norm{\E_{n\in[N]} \, \prod_{j\in[\ell],j\neq i} T_{{\eta_j}}^{p_j(n)} (\Delta_{{\b_1}, \ldots, {\b_{s}}; \uh-\uh'} f_j)\cdot \prod_{j\in[L+1]}\CD_{j,\uh,\uh'}(q_j(n))}_{L^2(\mu)}>0,
	\end{align*}
	where
	\begin{align*}
		\CD_{j,\uh,\uh'}(n):=\begin{cases} \Delta_{{\b_1}, \ldots, {\b_s}; \uh-\uh'}\CD_{j}(n),\; &j \in[L]\\
			T_{\eta_i}^n T^{-(\b_1 h_1'+\cdots +\b_s h'_s)}\prod\limits_{\ueps\in\{0,1\}^s}\CC^{|\ueps|}\CD_{s_1, T_{\eta_i}}(\Delta_{{\b_1}, \ldots, {\b_{s}}; \uh^{\ueps}}\tilde{f}_i),\; &j=L+1.
		\end{cases}
	\end{align*}
	Thus, the sequence of functions $\CD_{j,\uh,\uh'}(n)$ is a product of $2^s$ elements of $\FD_d$ if $j\in[L]$, and it is a product of $2^{s}$ elements of $\FD_{s_1}$ for $j=L+1$. Consequently, there exist $\varepsilon>0$ and a set $B'\subset\N^{2s}$ of positive lower density, such that for every $(\uh, \uh')\in B'$, we have
	\begin{align}\label{average2step}
		\lim_{N\to\infty}\norm{\E_{n\in[N]} \, \prod_{j\in[\ell], j\neq i} T_{\eta_j}^{p_j(n)} (\Delta_{{\b_1}, \ldots, {\b_{s}}; \uh-\uh'} f_j)\cdot \prod_{j\in[L+1]}\CD_{j,\uh,\uh'}(q_j(n))}_{L^2(\mu)} > \varepsilon.
	\end{align}
	Each of the averages in \eqref{average2step} has length $\ell-1$. Applying Corollary \ref{C:Host Kra characteristic pairwise independent 2} inductively, we
        find $s'\in\N$ (depending only on $d,\ell,L,s$) and $\delta>0$ such that for every $(\uh, \uh')\in B'$, we have
	\begin{align*}
		\nnorm{\Delta_{{\b_1}, \ldots, {\b_{s}}; \uh-\uh'} f_\ell}_{s', T_{\eta_\ell}}>\delta.
	\end{align*}
	Consequently, we deduce that
	\begin{align}\label{smoothing pairwise:inductive inequality}
		\liminf_{H\to\infty}\E_{\uh, \uh'\in [H]^{s}}\nnorm{\Delta_{{\b_1}, \ldots, {\b_{s}}; \uh-\uh'} f_\ell}_{s', T_{\eta_\ell}}>0.
	\end{align}
	Using  Lemma \ref{difference sequences} and then \eqref{inductive formula} together with  H\"older's inequality,  we deduce  that $\nnorm{f_\ell}_{{\b_1}, \ldots, {\b_{s}}, \be_{\eta_\ell}^{\times s'}}>0$, as claimed.
\end{proof}

Finally, we prove Proposition \ref{Host Kra characteristic pairwise independent 2} for averages of length $\ell$ and type $w$.

\begin{proof}[Proof of Proposition \ref{Host Kra characteristic pairwise independent 2}]
	In the base case $\ell=1$, Proposition \ref{Host Kra characteristic pairwise independent 2} follows directly from Proposition \ref{strong PET bound}. We assume therefore that $\ell>1$.
	It follows from Proposition \ref{P:iterated smoothing pairwise} that there exist $s\in\N$ (depending only on $d,\ell, L$) and $m\in[\ell]$ such that \eqref{general average vanishes} holds whenever $\nnorm{f_m}_{s, T_{\eta_m}}=0$. Suppose now that
	\begin{align*}
		\lim_{N\to\infty}\norm{\E_{n\in [N]} \,\prod_{j\in[\ell]}T_{\eta_j}^{p_j(n)}f_j \prod_{j\in[L]}\CD_{j}(q_j(n))}_{L^2(\mu)} > 0.
	\end{align*}
	Applying the fact that $\nnorm{f_m}_{s, T_{\eta_m}}$ controls this average, Proposition \ref{dual decomposition}, and the pigeonhole principle, we replace  $f_m$ by a dual function of level $s$, so that we have
	\begin{align*}
		\lim_{N\to\infty}\norm{\E_{n\in [N]} \,\prod_{j\in[\ell], j\neq m}T_{\eta_j}^{p_j(n)}f_j \cdot \CD(p_m(n)) \cdot\prod_{j\in[L]}\CD_{j}(q_j(n))}_{L^2(\mu)} > 0
	\end{align*}
	for some $\CD\in\FD_s$. This average has length $\ell-1$, and we apply Proposition \ref{Host Kra characteristic pairwise independent 2} inductively to deduce that the seminorms $\nnorm{f_j}_{s_j, T_{\eta_j}}$ control \eqref{general average} for all $j\in[\ell]$ and some $s_j\in\N$ (depending only on $d,\ell, L$).
\end{proof}

\appendix

\section{Soft quantitative seminorm estimates}\label{A:soft}
\subsection{Main result}
For the purposes of this section only,  {\em a system} $(X, \CX, \mu, T_1,\ldots, T_\ell)$ is a collection of (not necessarily commuting) measure preserving transformations  $T_1,\ldots, T_\ell$ acting on a probability space $(X,\CX,\mu) $. We are interested in mean convergence properties of the averages
\begin{equation}\label{E:averages}
	\E_{n\in[N]}  \, T_1^{a_1(n)}f_1\cdots T_\ell^{a_\ell(n)}f_\ell\,
\end{equation}
where $f_1,\ldots, f_\ell\in L^\infty(\mu)$.

If $X$ is a normed space, we denote its closed unit ball with   $B_X$.

\begin{definition} Let $(X, \CX, \mu, T_1,\ldots, T_\ell)$ be a system,  $a_1,\ldots, a_\ell\colon \N\to \Z$ be sequences,    $\nnorm{\cdot}$ be a seminorm on $L^\infty(\mu)$,  and  $Y_1,\ldots, Y_\ell$ be subspaces of $L^\infty(\mu)$.
	\begin{itemize}
		\item
		We say that  for some $m\in[\ell]$ the seminorm $\nnorm{f_m}$ {\em controls  the averages \eqref{E:averages}
			for  the subspaces $Y_1,\ldots, Y_\ell$ (and the given system),}
		if    $\nnorm{f_m}=0$ implies that the averages \eqref{E:averages} converge to $0$ in $L^2(\mu)$ for every $f_1\in Y_1, \ldots, f_\ell \in Y_\ell$.

		\item We say that the  collection $a_1,\ldots, a_\ell$ is  {\em good for mean convergence for the subspaces  $Y_1,\ldots, Y_\ell$ (and the given system),} if the averages \eqref{E:averages} converge in $L^2(\mu)$ for all $f_1\in Y_1, \ldots, f_\ell\in Y_\ell$.
	\end{itemize}
\end{definition}

The next  result deals with averages that are controlled by the seminorms $\nnorm{\cdot}_{s}$ and is used in our seminorm smoothing argument in Section \ref{S:smoothing}. 	
\begin{proposition}\label{P:Us}
	Let $(X, \CX, \mu, T_1,\ldots, T_\ell)$ be a system and $\CY_1, \ldots, \CY_\ell\subseteq\CX$ be sub-$\sigma$-algebras.	Suppose that the sequences   $a_1,\ldots, a_\ell\colon \N\to \Z$  are good for mean convergence for the subspaces $L^\infty(\CY_1, \mu), \ldots, L^\infty(\CY_\ell, \mu)$. Suppose moreover that for the same subspaces the seminorm $\nnorm{f_m}_{s}$ controls   the averages \eqref{E:averages} for some $m\in[\ell]$ and $s\in \N$.
	Then for  every $\varepsilon>0$
	there exists  $\delta>0$ 
	such that if
	$f_1 \in L^\infty(\CY_1, \mu),\ldots, f_\ell \in L^\infty(\CY_\ell, \mu)$  are $1$-bounded and  $\nnorm{f_m}_{s}\leq \delta$, then
	$$
	\lim_{N\to\infty} \norm{\E_{n\in[N]} \,  T_1^{a_1(n)}f_1\cdots T_\ell^{a_\ell(n)}f_\ell}_{L^2(\mu)}\leq \varepsilon.
	$$
\end{proposition}
\begin{remarks}
	
	$\bullet$ The important point  is that the choice $\delta$ does not depend on the $1$-bounded functions $f_1,\ldots, f_\ell$ (but we do not claim  independence of  the system or the subspaces).
	
	$\bullet$	It is possible to prove variants of these results that do not have mean convergence assumptions, but then the conclusion will be a  subsequential weak convergence  property.
	
	$\bullet$ Using straightforward modifications, similar variants can be proved for other averaging schemes (averages along specific F\o lner sequences,  weighted,  logarithmic, etc).
	
	$\bullet$ In applications, the $\sigma$-algebras $\CY_1, \ldots, \CY_\ell$ are factors of $T_1, \ldots, T_\ell$, but the result does not require this assumption.
\end{remarks}

\subsection{A more abstract statement}
We are going to deduce Propositions~\ref{P:Us} from a more abstract statement (see Proposition~\ref{P:abstract}) that deals with continuous multilinear functionals on reflexive Banach spaces.
\begin{definition}
	Let $(X_j,\norm{\cdot}_{X_j})$ be a normed space for every $j\in[\ell]$.
	\begin{itemize}
		\item  On the product space    $X_1\times \cdots \times X_\ell$, we always consider the norm $\norm{\cdot}_{X_1\times \cdots \times X_\ell}$ defined by
		$$
		\norm{(x_1,\ldots, x_\ell)}_{X_1\times \cdots \times X_\ell}:=\norm{x_1}_{X_1}+\cdots +\norm{x_\ell}_{X_\ell}.
		$$
		
		\item A mapping $A\colon X_1\times \cdots \times X_\ell \to \R$ is called a {\em multilinear functional}, if it is linear with respect to each of its $\ell$ coordinates when the other $\ell-1$ coordinates are fixed.
	\end{itemize}
\end{definition}

It is easy to verify that a multilinear functional $A\colon X_1\times \cdots \times X_\ell \to \R$ is continuous if and only if there exists $C>0$ such that
\begin{equation}\label{E:multi}
	|A(x_1,\ldots, x_\ell)|\leq C \prod_{j\in[\ell]} \norm{x_j}_{X_j}
\end{equation}
for all $(x_1,\ldots, x_\ell)\in X_1\times \cdots \times X_\ell$.
The forward implication  is immediate since continuity at $0$ implies $1$-boundedness on some ball centered at the origin, and then multilinearity easily implies the stated estimate. One gets the reverse implication by  using a telescoping argument. We start by noting  that
$$
|A(y_1,\ldots, y_\ell)-A(x_1,\ldots, x_\ell)|\leq \sum_{j=1}^{\ell}
| A(x_1,\ldots,  x_{j-1},y_j,\ldots, y_\ell)-A(x_1,\ldots,  x_j,y_{j+1},\ldots, y_\ell)|.
$$
Using multilinearity  we get that the right hand side is equal to
$$
\sum_{j=1}^{\ell}
| A(x_1,\ldots,  x_{j-1},y_j-x_j,\ldots, y_\ell)|,
$$
and for $x_j,y_j\in B_{X_j}$, $j\in[\ell],$ this can be  bounded using \eqref{E:multi} by
$$
C\sum_{j=1}^{\ell} \norm{y_j-x_j}_{X_j} .
$$
For future use, we record the following uniform continuity property of continuous multilinear functionals that follows from the previous discussion: for every $(x_1,\ldots, x_{\ell-1})\in X_1\times \cdots \times X_{\ell-1}$, we have
\begin{equation}\label{E:continuous}
	\lim_{(y_1,\ldots, y_{\ell-1})\to (x_1,\ldots, x_{\ell-1}) }\sup_{x_\ell\in B_{X_\ell}}|A(y_1,\ldots, y_{\ell-1},x_\ell)-A(x_1,\ldots, x_{\ell-1},x_\ell)|=0
\end{equation}
where convergence is considered with the norm $\norm{\cdot}_{X_1 \times \cdots \times X_{\ell-1}}$
on $X_1 \times \cdots \times X_{\ell-1}$.

\begin{definition} Let $\nnorm{\cdot}_1$ be a seminorm on a linear space $X$ and  $\nnorm{\cdot}_2$ be a seminorm on a subspace $X'$ of $X$. Let also   $X_0\subseteq X'$ be a subset.  We say that
	\begin{itemize}
		\item  $\nnorm{\cdot}_2$  {\em controls the  convergence of  $\nnorm{\cdot}_1$ on $X_0$}
			if whenever $x_n,x\in X_0$, $n\in\N$,  and $\lim_{n\to\infty} \nnorm{x_n-x}_2=0$, we have  $\lim_{n\to\infty} \nnorm{x_n-x}_1=0$.
			
			\item    $\nnorm{\cdot}_2$  {\em  dominates  $\nnorm{\cdot}_1$ on $X_0$} if    for some  $C>0$, we have $\nnorm{x}_1\leq C \nnorm{x}_2$ for every $x\in X_0$.
		\end{itemize}
	\end{definition}

	The next result enables us to infer quantitative information from purely qualitative input.  The subspaces $X'_j$ are introduced for purely technical reasons (see the first remark following the theorem), and on a first reading one is advised to  assume that the spaces $X'_1 = \cdots = X'_\ell =X_1 =\cdots = X_\ell$ are all the same.
	
	\begin{proposition}\label{P:abstract}
		Let $(X_j,\norm{\cdot}_{X_j})$, $j\in [\ell]$,  be  separable, reflexive,  Banach spaces, and  $A\colon X_1\times\cdots \times  X_\ell\to \R$ be  a continuous multilinear  functional. For $j\in [\ell]$, let  $X_j'$ be  $\norm{\cdot}_{X_j}$-dense subspaces of $X_j$ and
		$\norm{\cdot}_{X_j'}$ be a norm on  $X_j'$ that  dominates the norm $\norm{\cdot}_{X_j}$ on $X'_j$ and  such that $B_{X_j'}$ is $\norm{\cdot}_{X_j}$-closed.\footnote{Throughout this statement $B_{X_j'}$ is defined using the norm $\norm{\cdot}_{X_j'}$.}
		Let also  $m\in [\ell]$ and   $\nnorm{\cdot}$ be a  seminorm on $X_m'$, such that the norm $\norm{\cdot}_{X_m}$
		controls the convergence of the seminorm $\nnorm{\cdot}$ on $B_{X_m'}$.  Then the following two statements are equivalent:
		\begin{enumerate}
			\item (Qualitative property) Whenever  $x_m\in X_m'$ satisfies $\nnorm{x_m} =0$, we have   for all  $x_j\in X_j'$, $j\in [\ell]$, $j\neq m$,   that
			$$
			A(x_1,\ldots, ,x_\ell)=0.
			$$

			\item (Quantitative property) For every $\varepsilon>0$ there exists $\delta>0$ such that if some    $x_m\in B_{X_m'}$ satisfies    $\nnorm{x_m}\leq \delta $, then
			$$
			\sup_{x_j\in B_{X_j}, j\in [\ell], j\neq m} |A(x_1,\ldots,x_\ell)|\leq \varepsilon.
			$$
		\end{enumerate}
	\end{proposition}
	
	\begin{remarks}
		
		$\bullet$ In the proof of Proposition \ref{P:Us}, we will use this result for $X_j:=L^{p}(\CY_j, \mu)$,  $X'_j:=L^\infty(\CY_j, \mu)$, $\norm{\cdot}_{X_j}:=\norm{\cdot}_{L^p(\mu)}$, $\norm{\cdot}_{X'_j}:=\norm{\cdot}_{L^\infty(\mu)}$, $j\in[\ell]$,
		$\nnorm{\cdot}:=\nnorm{\cdot}_{s}$,  for appropriate values of  $p\in (1,+\infty)$ and $s\in \N$, the functional
		$$A(f_0,\ldots, f_\ell):=\lim_{N\to\infty}\E_{n\in [N]}\int f_0\cdot T_1^{a_1(n)}f_1\cdots
		T_\ell^{a_\ell(n)}f_\ell\, d\mu$$ and $\sigma$-algebras $\CY_1, \ldots, \CY_\ell$ (which in our applications in Section \ref{S:smoothing} are factors of $T_1, \ldots, T_\ell$).  In our applications, it will be possible to show that there exists $s\in \N$ such that if $f_0\in L^\infty(\CX, \mu)$, $f_j\in L^\infty(\CY_j, \mu)$ for every $j\in[\ell]$ and additionally $\nnorm{f_m}_s=0$, then $A(f_0,\ldots, f_\ell)=0$. Establishing this implication for functions $f_j\in L^p(\CY_j, \mu)$, $j\in [\ell]$, for some appropriate $p\in \N$ is strictly harder and cannot be derived using an approximation argument from $L^\infty(\mu)$ functions. This is the reason why we introduce the auxiliary spaces $X'_1, \ldots, X'_\ell$ in the statement.

		
		$\bullet$ 	The quantity $\delta$ is going to depend on the functional $A$, spaces $X_j, X_j'$, index $\ell$, norms $\norm{\cdot}_{X_j}, \norm{\cdot}_{X'_j}$, seminorm $\nnorm{\cdot}$, and the variable $\varepsilon$. What is important for our applications is that it does not depend on elements $x_j\in B_{X_j}$ for $j\in [\ell], j\neq m$ or $x_m\in B_{X'_m}$ (as long as $\nnorm{x_m}\leq \delta $).
	\end{remarks}
	
	The proof of Proposition~\ref{P:abstract} uses  two classical  Banach space results:  Mazur's lemma (see next subsection) is used to get the needed uniformity with respect to the variable $x_m$, and then the Baire category theorem is used to get uniformity with respect to the other variables.

	\subsection{Mazur's lemma}
	If $A$ is a subset of a linear space we denote with $\text{co}(A)$ the set of (finite) convex combinations of elements of $A$.
	The next is a  standard fact from Banach space theory.
	\begin{lemma}[Mazur's Lemma]
		Let $X$ be a normed \ space and suppose that the sequence $(x_n)$, of elements of $X$,   converges weakly to some $x\in X$. Then there exist $y_n\in \text{co}(x_1,x_2,\ldots)$, $n\in \N$, such that the sequence  $(y_n)$ converges strongly  to $x$.
	\end{lemma}
	We are going to use the following consequence.
	\begin{corollary}\label{C:strong}
		Let $X$ be a reflexive Banach space and suppose  that $(x_n)$ is a bounded sequence of elements of $X$. Then there exist  $y_n\in \text{co}(x_{n+1},x_{n+2},\ldots)$, $n\in \N$,  and $x\in X$, such that the sequence  $(y_n)$ converges strongly to $x$.
	\end{corollary}
	\begin{remarks}
		$\bullet$	We are going to use this for $X:=L^p(\mu)$ for appropriate values of  $p\in \N$.
		
		$\bullet$	If $X$ is a Hilbert space, more is true and rather easy to prove. If $x_n$ converges to $x_0$ weakly, then  there exists a subsequence $k_n\to\infty$ such that $\frac{x_{k_1}+\cdots +x_{k_n}}{n} \to x_0$ strongly. To see this we can assume that $x_0=0$ and $\norm{x_n}\leq 1$ for every $n\in \N$. Since $\langle x_n,x\rangle\to 0$ for every $x\in X$,  we can choose  inductively $k_n$ such that $|\langle x_{k_n},x_{k_i}\rangle|\leq \frac{1}{2^n}$, $i=1,\ldots,n-1$ for every $n\in \N$. Then  $\norm{\frac{x_{k_1}+\cdots +x_{k_n}}{n}}^2\leq \frac{3}{n}$ for every $n\in\N$. A result of Banach-Saks claims that a similar result is true when  $X=L^p(\mu)$ is separable (which is the case since we work with Lebesgue spaces) for any $p\in (1,+\infty)$ but this is much harder to prove (see for example page 101 in \cite{Wo}).
	\end{remarks}
	\begin{proof}
		Since $X$ is reflexive, its unit ball is  weakly compact,  hence the sequence $(x_n)$ has a subsequence $(x_{n_k})$ that converges weakly to some $x\in X$.
		Since the sequence $x_{n_1},x_{n_2},\ldots $  converges weakly to $x$ the same holds for the sequence $x_{n_k},x_{n_{k+1}},\ldots $ for every $k\in \N$.
		Hence, by the previous lemma, for every $k\in \N$ there exists $y_k\in  \text{co}(x_{n_k},x_{n_{k+1}},\ldots)$ that satisfies $\norm{x-y_k}\leq 1/k$.  Since $n_k\geq k$ the conclusion follows.
	\end{proof}
	
	\subsection{Proof of Proposition~\ref{P:abstract}}
	The reverse implication is trivial, so we are going to prove the forward implication. We assume   property~(i), and  we are going to
	establish property~(ii). Without loss of generality, we can assume that $m=\ell$.
	
	\medskip
	
	{\bf First Step.} (Uniformity with respect to $x_\ell$)
	We claim that under the stated assumptions the following holds:  if
	$(x_1,\ldots, x_{\ell-1})\in X_1'\times \cdots \times X_{\ell-1}'$  are fixed,
	then 	there exists $\delta>0$ (depending on $x_1, \ldots, x_{\ell-1}$ but not on $x_\ell$)
	such that if  $x_\ell\in B_{X'_\ell}$ satisfies  $\nnorm{x_\ell}\leq \delta$, then   $|A(x_1,\ldots, x_\ell)|\leq \varepsilon$.
	By considering the element $-x_\ell$ in place of $x_\ell$  and using that  $A(x_1,\ldots,  -x_\ell)=-A(x_1, \ldots, x_\ell)$, it suffices to verify that  $A(x_1,\ldots, x_\ell)\leq \varepsilon$.

	Arguing by contradiction, suppose that there   exist $\varepsilon_0>0$ and   elements  $x_{\ell,k}\in B_{X'_\ell}$, $k\in \N$,  such that
	\begin{equation}\label{E:norm0gen}
		\lim_{k\to\infty}	\nnorm{x_{\ell,k}}=0
	\end{equation}
	and
	\begin{equation}\label{E:epsilongen}
		A(x_1, \ldots, x_{\ell-1},x_{\ell,k})  \geq \varepsilon_0, \quad k\in \N.
	\end{equation}
	
	Since $\norm{\cdot}_{X'_\ell}$  dominates the norm $\norm{\cdot}_{X_\ell}$ on $X'_\ell$, the sequence $(x_{\ell,k})_{k\in\N}$ is $\norm{\cdot}_{X_\ell}$-bounded.
	Since $X_\ell$ is a reflexive  Banach space, by Corollary~\ref{C:strong}, there exist $x_\ell\in X_\ell$, sequence of integers   $J_k\to \infty$, and  $c_{1,k},\ldots, c_{J_k,k}\in [0,1]$,  $k\in \N$, with sum $1$, such that
	\begin{equation}\label{E:jkgen}
		\lim_{k\to\infty}\norm{\sum_{j=1}^{J_k}\, c_{j,k}\, x_{\ell, k+j}-x_\ell}_{X_\ell}=0.
	\end{equation}
	Furthermore, since  $x_{\ell,k}\in B_{X'_\ell}$ for all  $k\in \N$,  $c_{j,k}\geq 0$  and $\sum_{j=1}^{J_k}\, c_{j,k}=1$, we have that $\sum_{j=1}^{J_k}\, c_{j,k}x_{\ell, k+j}\in B_{X'_\ell}$ for all $k\in \N$, hence using our hypothesis that $B_{X'_\ell}$ is $\norm{\cdot}_{X_\ell}$-closed  and identity \eqref{E:jkgen}, we get that $x_\ell\in B_{X'_\ell}$.
	We claim that this element  $x_\ell$ contradicts our assumptions.
	
	On the one hand,
	\eqref{E:epsilongen} combined with $c_{j,k}\geq 0$ and $\sum_{j=1}^{J_k}\, c_{j,k}=1$ and the linearity of $A$ with respect to the variable $x_\ell$, give
	\begin{equation}\label{E:e0}
		A\big(x_1,\ldots,x_{\ell-1},\sum_{j=1}^{J_k}\, c_{j,k}\, x_{\ell,k+j}\big)=
		\sum_{j=1}^{J_k}\, c_{j,k}\,  A(x_1,\ldots,x_{\ell-1}, x_{\ell,k+j})
		\geq \varepsilon_0, \quad k\in \N.
	\end{equation}
	Using the $\norm{\cdot}_{X_\ell}$-continuity of $A$ with respect to the variable $x_\ell$ and  equations \eqref{E:jkgen} and \eqref{E:e0}, we deduce
	$$
	A(x_1,\ldots, x_\ell)\geq \varepsilon_0.
	$$
	
	On the other hand, we have
	$$
	\nnorm{x_\ell}
	\leq \limsup_{k\to\infty}\nnorm{\sum_{j=1}^{J_k}\, c_{j,k}\, x_{\ell, k+j}}\leq
	\limsup_{k\to\infty}\sum_{j=1}^{J_k}\, c_{j,k}\, \nnorm{x_{\ell, k+j}}\leq
	\limsup_{k\to\infty}\sup_{j\in\N}\nnorm{x_{\ell, k+j}}=0,
	$$
	where the first estimate follows from \eqref{E:jkgen} and  our assumption that    the norm $\norm{\cdot}_{X_\ell}$  controls the convergence of the seminorm  $\nnorm{\cdot}$ on $B_{X'_\ell}$, the second estimate holds since $c_{j,k}\geq 0$ and $\nnorm{\cdot}$ is a seminorm, the  third estimate holds since $c_{j,k}\geq 0$ and  $\sum_{j=1}^{J_k}\, c_{j,k}=1$,
	and the last identity because of \eqref{E:norm0gen}. Hence, 	$\nnorm{x_\ell}=0$.
	
	Combining the  previous two facts, we have that
	$\nnorm{x_\ell}=0$  but 	$A(x_1,\ldots, x_\ell)\neq 0$.
	This contradicts property~(i) and concludes the proof of  the claim.
	
	\medskip
	
	{\bf Second Step.} (Uniformity with respect to $x_1,\ldots, x_{\ell-1}$, part I)  We plan to   get some uniformity with respect to the variables $x_1,\ldots, x_{\ell-1}$ on some non-empty open subset of $X_1 \times \cdots \times X_{\ell-1}$ using the Baire category theorem.
	
	For each $j\in[\ell-1]$, let $Y_j\subset X'_j$ be a countable dense set in $(X_j,\norm{\cdot}_{X_j})$; such a set can be found since,
	by assumption, $X'_j$ is $\norm{\cdot}_{X_j}$-dense in $X_j$ and
	$(X'_j,\norm{\cdot}_{X_j})$ is separable (since $(X_j,\norm{\cdot}_{X_j})$   is separable by assumption).
	Let also $\varepsilon>0$. Using the first step we get that  for every $y=(y_j)_{j\in[\ell-1]}\in Y_1 \times \cdots \times Y_{\ell-1}$ there exists $\delta_y=\delta_y(\varepsilon)$, such that if  $x_\ell\in B_{X'_\ell}$  satisfies  $\nnorm{x_\ell}\leq \delta_y$, then
	$$
	A(y_1,\ldots, y_{\ell-1}, x_\ell)\leq \varepsilon/2.
	$$
	Let now $(x_1,\ldots, x_{\ell-1})\in X_1\times \cdots \times X_{\ell-1}$  be arbitrary.  Since
	$ Y_1 \times \cdots \times Y_{\ell-1}$ is $\norm{\cdot}_{ X_1 \times \cdots \times X_{\ell-1}}$-dense in $X_1 \times \cdots \times X_{\ell-1}$ and
	$A$ satisfies \eqref{E:continuous} (since it is multilinear and continuous), there exists
	$y=(y_j)_{j\in[\ell-1]}\in Y_1 \times \cdots \times Y_{\ell-1}$   such that
	$$
	\sup_{x_\ell\in B_{X'_\ell}}|A(y_1,\ldots, y_{\ell-1}, x_\ell)-A(x_1,\ldots,x_{\ell-1},x_\ell)|\leq \varepsilon/2.
	$$
	Combining the above, we get that for every $(x_j)_{j\in[\ell-1]}\in X_1 \times \cdots \times X_{\ell-1}$ there exists 	$y=(y_j)_{j\in[\ell-1]}\in Y_1 \times \cdots \times Y_{\ell-1}$ (depending on $(x_j)_{j\in[\ell-1]}$ and on $\varepsilon$) such that the following holds: for all  $x_\ell\in B_{X'_\ell}$ with $\nnorm{x_\ell}\leq \delta_y$, we have
	\begin{equation}\label{E:gfegen}
		|A(x_1,\ldots, x_{\ell-1}, x_\ell)|\leq \varepsilon.
	\end{equation}

	For $y\in Y_1 \times \cdots \times Y_{\ell-1}$, we define the subset $F_y=F_y(\varepsilon)$ of $X_1 \times \cdots \times X_{\ell-1}$ by
	\begin{equation}\label{E:Fy}
		F_y:= \{(x_j)_{j\in[\ell-1]} \in X_1 \times \cdots \times X_{\ell-1}\colon \eqref{E:gfegen} \text{ holds for all }  x_\ell\in B_{X'_\ell}  \text{ with }  \nnorm{x_\ell}\leq \delta_y  \}.
	\end{equation}
	Since the multilinear functional  $A$ is continuous, we deduce that  for every $y\in Y_1 \times \cdots \times Y_{\ell-1}$, the subset $F_y$  of  $ X_1 \times \cdots \times X_{\ell-1}$ is $\norm{\cdot}_{X_1 \times \cdots \times X_{\ell-1}}$-closed and by our previous discussion, we have
	$$
	X_1 \times \cdots \times X_{\ell-1}=\bigcup_{y\in Y_1 \times \cdots \times Y_{\ell-1}} F_y.
	$$
	Since the set $ Y_1 \times \cdots \times Y_{\ell-1}$ is countable and the space $X_1 \times \cdots \times X_{\ell-1}$ is $\norm{\cdot}_{ X_1 \times \cdots \times X_{\ell-1}}$-complete,
	the Baire category theorem gives that some $F_y$ has non-empty interior with respect to the norm $\norm{\cdot}_{ X_1 \times \cdots \times X_{\ell-1}}$. Hence, there exist $y_0=y_0(\varepsilon)\in Y_1 \times \cdots \times Y_{\ell-1}$, $x'_1=x'_1(\varepsilon)\in X_1, \ldots, x'_{\ell-1}=x'_{\ell-1}(\varepsilon)\in X_{\ell-1}$ and $r_0=r_0(\varepsilon)>0$ such that
	\begin{equation}\label{E:Fy0}
		\prod_{j\in[\ell-1]} D_{X_j}(x'_j,2r_0) \subset F_{y_0}
	\end{equation}
	where $D_{X_j}(x_0,r):=\{x\in X_j\colon \norm{x-x_0}_{X_j}< r\}$.
	
	\medskip
	
	{\bf Third Step.} (Uniformity with respect to $x_1,\ldots, x_{\ell-1}$, part II)
	We are going to use the multilinearity of the expression $A(x_1,\ldots,x_\ell)$
	in order to transfer the  estimate \eqref{E:gfegen} from the open subset of $ X_1 \times \cdots \times X_{\ell-1}$ given in \eqref{E:Fy0} to $B_{ X_1} \times \cdots \times B_{X_{\ell-1}}$.
	
	Let  $(x_1, \ldots, x_{\ell-1})\in  B_{X_1}\times \cdots \times B_{X_{\ell-1}}$. Then by \eqref{E:Fy0}, we have
	$$
	(x'_j+r_0x_j)_{j\in[\ell-1]}\in 	\prod_{j\in[\ell-1]}	D(x'_j,2r_0)\subset F_{y_0},
	$$
	hence by \eqref{E:Fy}, we get that
	\begin{equation}\label{E:x'1}
		|A(x'_1+r_0x_1,\ldots,   x'_{\ell-1}+r_0x_{\ell-1},x_\ell)|\leq \varepsilon
	\end{equation}
	holds for all $(x_1, \ldots, x_{\ell-1})\in  B_{X_1}\times \cdots \times B_{X_{\ell-1}}$,   and $x_\ell\in B_{X'_\ell}$ with $\nnorm{x_\ell}\leq \delta_{y_0}$. For each $j\in[\ell-1]$, we split $x_j = \frac{1}{r_0}((x'_j + r_0 x_j)-x'_j)$, and we use multilinearity to decompose $A(x_1,\ldots,x_\ell)$ into $2^{\ell-1}$ terms, so that
	\begin{align}\label{E:A decomposed}
	    |A(x_1,\ldots,x_\ell)|\leq \frac{1}{r_0^{\ell-1}}\sum_{\ueps\in\{0,1\}^{\ell-1}}|A(x_1' + \eps_1 r_0 x_1, \ldots, x'_{\ell-1} + \eps_{\ell-1}r_0x_{\ell-1}, x_\ell)|,
	\end{align}
where $\ueps=(\eps_1,\ldots, \eps_{\ell-1})$.
	Recalling that $x'_j, x'_j + r_0 x_j\in D(x'_j, 2r_0)$, we use \eqref{E:x'1} and \eqref{E:A decomposed} to conclude that
	$$
	|A(x_1,\ldots, x_\ell)|\leq (2/r_0)^{\ell-1}\varepsilon
	$$
	for all $(x_1, \ldots, x_{\ell-1})\in  B_{X_1}\times \cdots \times B_{X_{\ell-1}}$ and $x_\ell\in B_{X'_\ell}$ with $\nnorm{x_\ell}\leq \delta_{y_0}$.
	Hence, if  $\nnorm{x_\ell}\leq (r_0/2)^{\ell-1}\delta_{y_0}$, we have   $\nnorm{(2/r_0)^{\ell-1} x_\ell}\leq \delta_{y_0}$, and the previous estimate implies
	$$
	|A(x_1,\ldots, x_{\ell-1}, (2/r_0)^{\ell-1} x_\ell)|
	\leq(2/r_0)^{\ell-1}\varepsilon
	$$
	for all $(x_1, \ldots, x_{\ell-1})\in  B_{X_1}\times \cdots \times B_{X_{\ell-1}}$.
	We deduce from this and the linearity of the expression $A(x_1,\ldots, x_\ell)$ with respect to the variable $x_\ell$,  that  for $\delta:= \delta_{y_0} (r_0/2)^{\ell-1}$, the inequality
	$$
	|A(x_1,\ldots, x_\ell)|
	\leq \varepsilon
	$$
	holds for all $(x_1, \ldots, x_{\ell-1})\in  B_{X_1}\times \cdots \times B_{X_{\ell-1}}$ and $x_\ell\in B_{X'_\ell}$ with
	$\nnorm{x_\ell}\leq \delta$.
	This completes the proof of Proposition~\ref{P:abstract}.
	\subsection{Proof of the main result} We are going to use Proposition~\ref{P:abstract} and some pretty standard maneuvers in order to deduce Proposition~\ref{P:Us}.
	\begin{proof}[Proof  of Proposition~\ref{P:Us}]
		We can assume that all functions are real valued and  $m=\ell$.
		
		We plan to apply Proposition~\ref{P:abstract}. For $j\in[\ell]$, we let
		\begin{gather*}
			X_j:=L^{\ell+1}(\CY_j, \mu),\quad \norm{\cdot}_{X_j}:=\norm{\cdot}_{L^{\ell+1}(\mu)},\\ X'_j:=L^\infty(\CY_j, \mu),\quad \norm{\cdot}_{X'_j}:=\norm{\cdot}_{L^\infty(\mu)}, \quad \nnorm{\cdot}:=\nnorm{\cdot}_{s}.		
		\end{gather*}
		
		Fix $j\in[\ell]$ and recall the following facts:
		\begin{itemize}
			\item $(X_j, \norm{\cdot}_{X_j}) = (L^{\ell+1}(\CY_j, \mu), \norm{\cdot}_{L^{\ell+1}(\mu)})$ is a separable, reflexive Banach space;
			\item $\norm{\cdot}_{L^{\ell+1}(\mu)}\leq \norm{\cdot}_{L^{\infty}(\mu)}$, hence $\norm{\cdot}_{L^\infty(\mu)}$ dominates $\norm{\cdot}_{L^{\ell+1}(\mu)}$ on $L^\infty(\CY_j, \mu)$;
			\item the norm  $\norm{\cdot}_{L^{\ell+1}(\mu)}$ controls the convergence of $\nnorm{\cdot}_{s}$ on $B_{L^\infty(\CY_j, \mu)}$;\footnote{It suffices to show that if $g_n\in L^\infty(\mu)$, $n\in\N$,  are $1$-bounded and $\lim_{n\to\infty}\norm{g_n}_{L^{\ell+1}(\mu)}=0$, then also $\lim_{n\to\infty}\norm{g_n}_{s}=0$. If $\ell+1\geq 2^s$ this follows from $\norm{g_n}_{s}\leq \norm{g_n}_{L^{2^s}(\mu)}\leq \norm{g_n}_{L^{\ell+1}(\mu)}$. If $\ell+1< 2^s$, then for $1$-bounded functions $g_n$, we have $\norm{g_n}_{s}^{2^s}\leq \norm{g_n}_{L^1(\mu)}\leq \norm{g_n}_{L^{\ell+1}(\mu)}$.}
			\item the space $L^\infty(\CY_j,\mu)$ is  $L^{\ell+1}(\mu)$-dense in $L^{\ell+1}(\CY_j, \mu)$, and  $B_{L^{\infty}(\CY_j, \mu)}$ is $\norm{\cdot}_{L^{\ell+1}(\mu)}$-closed.
		\end{itemize}

			We define the multilinear functional $A\colon X_0\times \cdots \times X_\ell\to \R$, where $X_0:=L^{\ell+1}(\CX, \mu)$ and $X'_0:= L^\infty(\CX, \mu)$,  by the formula
			$$
			A(f_0,f_1, \ldots, f_\ell):=\lim_{N\to\infty}\E_{n\in[N]}\int f_0\cdot  T_1^{a_1(n)}f_1\cdots T_\ell^{a_\ell(n)}f_\ell\, d\mu.
			$$
			Note that by our good convergence assumption, the previous limit exists whenever $(f_0, \ldots, f_\ell)\in X'_0\times \cdots \times X'_\ell$, i.e. functions $f_0, \ldots, f_\ell\in L^{\infty}(\mu)$ measurable with regards to $\CX, \CY_1, \ldots, \CY_\ell$ respectively. Approximating elements of $X_0\times \cdots \times X_\ell$ by elements of $X'_0\times \cdots \times X'_\ell$ in $\norm{\cdot}_{X_0\times\cdots\times X_\ell}$, we deduce that the limit converges for all   $(f_0, \ldots, f_\ell)\in X_0\times \cdots \times X_\ell$.

			Using  H\"older's inequality and the fact  that
			the transformations $T_1,\ldots, T_\ell$ preserve the measure $\mu$, we get
			$$
			|A(f_0,\ldots, f_\ell)|\leq \prod_{j=0}^\ell \norm{f_j}_{L^{\ell+1}(\mu)}
			$$
			for all $(f_0,\ldots, f_\ell)\in X_0\times \cdots \times X_\ell$.
			Hence, $A$ is continuous.
			Note also that by assumption, the seminorm $\nnorm{f_\ell}_{s}$ controls  the averages \eqref{E:averages} for the subspaces $X_1, \ldots, X_\ell$, hence, using the Cauchy-Schwarz  inequality we get that if  $f_0,\ldots, f_\ell\in L^\infty(\mu)$ are measurable with regards to $\CX, \CY_1, \ldots, \CY_\ell$ respectively  and  $\nnorm{f_\ell}_{s}=0$, then $A(f_0,\ldots, f_\ell)=0$.
			
			From the previous discussion we get that the assumptions and property~(i) of Proposition~\ref{P:abstract} are satisfied and we deduce that property~(ii) holds, namely:
			
			For every $\varepsilon>0$ there exists $\delta>0$
			such that if
			$f_0,\ldots, f_\ell\in L^{\infty}(\mu)$ are $1$-bounded and measurable with regards to $\CX, \CY_1, \ldots, \CY_\ell$ respectively, and additionally if $\nnorm{f_\ell}_{s}\leq \delta$, then
			$$
			\lim_{N\to\infty}\Big|\E_{n\in[N]}\int f_0\cdot  T_1^{a_1(n)}f_1\cdots T_\ell^{a_\ell(n)}f_\ell\, d\mu \Big|\leq \varepsilon.
			$$
			Letting  $f_0$ be the complex conjugate of the function $\lim_{N\to\infty}\E_{n\in[N]} \, T_1^{a_1(n)}f_1\cdots T_\ell^{a_\ell(n)}f_\ell$ where the limit is taken in $L^2(\mu)$
			(the limit exists by our good convergence  assumption and is 1-bounded by the 1-boundedness of $f_1, \ldots, f_\ell$),
			we get that
			$$
			\lim_{N\to\infty}	\norm{\E_{n\in[N]} \,  T_1^{a_1(n)}f_1\cdots T_\ell^{a_\ell(n)}f_\ell}_{L^2(\mu)}\leq \varepsilon.
			$$
			This completes the proof.
		\end{proof}

		\subsection{Further extensions}
		We record here an extension of Proposition~\ref{P:abstract}  that can be proved with minor changes.
		The application we have in mind is given in Proposition~\ref{P:Krat}. Although not used in this article, it may be useful in other contexts.

		\begin{definition}
			We say that  a sequence of seminorms $\nnorm{\cdot}_{(d)}$, $d\in \N$, on a linear space $X$  is {\em increasing}, if
			$\nnorm{x}_{(d)} \leq \nnorm{x}_{(d+1)}$ for every $d\in \N$ and $x\in X$.
		\end{definition}

		\begin{proposition}\label{P:abstract-general}
			Let $(X_j,\norm{\cdot}_{X_j})$, $j\in [\ell]$,  be  separable, reflexive,  Banach spaces, and  $A\colon X_1\times\cdots \times  X_\ell\to \R$ be  a continuous multilinear  functional. For $j\in [\ell]$, let  $X_j'$ be  $\norm{\cdot}_{X_j}$-dense subspaces of $X_j$ and
			$\norm{\cdot}_{X_j'}$ be a norm on  $X_j'$ that  dominates the norm $\norm{\cdot}_{X_j}$ on $X'_j$ and  such that $B_{X_j'}$ is $\norm{\cdot}_{X_j}$-closed.
			Let also  $m\in [\ell]$ and   $\nnorm{\cdot}_{(d)}$, $d\in \N$,  be an increasing sequence of  seminorms on $X_m'$, such that  for every $d\in \N$, the norm $\norm{\cdot}_{X_m}$
			controls the convergence of the seminorm $\nnorm{\cdot}_{(d)}$ on $B_{X_m'}$.  Then the following two statements are equivalent:
			\begin{enumerate}
				\item (Qualitative property) Whenever  $x_m\in X_m'$ satisfies $\nnorm{x_m}_{(d)} =0$ for every $d\in \N$, we have   for all  $x_j\in X_j'$, $j\in [\ell]$, $j\neq m$,   that
				$$
				A(x_1,\ldots, ,x_\ell)=0.
				$$

				\item (Quantitative property) For every $\varepsilon>0$ there exists $\delta>0$ such that if some    $x_m\in B_{X_m'}$ satisfies    $\nnorm{x_m}_{(d)}\leq \delta $ where $d:=[\delta^{-1}]$, then
				$$
				\sup_{x_j\in B_{X_j}, j\in [\ell], j\neq m} |A(x_1,\ldots,x_\ell)|\leq \varepsilon.
				$$
			\end{enumerate}
		\end{proposition}
		
		For a given system $(X, \CX, \mu,T)$,  with $\CK_{d}(T)$ we denote the factor $\CI(T^{d})$ of $T^{d}$-invariant functions. Since $\CK_{d!}(T)\subseteq \CK_{(d+1)!}(T)$ and $\CK_{1}(T), \ldots, \CK_{d}(T)\subseteq \CK_{d!}(T)$, the sequence of factors $(\CK_{d!}(T))_{d\in\N}$ is increasing, and its inverse limit is the Kronecker factor $\Krat(T)$.	In our  second application, we get the next result  dealing with averages that are controlled by the seminorm $\nnorm{f_m}_{\CK_{rat}(T)}:=\norm{\E(f_m|\CK_{rat}(T))}_{L^2(\mu)}$.
		\begin{proposition}\label{P:Krat}
			Let $(X, \CX, \mu, T_1,\ldots, T_\ell)$ be a system and $\CY_1, \ldots, \CY_\ell\subseteq\CX$ be sub-$\sigma$-algebras.	Suppose that the sequences   $a_1,\ldots, a_\ell\colon \N\to \Z$  are good for mean convergence for the subspaces $L^\infty(\CY_1, \mu), \ldots, L^\infty(\CY_\ell, \mu)$.
			Suppose moreover that for some $m\in [\ell]$, the seminorm $\nnorm{f_m}_{\Krat(T_m)}$  controls   the averages \eqref{E:averages}.
			Then for every
			$\varepsilon>0$
			there exists $\delta>0$
			such that if the functions
			$f_j\in L^\infty(\CY_j, \mu)$, $j\in [\ell]$,  are $1$-bounded and  $\nnorm{f_m}_{\CK_{d!}(T_m)}\leq \delta$  where $d:= [\delta^{-1}]$, then
			$$
			\lim_{N\to\infty} \norm{\E_{n\in[N]} \,  T_1^{a_1(n)}f_1\cdots T_\ell^{a_\ell(n)}f_\ell}_{L^2(\mu)}\leq \varepsilon.
			$$
		\end{proposition}

		We just remark that in order to deduce
		Proposition~\ref{P:Krat}  from Proposition~\ref{P:abstract-general} we use the increasing sequence of seminorms defined by 	 $\nnorm{f}_{(d)}:=\norm{\E(f|\CK_{d!}(T_m))}_{L^2(\mu)}$, $d\in \N$, which satisfy $ \nnorm{f}_{(d)}\leq \norm{f}_{L^2(\mu)}\leq \norm{f}_{L^{\ell+1}(\mu)}$ for every $f\in L^{\ell+1}(\mu)$. Note  also that $\nnorm{f}_{\CK_{rat}(T_m)}=0$ if and
		only if $\nnorm{f}_{(d)}=0$ for every $d\in \N$.
		
		\section{PET bounds}\label{A:PET}
		For applications in the seminorm smoothing argument, we need the following consequence of Proposition \ref{PET bound}, in which we are essentially able to ignore the terms corresponding to dual functions. Recall that $\FD_d$ was defined in \eqref{E:Ds}.
		\begin{proposition}\label{strong PET bound}
			Let $d, \ell, L\in\N$, $\eta\in[\ell]^\ell$, and $p_1, \ldots, p_\ell, q_1, \ldots, q_L\in\Z[n]$ be polynomials of the form
			$
			p_j(n) = \sum_{i=0}^d a_{ji} \, n^i, \,  j\in [\ell].
			$
			 Let also $p_0:=0$,
			 $$
			 d_{\ell j} := \deg(p_\ell \be_{\eta_\ell} - p_j \be_{\eta_j}),\quad j=0,\ldots, \ell-1,
			 $$
			 and  suppose that $\deg p_\ell = d$. Then there exists $s\in\N$ (depending only on $d,\ell,L$) and vectors
			\begin{align}\label{E:coefficientvectors}
				\b_1, \ldots, \b_{s+1} \in \{a_{\ell d_{\ell j}}\be_{\eta_\ell}- a_{j d_{\ell j}}\be_{\eta_j}:\; j = 0, \ldots, \ell-1\},
			\end{align}
			with the following property: for every system $(X, \CX, \mu, T_1, \ldots, T_\ell)$, 1-bounded functions $f_1, \ldots, f_\ell\in L^\infty(\mu)$ and sequences of functions $\CD_1, \ldots, \CD_L\in\FD_d$, we have
			\begin{align}\label{vanishing average strong PET bound}
				\lim_{N\to\infty}\norm{\E_{n\in[N]}\prod_{j\in[\ell]}T_{\eta_j}^{p_j(n)}f_j\cdot\prod_{j\in[L]}\CD_{j}(q_j(n))}_{L^2(\mu)}=0
			\end{align}
			whenever $\nnorm{f_\ell}_{{\b_1}, \ldots, {\b_{s+1}}}=0$.
		\end{proposition}
		We note in particular that if $d_{\ell j}  >0$ for every $j=0, \ldots, \ell-1$, i.e. $p_\ell\be_{\eta_\ell}$ is nonconstant and essentially distinct from $p_j \be_{\eta_j}$ for $j\in[\ell-1]$, then the vectors in \eqref{E:coefficientvectors} are nonzero.

		The proof of Proposition \ref{strong PET bound} combines Proposition \ref{PET bound} with the following result which allows us to get rid of sequences coming from dual functions evaluated at polynomial times.
				\begin{proposition}\label{removing duals}
			Let $d, \ell, L\in\N$ and $q_1, \ldots, q_L\in\Z[n]$ be polynomials of degree at most $d$. There exists $s\in\N$ (depending only on $d,\ell,L$) and $C>0$ such that for any system $(X, \CX, \mu, T_1, \ldots, T_\ell)$ and sequences of 1-bounded functions $A, \CD_1, \ldots, \CD_L$ satisfying $\CD_1, \ldots, \CD_L\in\FD_d$, we have
			\begin{align*}
				&\limsup_{N\to\infty}\norm{\E_{n\in[N]} A(n)\cdot\prod_{j\in [L]} \CD_{j}(q_i(n))}_{L^2(\mu)}^{2^s}\\
				&\leq C \limsup_{H_1\to\infty}\E_{h_1\in[H_1]}\cdots \limsup_{H_s\to\infty}\E_{h_s\in[H_s]}\limsup_{N\to\infty}\norm{\E_{n\in[N]}\prod_{\ueps\in\{0,1\}^s}\CC^{|\ueps|}A(n+\ueps\cdot\uh)}_{L^2(\mu)}.
			\end{align*}
		\end{proposition}
		Proposition \ref{removing duals} is a special case of \cite[Proposition 6.1]{Fr15a}. The latter result is more general and covers e.g. sequences coming from Hardy fields, which is the reason why  its statement also includes extra terms such as $\mathbf{1}_E(n)$ that do not appear in our case because we only deal with polynomials.

		\begin{proof}[Proof of Proposition \ref{strong PET bound}]
	For the purposes of this proof we call a subset of $\N^k$ small if it is contained in a finite union of hyperplanes of $\N^k$.  Applying Proposition \ref{removing duals} with $A(n) := \prod_{j\in[\ell]}T_{\eta_j}^{p_j(n)}f_j$, $n\in\N$,   we deduce that (the next limit exists but we do not need this fact)
				\begin{align}\label{E: PET 1}
				\limsup_{N\to\infty}\norm{\E_{n\in[N]}\prod_{j\in[\ell]}T_{\eta_j}^{p_j(n)}f_j\cdot\prod_{j\in[L]}\CD_{j}(q_j(n))}_{L^2(\mu)}
			\end{align}
		is bounded from above by
			\begin{align}\label{E:limsups}
				\limsup_{H_1\to\infty}\E_{h_1\in[H_1]}\cdots \limsup_{H_{s'}\to\infty}\E_{h_{s'}\in[H_{s'}]} \limsup_{N\to\infty}\norm{\E_{n\in[N]}\prod_{j\in[\ell]}\prod_{\ueps\in\{0,1\}^{s'}}\CC^{|\ueps|} T_{\eta_j}^{p_j(n+\ueps\cdot\uh)}f_j}_{L^2(\mu)}
			\end{align}
			(where $\uh=(h_1,\ldots, h_{s'}))$ for some $s'\in\N$ (depending only on $d,\ell, L$), functions $f_1, \ldots, f_\ell$, and sequences $\CD_1, \ldots, \CD_L$.
			
			If $d=1$ and $p_j(n) := a_{j1} n + a_{j0}$, $j\in [\ell]$, then \eqref{E:limsups} equals
			\begin{multline}\label{E: PET2}
				\limsup_{H_1\to\infty}\E_{h_1\in[H_1]}\cdots \limsup_{H_{s'}\to\infty}\E_{h_{s'}\in[H_{s'}]}\\
				\limsup_{N\to\infty}\norm{\E_{n\in[N]}\prod_{j\in[\ell]}T_{\eta_j}^{a_{j1} n+a_{j0}}\Big( \prod_{\ueps\in\{0,1\}^{s'}}\CC^{|\ueps|} T_{\eta_j}^{a_{j1}\ueps\cdot\uh} f_j\Big)}_{L^2(\mu)}
			\end{multline}
		where $\uh=(h_1,\ldots, h_{s'})$.
			Known results for linear averages (see for example \cite[Proposition~1]{Ho09}) imply that \eqref{E: PET2} is bounded from above by
			\begin{multline*}
				\limsup_{H_1\to\infty}\E_{h_1\in[H_1]}\cdots \limsup_{H_{s'}\to\infty}\E_{h_{s'}\in[H_{s'}]}\\
				\Bignnorm{\prod_{\ueps\in\{0,1\}^{s'}}\CC^{|\ueps|} T_{\eta_\ell}^{a_{\ell 1}\ueps\cdot\uh} f_\ell}_{a_{\ell 1}\be_{\eta_\ell}, a_{\ell 1}\be_{\eta_\ell} - a_{1 1}\be_{\eta_1}, \ldots, a_{\ell 1}\be_{\eta_\ell}- a_{(\ell-1)1}\be_{\eta_{\ell-1}}},
			\end{multline*}
			which by an iteration of the inductive formula \eqref{ergodic identity} equals
			\begin{align*}
				\nnorm{f_\ell}_{(a_{\ell 1}\be_{\eta_\ell})^{\times (s' + 1)}, a_{\ell 1}\be_{\eta_\ell} - a_{1 1}\be_{\eta_1}, \ldots, a_{\ell 1}\be_{\eta_\ell}- a_{(\ell-1)1}\be_{\eta_{\ell-1}}},
			\end{align*}
			implying the result in the case $d=1$.
			
			Suppose now $d>1$.
			We first show that for all $\uh\in \N^{s'}$ except a small set, the leading coefficients of the polynomials
			\begin{align}\label{E: difference polynomials}
			\{p_\ell(n+\underline{1}\cdot\uh)\be_{\eta_\ell} - p_j(n+\ueps\cdot\uh)\be_{\eta_j}\colon\; \ueps\in\{0,1\}^{s'},\ j\in\{0, \ldots, \ell\},\ (j, \ueps) \neq (\ell, \underline{1})\}
			\end{align}
			are nonzero rational multiples of the leading coefficients of the polynomials $$p_\ell \be_{\eta_\ell}, p_\ell \be_{\eta_\ell} - p_1 \be_{\eta_1}, \ldots, p_\ell \be_{\eta_\ell} - p_{\ell-1} \be_{\eta_{\ell-1}}.$$ Indeed, to analyse what the leading coefficients of the polynomials in \eqref{E: difference polynomials} are, we need to consider four cases:
			\begin{enumerate}
			    \item $\be_{\eta_\ell} \neq \be_{\eta_j}$;
			    \item $\be_{\eta_\ell} = \be_{\eta_j}$ and $a_{\ell d} \neq a_{\ell j}$;
			    \item $\be_{\eta_\ell} = \be_{\eta_j}$, $a_{\ell d} = a_{\ell j}$ and $\ueps = \underline{1}$;
			    \item $\be_{\eta_\ell} = \be_{\eta_j}$, $a_{\ell d} = a_{\ell j}$ and $\ueps \neq \underline{1}$.
			\end{enumerate}
			In the first three cases, the leading coefficient is $a_{\ell d_{\ell j}}\be_{\eta_\ell} - a_{j d_{\ell j}}\be_{\eta_j}$, where $d_{\ell j} = d$ in the first two cases and $0< d_{\ell j} < d$ in the third one. Thus, the leading coefficient of $p_\ell(n+\underline{1}\cdot\uh)\be_{\eta_\ell} - p_j(n+\ueps\cdot\uh)\be_{\eta_j}$ in these cases is the same as the leading coefficient of $p_\ell \be_{\eta_\ell}- p_j \be_{\eta_j}$.
			In the fourth case, the leading coefficient is
			$$
			(a_{\ell d}d(\underline{1}-\ueps)\cdot \uh + a_{\ell (d-1)} - a_{j(d-1)})\be_{\eta_\ell},
			$$
			 and it is nonzero for all $\uh\in\N^{s'}$ except a small subset $A_{j,\ueps}$ of $\N^{s'}$. Hence, for all $\uh\in\N^{s'}\setminus{A_{j,\ueps}}$, the leading coefficient in this case takes the form $c_{j,\ueps}(\uh)a_{\ell d}\be_{\eta_\ell}$ for some nonzero $c_{j, \ueps}(\uh)\in\Q$ such that $c_{j, \ueps}(\uh)a_{\ell d}\in\Z$.
			
			By Proposition \ref{PET bound}, there exist $s\in\N$ (depending only on $d,\ell, L$) and vectors
			\begin{align}\label{E:coefficient vectors strong PET}
			    \b'_1(\uh), \ldots, \b'_{s+1}(\uh)\in &\{a_{\ell d_{\ell j}}\be_{\eta_\ell} - a_{j d_{\ell j}}\be_{\eta_j}\colon \ j\in \{0, \ldots, \ell-1\}\}\\
			  \nonumber
			    &\cup\{c_{j, \ueps}a_{\ell d}\be_{\eta_\ell}\colon \ j\in\{0, \ldots, \ell\},\ \ueps\neq \underline{1}\}
			\end{align}
			such that for every $\uh\in\N^{s'}$, we have
            \begin{align}\label{E:vanishing}
                \lim_{N\to\infty}\norm{\E_{n\in[N]}\prod_{j\in[\ell]}\prod_{\ueps\in\{0,1\}^{s'}}\CC^{|\ueps|} T_{\eta_j}^{p_j(n+\ueps\cdot\uh)}f_j}_{L^2(\mu)} = 0
            \end{align}			
            whenever
            \begin{align*}
                \nnorm{f_\ell}_{\b'_1(\uh), \ldots, \b'_{s+1}(\uh)}=0.
            \end{align*}
Moreover, Lemma \ref{L:seminorm of power} and \eqref{E:coefficient vectors strong PET} imply the bound
            \begin{align*}
                \nnorm{f_\ell}_{\b'_1(\uh), \ldots, \b'_{s+1}(\uh)} \ll_{\{c_{j, \ueps}:\ (j, \ueps)\neq (1,\ell)\}} \nnorm{f_\ell}_{\b_1, \ldots, \b_{s+1}}
            \end{align*}
            for some vectors $\b_1, \ldots, \b_{s+1}$ from \eqref{E:coefficientvectors}. It follows that for all $\uh\in\N^{s'}\setminus{A}$, where $A = \bigcup_{(j, \ueps)\neq (\ell, \underline{1})}A_{j, \ueps}$ is small, the equality $\nnorm{f_\ell}_{\b_1, \ldots, \b_{s+1}} = 0$ implies  \eqref{E:vanishing}. The result follows from the fact that the identity \eqref{E:vanishing} for all but a small subset of $\uh\in \N^{s'}$ implies in turn using the bound \eqref{E:limsups} that the limit in \eqref{E: PET 1} is 0.

		\end{proof}

\end{document}